\documentclass[preprint,12pt]{elsarticle}

\usepackage{times}
\usepackage{amssymb}
\usepackage{amsmath}
\usepackage{latexsym}
\usepackage{stmaryrd}
\usepackage{prooftree}
\usepackage{amsthm}
\usepackage{txfonts}
\usepackage{hyperref}
\usepackage{tikz-cd}

\hypersetup{
    colorlinks=true,
    linkcolor=blue,
    filecolor=magenta,      
    urlcolor=cyan,
}

\newtheorem{definition}{Definition}[section]
\newtheorem{proposition}[definition]{Proposition}
\newtheorem{lemma}[definition]{Lemma}
\newtheorem{theorem}[definition]{Theorem}
\newtheorem{remark}[definition]{Remark}

\newtheorem{notation}[definition]{Notation}

\def\NN{\mathbb{N}}
\def\BB{\mathbb{B}}

\def\S{{\textup S}}
\def\K{{\textup K}}

\newcommand{\Pred}[1]{\mathbf{Pred}_{#1}}
\newcommand{\Formulas}[1]{\mathbf{Form}_{#1}}
\newcommand{\Ax}[1]{\mathbf{Ax}_{#1}}

\newcommand{\Assumption}[1]{({\bf A#1})}

\newcommand{\efq}{\text{efq}}
\newcommand{\ruleid}{\text{id}}

\newcommand{\rulecut}{\text{cut}}
\newcommand{\ruleper}{\text{per}}
\newcommand{\rulecon}{\text{con}}
\newcommand{\rulewkn}{\text{wkn}}
\newcommand{\true}{\textsc{T}}
\newcommand{\false}{\textsc{F}}

\newcommand{\IH}[1]{\textup{(}{\rm IH}$_{#1}$\textup{)}}

\newcommand{\wtype}[1]{{\rm wt}(#1)}
\newcommand{\btype}[1]{{\rm bt}(#1)}

\newcommand{\ifW}[1]{{\rm m}^{\BB}_{#1}}
\newcommand{\apW}[1]{{\rm m}^{\NN}_{#1}}

\newcommand{\ConUnit}{\textup{(}{\bf C}$_\eta$\textup{)}}
\newcommand{\ConStrength}{\textup{(}{\bf C}$_\sqcup$\textup{)}}
\newcommand{\ConApp}{\textup{(}{\bf C}$_\circ$\textup{)}}

\newcommand{\Quant}[1]{\textup{(}{\bf Q}$_{#1}$\textup{)}}

\newcommand{\WitS}{\textup{(}{\bf W}$_\S$\textup{)}}
\newcommand{\WitK}{\textup{(}{\bf W}$_\K$\textup{)}}
\newcommand{\WitApp}{\textup{(}{\bf W}$_{\rm Ap}$\textup{)}}

\newcommand{\ubq}[3]{\forall #2 \!\sqsubset_{#1}\! #3 \,}
\newcommand{\lubq}[4]{\forall #2 \!\stackrel{\bullet}{\sqsubset}{\!}_{#1}\, #3 \, #4}
\newcommand{\bubq}[4]{\forall #2 \!\stackrel{\circ}{\sqsubset}{\!}_{#1}\, #3 \, #4}
\newcommand{\bound}[3]{#1 \!\prec^{#3}\! #2}
\newcommand{\lbound}[3]{#1 \prec^{#3}_\bullet #2}
\newcommand{\bbound}[3]{#1 \prec^{#3}_\circ #2}
\newcommand{\singleton}[1]{\eta #1}
\newcommand{\join}[2]{#1 \sqcup #2}
\newcommand{\comp}[2]{#1 \circ #2}
\newcommand{\Wit}{\textup{W}}
\newcommand{\lWit}{\textup{W}^\bullet}

\newcommand{\eqleft}[1]{\begin{enumerate} \item[] $#1$ \end{enumerate}}

\newcommand{\mr}{\;{\sf mr}\;}
\newcommand{\bmr}{\;{\sf bmr}\;}

\newcommand{\dnInter}[3]{#1_{\mathrm{DN}}(#2;#3)}
\newcommand{\lTrans}[1]{#1^{\bullet}}
\newcommand{\fTrans}[1]{#1^\mathcal{F}}
\newcommand{\bTrans}[1]{#1^{\circ}}

\newcommand{\eTrans}[1]{#1^{\star}}

\newcommand{\AHAomega}{{\bf WE\textup{-}AHA}^{\omega}}
\newcommand{\NAHA}{{\bf N\textup{-}AHA}}
\newcommand{\AL}{{\bf AL}}
\newcommand{\IL}{{\bf IL}}

\newcommand{\HAomega}{{\bf WE\textup{-}HA}^\omega}
\newcommand{\nHAomega}{{\bf N\textup{-}HA}^\omega}
\newcommand{\eHAomegaS}{{\bf E\textup{-}HA}^{\omega*}}
\newcommand{\eAHAomegaS}{{\bf E\textup{-}AHA}^{\omega*}}
\newcommand{\suc}{{\rm Suc}}
\newcommand{\rec}{{\rm Rec}}
\newcommand{\ap}{{\rm Ap}}
\newcommand{\Type}{{\cal T}}

\newcommand{\FV}{{\rm FV}}

\newcommand{\Atarget}{{\mathcal A}_{\mathbf t}}
\newcommand{\Asource}{{\mathcal A}_{\mathbf s}}

\newcommand{\Itarget}{{\mathcal I}_{\mathbf t}}
\newcommand{\Isource}{{\mathcal I}_{\mathbf s}}

\newcommand{\pdefin}{:\equiv}

\newcommand{\uInter}[3]{|#1|^{#2}_{#3}}
\newcommand{\lInter}[3]{\{\!\{#1\}\!\}^{#2}_{#3}}
\newcommand{\bInter}[3]{(\!(#1)\!)^{#2}_{#3}}

\newcommand{\pvec}[1]{\boldsymbol{#1}}
\newcommand{\proves}{\vdash}

\newcommand{\cwedge}{\otimes}
\newcommand{\awedge}{\,\&\,}
\newcommand{\avee}{\,\oplus\,}
\newcommand{\bang}[1]{! #1}
\newcommand{\whynot}[1]{? #1}
\newcommand{\lto}{\multimap}

\newcommand{\pcond}[3]{(#1 = \true \to #2) \wedge (#1 = \false \to #3)}
\newcommand{\lpcond}[3]{(\bang (#1 = \true) \lto #2) \cwedge (\bang(#1 = \false) \lto #3)}

\newcommand{\st}{{\rm st}}

\newcommand*{\mforall}{\tilde\forall}

\journal{Annals of Pure and Applied Logic}

\begin{document} 

\begin{frontmatter}


\title{Parametrised Functional Interpretations}


\author[1]{Bruno Dinis
\fnref{fn1}}
\ead{bmdinis@fc.ul.pt}

\author[2]{Paulo Oliva}
\ead{p.oliva@qmul.ac.uk}

\address[1]{Departamento de Matem\'{a}tica, Faculdade de
Ci\^{e}ncias da Universidade de Lisboa,
Campo Grande, Ed. C6, 1749-016, Lisboa, Portugal}

\address[2]{School of Electronic Engineering and Computer Science, Queen Mary University of London, London E1 4NS, United Kingdom}

\fntext[fn1]{The author acknowledges the support of FCT - Funda\c{c}\~ao para a Ci\^{e}ncia e Tecnologia under the projects: UID/MAT/04561/2019 and PTDC/MAT-PUR/3971/2020 MATHLOGIC, and the research center Centro de Matem\'{a}tica, Aplica\c{c}\~{o}es Fundamentais e Investiga\c{c}\~{a}o Operacional, Universidade de Lisboa.}

\begin{abstract}
This paper presents a general framework for unifying functional interpretations. It is based on families of parameters allowing for different degrees of freedom on the design of the interpretation. In this way we are able to generalise previous work on unifying functional interpretations, by including in the unification the more recent bounded and Herbrandized functional interpretations.
\end{abstract} 

\begin{keyword}
functional interpretations \sep unification \sep intuitionism \sep proof theory 

\MSC 03F07 \sep 03F10 \sep 03F30 \sep 03F55

\end{keyword}

\end{frontmatter}

\section{Introduction}

Since G\"odel \cite{G(58)} published his functional (``Dialectica") interpretation in 1958, various other functional interpretations have been proposed\footnote{See \cite{AF(98)} for a survey on the ``Dialectica" interpretation.}. These include Kreisel's modified realizability \cite{K(59)}, the Diller-Nahm variant of the Dialectica interpretation \cite{DN(74)}, Stein's family of interpretations \cite{Stein(79)}, and more recently, the bounded functional interpretation \cite{FerreiraOliva2005}, the bounded modified realizability \cite{FerreiraNunes2006}, and ``Herbrandized" versions of modified realizability and the Dialectica \cite{BergBriseidSafarik2012}. In view of this picture, several natural questions arise: How are these different interpretations related to each other? What is the common structure behind all of them? Are there any other interpretations out there waiting to be discovered? 

These questions were addressed by the second author (and various co-authors) in a series of papers on \emph{unifying functional interpretations}. Starting with a unification of interpretations of intuitionistic logic \cite{Oliva2006}, which was followed by various analysis of functional interpretations within the finer setting of linear logic \cite{FO(11),Oliva2007,Oliva2008,Oliva2010}, a proposal on how functional interpretations could actually be combined in so-called \emph{hybrid functional interpretations} \cite{HO(08),Oliva2012}, and the inclusion of truth variants in the unification \cite{GO(10)}. 

Functional interpretations associate with each formula $A$ a new formula $\uInter{A}{\pvec x}{\pvec y}$ where $\pvec x$ and $\pvec y$ are fresh tuples of variables. Intuitively, $\pvec x$ captures the ``positive" quantifications in $A$, while $\pvec y$ captures the ``negative" quantifications. This is done in such a way that, in a suitable system, the truth of $A$ is equivalent to that of $\exists \pvec x \forall \pvec y \uInter{A}{\pvec x}{\pvec y}$. The key insight which arises from the programme of ``unifying functional interpretations" is that we have some degree of freedom when choosing the interpretation of the exponentials of linear logic $\bang A$ and $\whynot A$. For instance, we can take
\[
\begin{array}{lcl}
	\uInter{\bang A}{\pvec x}{\pvec y} \pdefin \; \bang \uInter{A}{\pvec x}{\pvec y} 
		& \hspace{5mm} & \mbox{(giving rise to the Dialectica interpretation)} \\[1mm]
	\uInter{\bang A}{\pvec x}{\pvec a} \pdefin \; \bang \forall \pvec y \in \pvec a \uInter{A}{\pvec x}{\pvec y} 
		& \hspace{5mm} & \mbox{(giving rise to the Diller-Nahm interpretation)} \\[1mm]
	\uInter{\bang A}{\pvec x}{} \pdefin \; \bang \forall \pvec y \uInter{A}{\pvec x}{\pvec y} 
		& \hspace{5mm} & \mbox{(giving rise to modified realizability)} \\[1mm]
	\uInter{\bang A}{\pvec x}{} \pdefin \; \bang \forall \pvec y \uInter{A}{\pvec x}{\pvec y} \; \otimes \; \bang A 
		& \hspace{5mm} & \mbox{(giving rise to modified realizability with truth)} \\[1mm]
	\mbox{and so on...}
		&  & 
\end{array}
\]
showing that each of these interpretations only differ in the way they treat the \emph{contraction axiom}. In particular, in the pure fragment of linear logic all these interpretations coincide!

So, it makes sense to introduce an abstract bounded quantification $\ubq{\tau}{x}{a}A$, capturing this degree of freedom on the design of a functional interpretation, and to try to isolate the properties of this parameter which ensure the soundness of the interpretation. With this one is able to define a ``unifying functional interpretation" which when instantiated gave rise to several of the existing functional interpretations, including the Dialectica interpretation, modified realizability (its q- and truth variants), Stein's family of interpretations, and the Diller-Nahm interpretation \cite{Oliva2006,Oliva2010}. This process led to the design of a ``Diller-Nahm with truth" interpretation \cite{GO(10)}, which at the time was not thought to be possible.

But the \emph{unifying functional interpretation} programme has so far been unable to capture the two more recent families of functional interpretations, namely the bounded functional interpretations \cite{DG(18),FerreiraGaspar2015,FerreiraNunes2006,FerreiraOliva2005}, and the Herbrandized functional interpretations \cite{BergBriseidSafarik2012,Ferreira2017}.

In this paper we propose a framework for a more general unification, introducing other families of parameters which allow for different interpretations of \emph{typed quantifications}. We demonstrate that, when devising a functional interpretation, we in fact have \emph{two crucial degrees of freedom}: we can choose how to interpret the contraction axiom, as discussed above, but also, we can choose how to interpret typed quantifications, which ultimately boils down to the choice of how predicate symbols are interpreted. 

\begin{table}[t] 
\[
\begin{array}{|rccc|}
\hline
& & & \\
\hspace{15mm} & 
\begin{prooftree} 
    \justifies A \proves A 
    \using (\ruleid) 
\end{prooftree}
& \hspace{5mm} &
\begin{prooftree}
    \justifies \Gamma, \bot \proves A  
    \using (\efq) 
\end{prooftree} 
\quad \quad \\[5mm]
& 
\begin{prooftree}
    \Gamma \proves A \quad \Delta, A \proves B
    \justifies \Gamma, \Delta \proves B 
    \using (\rulecut) 
\end{prooftree}
& &
\begin{prooftree} 
    \Gamma, A, B, \Delta \proves C 
    \justifies 
    \Gamma, B, A, \Delta \proves C
    \using (\ruleper) 
\end{prooftree} \\[5mm]
\hline
& & & \\
& \begin{prooftree} \Gamma \proves A \quad \Delta \proves B \justifies \Gamma,\Delta \proves A \cwedge B \using (\cwedge\textup{R}) \end{prooftree}
& &
\begin{prooftree} \Gamma, A, B \proves C \justifies \Gamma, A \cwedge B \proves C \using (\cwedge\textup{L}) \end{prooftree} \\[5mm]
& \begin{prooftree} \Gamma, A \proves B \justifies \Gamma \proves A \lto B \using (\lto\!\textup{R}) \end{prooftree}
& &
\begin{prooftree} \Gamma \proves A \quad \Delta, B \proves C \justifies \Gamma, \Delta, A \lto B \proves C \using (\lto\!\textup{L}) \end{prooftree} \\[5mm]
\hline
\hspace{15mm} & & & \\
& \begin{prooftree} 
    \Gamma \proves A 
    \justifies \Gamma \proves \forall x A 
    \using (\forall \textup{R}, x \not \in \FV(\Gamma)) 
\end{prooftree}
& &
\begin{prooftree}
    \Gamma, A[t/x] \proves B
    \justifies
    \Gamma, \forall x A \proves B
    \using (\forall \textup{L}) 
\end{prooftree} \\[5mm]
& 
\begin{prooftree}
    \Gamma \proves A[t/x]
    \justifies
    \Gamma \proves \exists x A
    \using (\exists \textup{R})
\end{prooftree}
& &
\begin{prooftree} 
    \Gamma, A \proves B
    \justifies
    \Gamma, \exists x A \proves B
    \using (\exists \textup{L}, x \not \in \FV(\Gamma, B))
\end{prooftree} \\[5mm]
\hline
& & & \\
\multicolumn{4}{|l|}{
  \quad
  \begin{prooftree} \Gamma, \bang A, \bang A \proves B \justifies \Gamma, \bang A \proves B \using (\rulecon) \end{prooftree}
  \quad \quad\quad
  \begin{prooftree} \Gamma \proves B \justifies \Gamma, A \proves B \using (\rulewkn) \end{prooftree}
  \quad \quad
  \begin{prooftree} \bang \Gamma \proves A \justifies \bang \Gamma \proves \bang A \using (\bang\textup{R}) \end{prooftree}
  \quad \quad
  \begin{prooftree} \Gamma, A \proves B \justifies  \Gamma, \bang A \proves B \using (\bang\textup{L}) \quad \end{prooftree}
} \\
& & & \\
\hline
\end{array}
\]
\caption{Sequent Calculus for Intuitionistic Affine Logic $\AL$} \label{ill-rules}
\end{table}

We will start by presenting (Section \ref{sec-al-inter}) this parametrised interpretation in the setting of affine logic ($\AL$). Then, via the two well-known Girard translations from intuitionistic logic ($\IL$) into $\AL$ \cite{Girard(87B)}, we will also obtain two parametrised interpretations of $\IL$ (Section \ref{IL-interpretations}). We conclude (Section \ref{sec-instances}) by showing how all of the functional interpretations mentioned above can be obtained by suitable choices of these parameters. In this process we have again discovered some new interpretations (see Section \ref{sec-instances}). 

\subsection{Intuitionistic affine logic and theories}
\label{sec-affine-theory}

A sequent calculus for \emph{intuitionistic affine logic} $\AL$ is shown in Table \ref{ill-rules}, with negation $A^\bot$ defined as $A \lto \bot$. An extension of $\AL$ with new predicate and function symbols, and non-logical axioms, will be called an \emph{intuitionistic affine theory}, or $\AL$-theory, for short. Given an intuitionistic affine theory $\mathcal{A}$ we will denote its set of predicate symbols by $\Pred{\mathcal{A}}$, its set of formulas by $\Formulas{\mathcal{A}}$, and its set of non-logical axioms by $\Ax{\mathcal{A}}$. Subsection \ref{Subsection_theories} defines the five $\AL$-theories that we will use in this paper.

\begin{notation} If $\mathcal{A}$ is an $\AL$-theory then we write $\Gamma \proves_{\mathcal{A}} A$ as an abbreviation for ``$\mathcal{A}$ proves the sequent $\Gamma \proves A$". We write $A \Leftrightarrow_{\mathcal{A}} B$ when we have both $A \proves_{\mathcal{A}} B$ and $B \proves_{\mathcal{A}} A$. When the theory $\mathcal{A}$ used is clear from the context, we omit the subscript. We use boldface letters $\pvec x, \pvec y, \ldots$ for tuples of variables or terms, and write $\varepsilon$ for the empty tuple. Given a formula $B(\pvec x)$ of an $\AL$-theory, we will make use of the following abbreviations $\forall {\pvec x}^B A \pdefin \forall \pvec x (\bang B(\pvec x) \lto A)$ and $\exists {\pvec x}^B A \pdefin \exists \pvec x (B(\pvec x) \otimes A)$.
\end{notation}

\subsection{Intuitionistic logic and theories}

An intuitionistic theory, or $\IL$-theory, is an extension of first-order intuitionistic logic $\IL$ with constant symbols, predicate symbols, and non-logical axioms. Given an intuitionistic theory $\mathcal{I}$ we will denote its set of predicate symbols by $\Pred{\mathcal{I}}$, its set of formulas by $\Formulas{\mathcal{I}}$, and its set of non-logical axioms by $\Ax{\mathcal{I}}$. Subsection \ref{Subsection_theories} also defines the five $\IL$-theories that we will use in this paper.

\begin{notation} \label{notation:qual-IL} Given a formula $B(\pvec x)$ of an $\IL$-theory, we will make use of the following abbreviations $\forall {\pvec x}^B A \pdefin \forall \pvec x (B(\pvec x) \to A)$ and $\exists {\pvec x}^B A \pdefin \exists \pvec x (B(\pvec x) \wedge A)$.
\end{notation}

\begin{definition}[Girard translations, \cite{Girard(87B)}] \label{g-trans} Define two translations\footnote{The $\lTrans{(\cdot)}$ translation of $A \wedge B$ in linear logic is normally $\lTrans{A} \awedge \lTrans{B}$, but, in the presence of weakening, one can also define it with multiplicative conjunction $\lTrans{A} \cwedge \lTrans{B}$. We prefer this latter version as it leads to simpler functional interpretations.} of an $\IL$-theory into a corresponding $\AL$-theory: (where $P$ ranges over predicate symbols)
\[
\begin{array}{llll}
\lTrans{(P(\pvec t))} & \pdefin P(\pvec t)
 & \bTrans{(P(\pvec t))} & \pdefin \;\bang P(\pvec t) \\[2mm]
\lTrans{(A \wedge B)} & \pdefin \lTrans{A} \cwedge \lTrans{B}
 &  \bTrans{(A \wedge B)}     &\pdefin \bTrans{A} \cwedge \bTrans{B} \\[2mm]
\lTrans{(A \to B)}     &\pdefin \;\bang \lTrans{A} \lto \lTrans{B} \hspace{15mm}
 & \bTrans{(A \to B)}     &\pdefin \;\bang (\bTrans{A} \lto \bTrans{B}) \\[2mm]
\lTrans{(\forall x A)} &\pdefin \forall x \lTrans{A}
 & \bTrans{(\forall x A)} &\pdefin \;\bang \forall x \bTrans{A} \\[2mm]
\lTrans{(\exists x A)} &\pdefin \exists x \bang \lTrans{A}
 &  \bTrans{(\exists x A)} &\pdefin \exists x \bTrans{A}
\end{array}
\]
Given an $\IL$-theory $\mathcal{I}$, let $\lTrans{\mathcal{I}}$ denote the $\AL$-theory with the same constants and predicate symbols as $\mathcal{I}$, and non-logical axioms $\Ax{\lTrans{\mathcal{I}}} = \{ \bang \lTrans{\Gamma} \proves \lTrans{A} \; \colon \; \Gamma \proves A \in \Ax{\mathcal{I}} \}$. Similarly, let $\bTrans{\mathcal{I}}$ denote the $\AL$-theory with the same constants and predicate symbols as $\mathcal{I}$, and non-logical axioms $\Ax{\bTrans{\mathcal{I}}} = \{ \bTrans{\Gamma} \proves \bTrans{A} \; \colon \; \Gamma \proves A \in \Ax{\mathcal{I}} \}$.
\end{definition}

\begin{proposition} \label{ltrans-prop} If $\Gamma \proves_{\mathcal{I}} A$ then $\bang \lTrans{\Gamma} \proves_{\lTrans{\mathcal{I}}} \lTrans{A}$ and $\bTrans{\Gamma} \proves_{\bTrans{\mathcal{I}}} \bTrans{A}$.
\end{proposition}

\begin{proof} A simple adaptation of the similar result from \cite{Girard(87B)}.
\end{proof}

\begin{proposition} \label{prop:lb-equivalence} For all $A \in \Formulas{\IL}$ we have that $\bTrans{A} \Leftrightarrow_{\AL} \, \bang \lTrans{A}$, and hence, for any $\IL$-theory $\mathcal{I}$, the $\AL$-theories $\lTrans{\mathcal{I}}$ and $\bTrans{\mathcal{I}}$ prove the same set of formulas.
\end{proposition}
\begin{proof} The first part is shown in \cite{GO(10)}. Using this it is easy to see that all the non-logical axioms of $\lTrans{\mathcal{I}}$ are derivable in $\bTrans{\mathcal{I}}$, and vice-versa.
\end{proof}

\begin{definition}[Forgetful translation] \label{forget} Define the following translation of an $\AL$-theory into a $\IL$-theory: (where $P$ ranges over predicate symbols)
\[
\begin{array}{lcccl}
\fTrans{(P(\pvec x))} & \pdefin P(\pvec x) & \quad &
    \fTrans{(\bang A)} &\pdefin \fTrans{A} \\[2mm]
\fTrans{(A \cwedge B)}  &\pdefin \fTrans{A} \wedge \fTrans{B} & &
    \fTrans{(\forall x A)} &\pdefin \forall x \fTrans{A} \\[2mm]
\fTrans{(A \lto B)}  &\pdefin \fTrans{A} \to \fTrans{B} & &
    \fTrans{(\exists x A)} &\pdefin \exists x \fTrans{A}
\end{array}
\]
Given an $\AL$-theory $\mathcal{A}$, let $\fTrans{\mathcal{A}}$ denote the $\IL$-theory with non-logical axioms $\Ax{\fTrans{\mathcal{A}}} = \{ \fTrans{\Gamma} \proves \fTrans{A} \colon \Gamma \proves A \in \Ax{\mathcal{A}} \}$ and the same constants and predicate symbols as $\mathcal{A}$.
\end{definition}

\subsection{Some concrete $\IL$-theories and $\AL$-theories}\label{Subsection_theories}

By the Girard translations (Definition \ref{g-trans}) $\IL$-theories give rise to $\AL$-theories\footnote{By Proposition \ref{prop:lb-equivalence}, it does not matter which Girard translation we use.}, and by the forgetful translation (Definition \ref{forget}) $\AL$-theories give rise to $\IL$-theories. In this section we will define the following pairs of $(\IL, \AL)$ theories:
\[
\begin{array}{lll}
	\IL\mbox{-theory} & \AL\mbox{-theory} & \mbox{theory of} \\[1mm]
	\hline
	\IL^{\rm eq} & \AL^{\rm eq} & \mbox{equality} \\[1mm] 
	\IL^{\BB} & \AL^{\BB} & \mbox{booleans (extends theory of equality)} \\[1mm] 
	\IL^{\omega} & \AL^{\omega} & \mbox{finite types (extends theory of booleans)} \\[1mm] 
	\HAomega & \AHAomega & \mbox{arithmetic in all finite types (decidable equality)} \\[1mm]
	\eHAomegaS & \eAHAomegaS & \mbox{arithmetic in all finite types (undecidable equality)}  
\end{array}
\]

Let $\IL^{\rm eq}$ denote intuitionistic predicate logic with equality (see \cite[Section 3.1]{K(08)}). Using Girard's translations we can then obtain an $\AL$-theory of equality $\AL^{\rm eq} = \lTrans{(\IL^{\rm eq})}$. 

An $\IL$-theory of booleans, which we will call $\IL^\BB$, can be obtained by extending $\IL^{\rm eq}$ with two constant symbols $\true$ and $\false$, a new predicate symbol $\BB(z)$, for ``$z$ is a boolean", and the following non-logical axioms:
\[
    \proves \BB(\true)
\quad
\quad
    \proves \BB(\false)
\quad
\quad
    \proves \neg (\true = \false)
\quad
\quad
    A[\true/z], A[\false/z], \BB(z) \proves A
\]
We will refer to the first two axioms as $\BB$\emph{R}, and the last axiom as $\BB$\emph{L}. We also assume that in $\IL^{\BB}$ we can define terms by cases, i.e. we have a function symbol ``${\rm if}$'' such that the following are derivable in $\IL^{\BB}$
\[
\proves {\rm if}(\true, x, y) = x \quad \quad \mbox{and} \quad \quad \proves {\rm if}(\false, x, y) = y 
\]
Again, by Girard's translations, we get an $\AL$-theory of booleans $\AL^{\BB} = \lTrans{(\IL^{\BB})}$.

\begin{proposition} \label{prop-def-disjunction} In $\IL^{\BB}$ disjunction is definable as
\[
\begin{array}{rcl}
	A \vee B & \pdefin & \exists z^\BB (((z = \true) \to A) \wedge ((z = \false) \to B))
\end{array}
\]
while in $\AL^{\BB}$ the additive connectives of linear logic are definable as
\[
\begin{array}{rcl}
	A \awedge B & \pdefin & \forall z (\bang \BB(z) \lto ((\bang(z = \true) \lto A) \cwedge (\bang(z = \false) \lto B))) \\[2mm]
	A \avee B & \pdefin & \exists z (\bang \BB(z) \cwedge ((\bang(z = \true) \lto A) \cwedge (\bang(z = \false) \lto B)))
\end{array}
\]
in the sense that their corresponding rules are derivable.
\end{proposition}

Let $A \avee B$ and $A \vee B$ be defined as in Proposition \ref{prop-def-disjunction}. Then the Girard translations of $\IL$ into $\AL$ extend to translations of $\IL^\BB$ into $\AL^\BB$:

\begin{proposition} \label{prop-definable-connectives} The following equivalences are provable in $\AL^{\BB}$
\begin{enumerate}
	\item[$(i)$] $\lTrans{(A \vee B)} \Leftrightarrow \; \bang \lTrans{A} \avee \bang\lTrans{B}$
	\item[$(ii)$] $\bTrans{(A \vee B)} \Leftrightarrow \bTrans{A} \avee \bTrans{B}$
\end{enumerate}
\end{proposition}

\begin{proof} $(i)$ We have
\[
\begin{array}{rcl}
	\lTrans{(A \vee B)} & \stackrel{\textup{P}\ref{prop-def-disjunction}}{\equiv} & \lTrans{(\exists z^\BB (((z = \true) \to A) \wedge ((z = \false) \to B)))} \\[1mm]
	& \stackrel{\textup{D}\ref{g-trans}}{\Leftrightarrow} & \exists z \bang (\BB(z) \cwedge((\bang(z = \true) \lto \lTrans{A}) \cwedge (\bang(z = \false) \lto \lTrans{B}))) \\[2mm]
	&\Leftrightarrow& \exists z (\bang \BB(z) \cwedge ((\bang(z = \true) \lto \bang \lTrans{A}) \cwedge (\bang(z = \false) \lto \bang\lTrans{B})))   \\[1mm]
	& \stackrel{\textup{P}\ref{prop-def-disjunction}}{\Leftrightarrow} & \bang \lTrans{A} \avee \bang\lTrans{B} 
\end{array}
\]
using that in affine logic $\bang (A \cwedge B)$ is equivalent to $\bang A \cwedge \bang B$; and that $\bang (\bang A \lto B)$ is equivalent to $\bang A \lto \bang B$. \\[1mm]
$(ii)$ We have $\bTrans{(A \vee B)} \Leftrightarrow \bang \lTrans{(A \vee B)} 
	\stackrel{\textup{Part}~(i)}{\Leftrightarrow} \bang (\bang \lTrans{A} \avee \bang \lTrans{B}) 
	\Leftrightarrow \bang \lTrans{A} \avee \bang \lTrans{B} 
	\Leftrightarrow \bTrans{A} \avee \bTrans{B}$ using the fact that $\bTrans{(A \vee B)}$, $\bTrans{A}$ and $\bTrans{B}$ are equivalent to formulas $\bang \lTrans{(A \vee B)}$, $\bang \lTrans{A}$ and $\bang \lTrans{B}$, by Proposition \ref{prop:lb-equivalence}.
\end{proof}

\begin{definition}[Finite types] The finite types $\Type$ are defined inductively as: $\BB, \NN \in \Type$ \emph{(base types)}, and if $\rho, \tau \in \Type$ then $\rho \to \tau \in \Type$ \emph{(function types)}.
\end{definition}

Let $\HAomega$ be the weakly extensional version of Heyting arithmetic in all finite types (see \cite{Troelstra(73)} and \cite[Section 3.3]{K(08)}). We will consider here a presentation of $\HAomega$ where terms are \emph{explicitly} typed, so that it can be considered an $\IL$-theory as described above. For that matter, we assume that $\HAomega$ contains explicit typing predicate symbols and axioms, i.e.
\begin{itemize}
	\item for each finite type $\sigma \in \Type$ we have predicate symbols $\sigma(x)$
	\item we have axioms $\proves \sigma(t)$ for each constant $c^\sigma$, i.e.
	\begin{enumerate}
		\item $\proves \NN(0)$
		\item $\proves (\NN \to \NN)(\suc)$
		\item $\proves (\sigma \to \tau \to \sigma) (\K_{\sigma, \tau})$, for each $\sigma, \tau \in \Type$
		\item $\proves (\sigma \to (\sigma \to \tau) \to (\sigma \to \tau \to \rho) \to \rho)(\S_{\sigma, \tau, \rho})$, for each $\sigma, \tau, \rho \in \Type$
		\item $\proves (\NN \to \sigma \to (\NN \to \sigma \to \sigma) \to \sigma) (\rec_\sigma)$, for each $\sigma \in \Type$
	\end{enumerate}
	\item A family of function symbols $\ap_{\sigma, \tau}(f, x)$ with axioms\footnote{As usual we will normally write the term $\ap_{\sigma, \tau}{(s)}{(t)}$ as simply $s t$.}
	\[ (\sigma \to \tau)(f), \sigma(x) \proves \tau(\ap_{\sigma, \tau}(f, x)) \]
\end{itemize}
With the above axioms we can indeed show that for each term $(t[\pvec x^{\pvec \sigma}])^\tau$ with intrinsic type $\tau$ and free-variables $\pvec x^{\pvec \sigma}$, we can derive $\pvec \sigma(\pvec x) \proves \tau(t[\pvec x])$ in the system above. Each formula $A$ (with intrinsic types) must also be mapped to a formula $\eTrans{A}$ (with explicit types) inductively -- e.g. taking $\eTrans{(\forall x^\sigma A)} := \forall x (\sigma(x) \to \eTrans{A})$. Then, each original axiom $\Gamma(\pvec x^{\pvec \sigma}) \proves A(\pvec x^{\pvec \sigma})$ can be stated with explicit types as $\pvec \sigma(\pvec x), \eTrans{(\Gamma(\pvec x^{\pvec \sigma}))} \proves \eTrans{(A(\pvec x^{\pvec \sigma}))}$. It will be important for the verification of the soundness (for the Dialectica interpretation), that in $\HAomega$ quantifier-free formulas are decidable (see \cite[Proposition 3.17]{K(08)}). For this decidability result to hold in our setting with explicit typing, we also need to assume that (in the verifying system) for each predicate symbol $\tau(x)$, where $\tau \in \Type$, we have in $\HAomega$ a term $t_\tau(x)$ such that $\tau(x) \proves \BB(t_\tau(x))$, and $t_\tau(x) = \true$ is provably equivalent to $\tau(x)$. 

We then define the $\AL$-theory $\AHAomega = \lTrans{(\HAomega)}$. Note that although the type $\BB$ (and its corresponding axioms) are not usually included explicitly in the definition of $\HAomega$, these are indeed definable by taking $\true := 0$ and $\false := 1$ and $\BB(x) := (x = 0) \vee (x = 1)$.

If we omit the arithmetical constants (zero, successor and recursors), and their corresponding axioms from $\HAomega$, we obtain a purely intuitionistic theory of finite types, which we call $\IL^\omega$. In this case the booleans are no longer definable, and hence we assume that $\IL^\omega$ is also an extension of $\IL^\BB$. Its corresponding $\AL$-theory will be denoted $\AL^\omega = \lTrans{(\IL^\omega)}$.

Finally, let us denote by $\Type^*$ the extension of the set of finite types with an extra closure condition: if $\rho \in \Type^*$ then $\rho^* \in \Type^*$ \emph{(finite sequence types)}, and let $\eHAomegaS$ be the system described in \cite[Section 2.1]{BergBriseidSafarik2012}, also presented with explicit types, as described above. 

\begin{remark}[Majorizability] In $\eHAomegaS$ we can extend Bezem's majorizability relation to include the finite sequence types:
\[
\begin{array}{lcl}
	x \leq_{\tau}^* y & \pdefin & x \leq_{\tau} y, \quad\quad \mbox{for $\tau  \in \{\NN, \BB\}$} \\[2mm]
	f \leq_{\tau \to \rho}^* g & \pdefin & \forall y^\tau, x^\tau \leq_{\tau}^* y (f x \leq_{\rho}^* g y \wedge g x \leq_{\rho}^* g y) \\[2mm]
	a \leq_{\tau^*}^* b & \pdefin & |a| \leq |b| \wedge \forall i < |a| (a_i \leq^*_\tau b_i) \wedge \forall i < |b| (b_i \leq^*_\tau b_i)
\end{array}
\]
The main property we need is that for each closed term $s$ (of type $\tau$) there exists a closed term $t$ (of the same type $\tau$) such that $s \leq^*_\tau t$ provably in $\eHAomegaS$. This is indeed the case by observing that 
\[ \tilde{L} s^{\sigma^*} f^{\tau \to \sigma \to \tau} z^\tau =_\tau L s (\lambda v^\tau \lambda a^\sigma . \mbox{$\max_\tau$} (v, f v a)) z  \]
majorizes the list recursor $L$ (see \cite[Section 2.1]{BergBriseidSafarik2012}), where $\max_\tau(\cdot, \cdot)$ is defined pointwise for function types, and for finite sequence types we take
\[ 
\mbox{$\max_{\tau^*}$}(s,t) = \langle \mbox{$\max_\tau$}(s_i, t_i) \rangle_{i < \max(|s|, |t|)}
\]
taking $s_i = 0$ when $i \geq |s|$. 
\end{remark}

Again, by the Girard translations, from the $\IL$-theory $\eHAomegaS$ we can obtain an $\AL$-theory $\eAHAomegaS = \lTrans{(\eHAomegaS)}$. 

\section{Parametrised Interpretation of $\AL$}
\label{sec-al-inter}

We present now a parametrised interpretation of a ``source" $\AL$-theory $\Asource$ into a ``target" $\AL$-theory $\Atarget$. In order to ensure that the parametrised interpretation is sound, we will need to stipulate a few assumption about $\Asource$ and $\Atarget$:

\begin{enumerate}

	\item[\Assumption{1}] The target theory $\Atarget$ is an extension of $\AL^\omega$ -- defined in Section \ref{Subsection_theories} -- so that we can work with typed $\lambda$-terms as witnesses.
	
	\item[\Assumption{2}] In the source theory $\Asource$, the predicate symbols are divided into two groups: the \emph{computational symbols}, denoted by $\Pred{\Asource}^c$, and the \emph{non-computational symbols}, denoted by $\Pred{\Asource}^{nc}$. The predicate symbols of $\Asource$ are also assumed to be predicate symbols of $\Atarget$.

	\item[\Assumption{3}] For each computational predicate symbol $P(\pvec x) \in \Pred{\Asource}^c$ of $\Asource$, of arity $n$, we have associated in $\Atarget$ a $(n+1)$-ary formula $\bound{\pvec x}{a}{P}$, and a finite type $\wtype{P}$ in which the witnesses $a$ of $P(\pvec x)$ will live in. We will call $\wtype{P}$ the \emph{witnessing type} of $P$. 
	We write $\forall \bound{\pvec x}{a}{P} \, A$ and $\exists \bound{\pvec x}{a}{P} \, A$ as abbreviations for $\forall \pvec x (\bound{\pvec x}{a}{P} \lto A)$ and $\exists \pvec x (\bound{\pvec x}{a}{P} \cwedge A)$, respectively. We assume that, over $\Atarget$, $\bound{\pvec x}{a}{P}$ is stronger than $P(\pvec x)$, i.e. 
	\begin{enumerate}
		\item[] $\bound{\pvec x}{a}{P} \proves_{\Atarget} P(\pvec x)$. 
	\end{enumerate}

	\item[\Assumption{4}] For each finite type $\tau$ we associate in $\Atarget$ a formula $\Wit_\tau(x)$, which we will use to restrict the domain of the witnesses and counter-witnesses. We also assume that $\bound{\pvec x}{a}{P}$ implies that $a$ is in $\Wit$, i.e. 
	\begin{enumerate}
		\item[] $\bound{\pvec x}{a}{P} \proves_{\Atarget} \Wit_{\wtype{P}}(a)$.
	\end{enumerate}
When $\pvec \tau$ is a tuple of finite types $\tau_1, \ldots, \tau_n$, we write $\Wit_{\pvec \tau}(x_1, \dots, x_n)$ as an abbreviation for $\Wit_{\tau_1}(x_1), \dots, \Wit_{\tau_n}(x_n)$, when this appears in the context of a sequent, or for $\Wit_{\tau_1}(x_1) \otimes \ldots \otimes \Wit_{\tau_n}(x_n)$, when this appears in the conclusion of a sequent. We assume that, provably in $\Atarget$, the combinators $\S_{\rho, \tau, \sigma}$ and $\K_{\rho, \tau}$ are in $\Wit$, and that the application of a function in $\Wit$ to an argument in $\Wit$ will also be in $\Wit$, i.e.
    \begin{enumerate}
        \item[\WitK] $\proves_{\Atarget} \Wit_{\rho \to \tau \to \rho}(\K_{\rho, \tau})$
        \item[\WitS] $\proves_{\Atarget} \Wit_{(\rho \to \tau \to \sigma) \to (\rho \to \tau) \to \rho \to \sigma}(\S_{\rho, \tau, \sigma})$
        \item[\WitApp] $\Wit_\tau(x), \Wit_{\tau \to \rho}(f) \proves_{\Atarget} \Wit_{\rho}(f x)$
    \end{enumerate}

	\item[\Assumption{5}] For each formula $A$ of $\Atarget$, tuple of variables $\pvec x = x_1, \ldots, x_n$, and finite types $\pvec \tau = \tau_1, \ldots, \tau_n$ we associate a tuple of \emph{bounding types} $\btype{\pvec \tau}$ and a formula $\ubq{\pvec \tau}{\pvec x}{\pvec a} A$, in which the variables $\pvec x$ are no longer free. We do not assume that the tuple of finite types $\btype{\pvec \tau}$ has the same length as $\pvec \tau$. The intuition is that $\pvec x$ ranges over elements of type $\pvec \tau$, whereas the bounds $\pvec a$ range over possibly different types $\btype{\pvec \tau}$. We use this parameter to interpret $\bang A$. This parameter is assumed to satisfy:
    \begin{enumerate}
        \item[\Quant{1}] If $A \proves_{\Atarget} B$ then  $\ubq{\pvec \tau}{\pvec x}{\pvec a} A \proves_{\Atarget} \ubq{\pvec \tau}{\pvec x}{\pvec a} B$
        \item[\Quant{2}] $\proves_{\Atarget} \ubq{\pvec \tau}{\pvec x}{\pvec a} \bang \Wit_{\pvec \tau}(\pvec x)$
    \end{enumerate}
    and, for each formula $A$ of $\Atarget$, tuple of variables $\pvec x$, and types $\pvec \tau$ and $\rho$ we assume that there exist terms $\pvec \singleton{(\cdot)}, \join{(\cdot)}{(\cdot)}$ and $\comp{(\cdot)}{(\cdot)}$ of $\Atarget$ such that
    \begin{enumerate}
    \item[\ConUnit] $\bang \Wit_{\pvec \tau}(\pvec z), \bang \ubq{\pvec \tau}{\pvec x}{\pvec \singleton{(\pvec z)}} A \proves_{\Atarget} A[\pvec z/\pvec x]$
        \item[]  $\proves_{\Atarget} \Wit_{\pvec \tau \to \btype{\pvec \tau}}(\pvec \eta)$
    \item[\ConStrength] $\bang \Wit_{\pvec \tau, \pvec \tau}(\pvec x_1,\pvec x_2), \bang \ubq{\pvec \tau}{\pvec x}{(\join{\pvec x_1}{\pvec x_2})} A 
            \proves_{\Atarget} \ubq{\pvec \tau}{\pvec x}{\pvec x_1} A \cwedge \ubq{\pvec \tau}{\pvec x}{\pvec x_2} A$
        \item[] $\proves_{\Atarget} \Wit_{\btype{\pvec \tau} \to \btype{\pvec \tau} \to \btype{\pvec \tau}}(\lambda \pvec x_1, \pvec x_2 . \join{\pvec x_1}{\pvec x_2})$
    \item[\ConApp]  $\bang \Wit_{\pvec \rho \to \btype{\pvec \tau}}(\pvec f), \bang \Wit_{\btype{\pvec \rho}}(\pvec z), \bang \ubq{\pvec \tau}{\pvec x}{(\comp{\pvec f}{\pvec z})} A 
            \proves_{\Atarget} \ubq{\pvec \rho}{\pvec y}{\pvec z} \bang \ubq{\pvec \tau}{\pvec x}{\pvec f \pvec y} A$
        \item[] $\proves_{\Atarget} \Wit_{(\pvec \rho \to \btype{\pvec \tau}) \to \btype{\pvec \rho} \to \btype{\pvec \tau}}(\lambda \pvec f, \pvec z . \comp{\pvec f}{\pvec z})$
	\end{enumerate}

\end{enumerate}

A term $t[\pvec x]$, with free variables $\pvec x$, is called \emph{typable} in $\Atarget$ if $\pvec \rho(\pvec x) \proves_{\Atarget} \tau(t[\pvec x])$ for some $\pvec \rho$ and $\tau$. We say that a typable term $t[\pvec x]$ \emph{is in $\Wit$} if $\Wit_{\pvec \rho}(\pvec x) \proves_{\Atarget} \Wit_{\pvec \tau}(t[\pvec x])$.

\begin{lemma} \label{lem-W-closure} Let $t$ be a term of $\Atarget$ built from variables $\pvec x = x_1, \ldots, x_n$ and the combinators $\K$ and $\S$ via application. Then $\Wit(x_1, \ldots, x_n) \proves_{\Atarget} \Wit(t)$.
\end{lemma}

\begin{proof} Induction on the structure of $t$ using assumptions \WitS, \WitK, \WitApp.
\end{proof}

In each instantiation we will consider different choices for the parameters $\{\bound{\pvec x}{a}{P}\}_{P \in \Pred{\Asource}^c}$, $\{ \wtype{P} \}_{P \in \Pred{\Asource}^c}$, $\{ \Wit_\tau(x) \}_{\tau \in \Type}$, and $\{\ubq{\pvec \tau}{\pvec x}{\pvec a} A\}_{A \in \Formulas{\Atarget}, \pvec \tau \in \Type}$ for each choice of variables $\pvec x$. 

\begin{definition}[Adequate parameters in $\AL$] Given theories $\Asource$ and $\Atarget$, a choice of parameters will be called \emph{adequate} for $(\Asource, \Atarget)$ if assumptions \Assumption{1} -- \Assumption{5} hold. Given a class of formulas ${\mathcal C} \subseteq \Formulas{\Atarget}$, we say that the choice of parameters in $\Atarget$ is ${\mathcal C}$-\emph{adequate} for $(\Asource, \Atarget)$ if it is adequate for $(\Asource, \Atarget)$ when assumption \Assumption{5} is only required to hold for formulas in ${\mathcal C}$.
\end{definition} 

\subsection{Parametrised interpretation of $\AL$ theories}

Assume now a given choice of $\AL$-theories $\Asource$ (source theory) and $\Atarget$ (target theory) satisfying the assumptions stated above. Recall that we write $\varepsilon$ for the empty tuple of terms. Let us use the same notation, and write $\varepsilon$ for an empty tuple of types as well.

\begin{definition} We generalise the notion of \emph{witnessing type} to all formulas by defining for each formula $A$ tuples of types $\pvec \tau^+_A$ and $\pvec \tau^-_A$ inductively as
\[
\begin{array}{lcl}
\tau^+_{P} & \pdefin & \wtype{P}, \quad \mbox{for $P \in \Pred{\Asource}^c$} \\[1mm]
\tau^+_{P} & \pdefin & \varepsilon, \quad \mbox{for $P \in \Pred{\Asource}^{nc}$} \\[1mm]
\pvec \tau^+_{A \lto B} & \pdefin & \pvec \tau^+_A \to \pvec \tau^+_B, \tau^+_A \to \tau^-_B \to \pvec \tau^-_A \\[1mm]
\pvec \tau^+_{A \cwedge B} & \pdefin & \pvec \tau^+_A, \pvec \tau^+_B \\[1mm]
\pvec \tau^+_{\exists z A} & \pdefin & \pvec \tau^+_A \\[1mm]
\pvec \tau^+_{\forall z A} & \pdefin & \pvec \tau^+_A \\[1mm]
\pvec \tau^+_{\bang A} & \pdefin & \pvec \tau^+_A
\end{array}
\quad
\begin{array}{lcl}
\tau^-_{P} & \pdefin & \varepsilon, \quad \mbox{for $P \in \Pred{\Asource}^c$} \\[1mm]
\tau^-_{P} & \pdefin & \varepsilon, \quad \mbox{for $P \in \Pred{\Asource}^{nc}$} \\[1mm]
\pvec \tau^-_{A \lto B} & \pdefin & \pvec \tau^+_A, \tau^-_B \\[1mm]
\pvec \tau^-_{A \cwedge B} & \pdefin & \pvec \tau^-_A, \pvec \tau^-_B \\[1mm]
\pvec \tau^-_{\exists z A} & \pdefin & \pvec \tau^-_A \\[1mm]
\pvec \tau^-_{\forall z A} & \pdefin & \pvec \tau^-_A \\[1mm]
\pvec \tau^-_{\bang A} & \pdefin & \btype{\pvec \tau}^-_A
\end{array}
\]
Given a tuple of formulas $\Gamma = A_1, \ldots, A_n$ we write $\tau_{\Gamma}^+$ (resp., $\tau_{\Gamma}^-$) for the tuple $\tau_{A_1}^+, \ldots, \tau_{A_n}^+$ (resp. $\tau_{A_1}^-, \ldots, \tau_{A_n}^-$).
\end{definition}

We can now present the parametrised interpretation of $\Asource$ into $\Atarget$:

\begin{definition}[Parametrised $\AL$-interpretation] \label{inter} For each formula $A$ of $\Asource$, let us associate a formula $\uInter{A}{\pvec x}{\pvec y}$ of $\Atarget$, with two fresh lists of free-variables $\pvec x$ and $\pvec y$, inductively as follows: for computational predicate symbols $P \in \Pred{\Asource}^c$ we let 
\eqleft{
\begin{array}{lcl}
\uInter{P(\pvec x)}{a}{\varepsilon} & \pdefin & \bound{\pvec x}{a}{P}, 
\end{array}
}
whereas for non-computational predicate symbols $P \in \Pred{\Asource}^{nc}$ we let
\eqleft{
\begin{array}{lcl}
\uInter{P(\pvec x)}{\varepsilon}{\varepsilon} & \pdefin & P(\pvec x).
\end{array}
}
Assuming $A$ and $B$ have interpretations $\uInter{A}{\pvec x}{\pvec y}$ and $\uInter{B}{\pvec v}{\pvec w}$, then we define
\eqleft{
\begin{array}{rclcrcl}
\uInter{A \lto B}{\pvec f, \pvec g}{\pvec x, \pvec w} & \pdefin & \uInter{A}{\pvec x}{\pvec g \pvec x \pvec w} \lto \uInter{B}{\pvec f \pvec x}{\pvec w} & \quad &
\uInter{\exists z A}{\pvec x}{\pvec y} & \pdefin & \exists z \uInter{A}{\pvec x}{\pvec y} \\[2mm]
\uInter{A \cwedge B}{\pvec x, \pvec v}{\pvec y, \pvec w} & \pdefin & \uInter{A}{\pvec x}{\pvec y} \cwedge \uInter{B}{\pvec v}{\pvec w} & &
\uInter{\forall z A}{\pvec x}{\pvec y} & \pdefin & \forall z \uInter{A}{\pvec x}{\pvec y} \\[2mm]
\uInter{\bang A}{\pvec x}{\pvec a} & \pdefin & \bang \ubq{\pvec \tau^-_A}{\pvec y}{\pvec a} \uInter{A}{\pvec x}{\pvec y} & & & &
\end{array}
}
Given a tuple of formulas $\Gamma = A_1, \ldots, A_n$, we write $\uInter{\Gamma}{\pvec x_1, \ldots, \pvec x_n}{\pvec y_1, \ldots, \pvec y_n}$ is an abbreviation for $\uInter{A_1}{\pvec x_1}{\pvec y_1}, \ldots, \uInter{A_n}{\pvec x_n}{\pvec y_n}$, assuming $A_i$ has interpretation $\uInter{A}{\pvec x_i}{\pvec y_i}$.
\end{definition}

If $A$ has interpretation $\uInter{A}{\pvec x}{\pvec y}$ we call $\pvec x$ the \emph{witnesses} of $A$, and $\pvec y$ the \emph{counter-witnesses}. We say that a formula $A$ has \emph{no computational content} if its interpretation is $\uInter{A}{\varepsilon}{\varepsilon}$, i.e. if the tuples of witnesses and counter-witnesses are both empty. Note that all of the computational content of a formula comes from the interpretation of the computational predicate symbols. The logical connectives ($\lto$ and $\otimes$), the quantifiers ($\forall$ and $\exists$) and the exponential ($\bang$) simply translate witnesses and counter-witness for the subformulas into witnesses and counter-witnesses for the compound formula. If the subformulas have no computational content then the compound formula will not have any computational content either.

\subsection{Soundness}

Given a tuple of types $\pvec \rho = \rho_1, \ldots, \rho_n$ and a type $\sigma$, let us write $\pvec \rho \to \sigma$ as an abbreviation for the type $\rho_1 \to \ldots \to \rho_n \to \sigma$. Given tuples of terms $\pvec t = t_1, \ldots, t_n$ and $\pvec s$ we write $\pvec t \pvec s$ for the tuple $t_1 \pvec s, \ldots, t_n \pvec s$.

\begin{definition}[Witnessable $\AL$ sequents] \label{def-witnessable} A sequent $\Gamma \proves A$ of $\Asource$ is said to be \emph{witnessable} in $\Atarget$ if there are tuples of closed terms $\pvec \gamma, \pvec a$ of $\Atarget$ such that 
\begin{enumerate}
	\item[(i)] $\proves_{\Atarget} \Wit_{\pvec \tau^+_{\Gamma} \to \pvec \tau^-_A \to \tau^-_{\Gamma}}(\pvec \gamma)$ and $\proves_{\Atarget} \Wit_{\pvec \tau^+_{\Gamma} \to \pvec \tau^+_A}(\pvec a)$
	\item[(ii)] $\bang \Wit_{\pvec \tau^+_{\Gamma}, \pvec \tau^-_A}(\pvec x,\pvec w), \uInter{\Gamma}{\pvec x}{\pvec \gamma \pvec x \pvec w} 
                    	\proves_{\Atarget} \uInter{A}{\pvec a \pvec x}{\pvec w}$
\end{enumerate}
\end{definition}

\begin{definition}[Sound $\AL$-interpretation] An $\AL$-interpretation of $\Asource$ into $\Atarget$ is said to be \emph{sound} if the provable sequents of $\Asource$ are witnessable in $\Atarget$. 
\end{definition} 

\begin{theorem}[Soundness of $\AL$-interpretation] \label{soundness1} Assume a fixed choice of the parameters in $\Atarget$. If
\begin{enumerate}
	\item[(i)] this choice is adequate for the formulas $\uInter{A}{\pvec x}{\pvec y}$, for all $A$ in $\Asource$, and
	\item[(ii)] the non-logical axioms of $\Asource$ are witnessable in $\Atarget$, 
\end{enumerate}
then this instance of the parametrised $\AL$-interpretation of Definition \ref{inter} is sound.
\end{theorem}

\begin{proof} The proof is similar to that of \cite[Thm. 2.2]{FO(11)}. Under the assumptions of the theorem, we must show that the provable sequents $\Gamma \proves A$ of $\Asource$ are witnessable in $\Atarget$. We do this by induction on the derivation of $\Gamma \proves_{\Asource} A$. The axioms of $\Asource$ are witnessable by assumption. Let us consider each of the logical rules and show that they turn witnessable premisses into witnessable conclusions. In each case we need to prove points $(i)$ and $(ii)$ of Definition \ref{def-witnessable}. Point $(i)$, however, will follow directly from the induction hypothesis and Lemma \ref{lem-W-closure}, since the terms witnessing the conclusion of each rule will be build from the terms witnessing the premise via simple $\lambda$-term constructions (definable from $\S$ and $\K$). Therefore, we will focus our attention on proving point $(ii)$. \\[1mm]
\emph{Cut}. By induction hypothesis we have closed terms $\pvec a_0, \pvec a_1, \pvec \gamma, \pvec \delta, \pvec b$ such that
\begin{itemize}
	\item[\IH{(i)}] $\proves \Wit(\pvec a_0)$ and $\proves \Wit(\pvec \gamma)$ and $\proves \Wit(\pvec b)$ and $\proves \Wit(\pvec \delta)$ and $\proves \Wit(\pvec a_1)$
	\item[\IH{(ii)}] $\bang \Wit(\pvec u, \pvec y), \uInter{\Gamma}{\pvec u}{\pvec \gamma \pvec u \pvec y} 
                    	\proves \uInter{A}{\pvec a_0 \pvec u}{\pvec y}$
			and
			$\bang \Wit(\pvec v, \pvec x, \pvec w), 
                    	\uInter{\Delta}{\pvec v}{\pvec \delta \pvec v \pvec x \pvec w}, 
			\uInter{A}{\pvec x}{\pvec a_1 \pvec v \pvec x \pvec w} 
                    	\proves \uInter{B}{\pvec b \pvec v \pvec x}{\pvec w}$
\end{itemize}
We claim that the terms $\tilde{\pvec \gamma} \pdefin \lambda \pvec u, \pvec v, \pvec w . \pvec \gamma \pvec u (\pvec a_1 \pvec v (\pvec a_0 \pvec u) \pvec w)$ and $\tilde{\pvec \delta} \pdefin \lambda \pvec u, \pvec v, \pvec w . \pvec \delta \pvec v (\pvec a_0 \pvec u) \pvec w$ and $\tilde{\pvec b} \pdefin \lambda \pvec u, \pvec v . \pvec b \pvec v (\pvec a_0 \pvec u)$ witness the cut rule. Let $\tilde{\pvec a} = \pvec a_1 \pvec v (\pvec a_0 \pvec u) \pvec w$. Using the induction hypothesis \IH{(i)} and \IH{(ii)} we have
\[
\begin{prooftree}
        \[
            \justifies
            \bang \Wit(\pvec u, \pvec v, \pvec w) \proves \bang \Wit(\tilde{\pvec a})
            \using \mbox{\IH{(i)}} 
        \]
        \[
            \[
                \justifies
                \bang \Wit(\pvec u, \pvec y), 
                \uInter{\Gamma}{\pvec u}{\pvec \gamma \pvec u \pvec y} 
                    \proves \uInter{A}{\pvec a_0 \pvec u}{\pvec y}
                \using \mbox{\IH{(ii)}}
            \]
            \justifies
            \bang \Wit(\pvec u, \tilde{\pvec a}), 
            \uInter{\Gamma}{\pvec u}{\tilde{\pvec \gamma} \pvec u \pvec v \pvec w} 
                \proves \uInter{A}{\pvec a_0 \pvec u}{\tilde{\pvec a}}
            \using {[\frac{\tilde{\pvec a}}{\pvec y}]}
        \]
        \justifies
        \bang \Wit(\pvec u, \pvec v, \pvec w), 
        \uInter{\Gamma}{\pvec u}{\tilde{\pvec \gamma} \pvec u \pvec v \pvec w} 
            \proves \uInter{A}{\pvec a_0 \pvec u}{\tilde{\pvec a}}
        \using (\rulecut) 
\end{prooftree}
\]
and
\[
\begin{prooftree}
        \[
            \justifies
            \Wit(\pvec u) \proves \Wit(\pvec a_0 \pvec u)
            \using \mbox{\IH{(i)}} 
        \]
        \[
            \[
                \justifies
                \bang \Wit(\pvec v, \pvec x, \pvec w), 
                \uInter{\Delta}{\pvec v}{\pvec \delta \pvec v \pvec x \pvec w}, 
                \uInter{A}{\pvec x}{\pvec a_1 \pvec v \pvec x \pvec w} 
                    \proves \uInter{B}{\pvec b \pvec v \pvec x}{\pvec w}
                \using \mbox{\IH{(ii)}}
            \]
            \justifies
            \bang \Wit(\pvec v, \pvec a_0 \pvec u, \pvec w), 
            \uInter{\Delta}{\pvec v}{\tilde{\pvec \delta} \pvec u \pvec v \pvec w}, 
            \uInter{A}{\pvec a_0 \pvec u}{\tilde{\pvec a}} 
                \proves \uInter{B}{\tilde{\pvec b} \pvec u \pvec v}{\pvec w}
            \using {[\frac{\pvec a_0 \pvec u}{\pvec x}]}
        \]
        \justifies
        \bang \Wit(\pvec u, \pvec v, \pvec w), 
        \uInter{\Delta}{\pvec v}{\tilde{\pvec \delta} \pvec u \pvec v \pvec w}, 
        \uInter{A}{\pvec a_0 \pvec u}{\tilde{\pvec a}} 
            \proves \uInter{B}{\tilde{\pvec b} \pvec u \pvec v}{\pvec w}
        \using (\rulecut) 
\end{prooftree}
\]
so that with another cut we get $\bang \Wit(\pvec u, \pvec v, \pvec w),
    \uInter{\Gamma}{\pvec u}{\tilde{\pvec \gamma} \pvec u \pvec v \pvec w},   
    \uInter{\Delta}{\pvec v}{\tilde{\pvec \delta} \pvec u \pvec v \pvec w}
        \proves \uInter{B}{\tilde{\pvec b} \pvec u \pvec v}{\pvec w}$. \\[1mm]
$(\cwedge\textup{R})$. By induction hypothesis we have closed terms $\pvec a, \pvec b, \pvec \gamma, \pvec \delta$ in $\Wit$ such that
\begin{enumerate}
	\item[(\rm{IH})] $\bang \Wit(\pvec u, \pvec y), \uInter{\Gamma}{\pvec u}{\pvec \gamma \pvec u \pvec y} 
                    	\proves \uInter{A}{\pvec a \pvec u}{\pvec y}$
			and
			$\bang \Wit(\pvec v, \pvec w), 
                    	\uInter{\Delta}{\pvec v}{\pvec \delta \pvec v \pvec w}
                    	\proves \uInter{B}{\pvec b \pvec v}{\pvec w}$
\end{enumerate}
We claim that the terms $\lambda \pvec u, \pvec y, \pvec v, \pvec w . \pvec \gamma \pvec u \pvec y$ and $\lambda \pvec u, \pvec y, \pvec v, \pvec w . \pvec \delta \pvec v \pvec w$ and $\lambda \pvec u, \pvec v . \pvec a \pvec u$ and $\lambda \pvec u, \pvec v . \pvec b \pvec v$ witness the conclusion of this rule. We have
\[
\begin{prooftree}
\[
    \[
        \justifies
        \bang \Wit(\pvec u, \pvec y), 
        \uInter{\Gamma}{\pvec u}{\pvec \gamma \pvec u \pvec y} 
            \proves \uInter{A}{\pvec a \pvec u}{\pvec y}        
	\using \mbox{(\rm{IH})} 
    \]
    \[
        \justifies
        \bang \Wit(\pvec v, \pvec w), 
        \uInter{\Delta}{\pvec v}{\pvec \delta \pvec v \pvec w}
            \proves \uInter{B}{\pvec b \pvec v}{\pvec w}
        \using \mbox{(\rm{IH})} 
    \]
    \justifies
    \bang \Wit(\pvec u, \pvec y, \pvec v, \pvec w),
    \uInter{\Gamma}{\pvec u}{\pvec \gamma \pvec u \pvec y},
    \uInter{\Delta}{\pvec v}{\pvec \delta \pvec v \pvec w}
        \proves \uInter{A}{\pvec a \pvec u}{\pvec y} \cwedge \uInter{B}{\pvec b \pvec v}{\pvec w}
    \using (\cwedge\textup{R})
\]
\justifies
\bang \Wit(\pvec u, \pvec y, \pvec v, \pvec w),
\uInter{\Gamma}{\pvec u}{\pvec \gamma \pvec u \pvec y},
\uInter{\Delta}{\pvec v}{\pvec \delta \pvec v \pvec w}
    \proves \uInter{A \cwedge B}{\pvec a \pvec u, \pvec b \pvec v}{\pvec y, \pvec w}
\using (\textup{D}\ref{inter})
\end{prooftree}
\]
$(\cwedge\textup{L})$. By induction hypothesis we have closed terms $\pvec \gamma, \pvec a, \pvec b, \pvec c$  in $\Wit$ such that
\begin{enumerate}
	\item[(\rm{IH})] $\bang \Wit(\pvec u, \pvec x, \pvec v, \pvec w),
                            	\uInter{\Gamma}{\pvec u}{\pvec \gamma \pvec u \pvec x \pvec v \pvec w}, 
				\uInter{A}{\pvec x}{\pvec a \pvec u \pvec x \pvec v \pvec w}, 
				\uInter{B}{\pvec v}{\pvec b \pvec u \pvec x \pvec v \pvec w}
                                    \proves \uInter{C}{\pvec c \pvec u \pvec x \pvec v}{\pvec w}$
\end{enumerate}
We claim that the terms $\pvec \gamma, \pvec a, \pvec b$ and $\pvec c$ witness the conclusion of the rule. We have
\[
\begin{prooftree}
\[
    \[
        \justifies
        \bang \Wit(\pvec u, \pvec x, \pvec v, \pvec w),
        \uInter{\Gamma}{\pvec u}{\pvec \gamma \pvec u \pvec x \pvec v \pvec w}, 
        \uInter{A}{\pvec x}{\pvec a \pvec u \pvec x \pvec v \pvec w}, 
        \uInter{B}{\pvec v}{\pvec b \pvec u \pvec x \pvec v \pvec w}
            \proves \uInter{C}{\pvec c \pvec u \pvec x \pvec v}{\pvec w}
	\using \mbox{(\rm{IH})} 
    \]
    \justifies
    \bang \Wit(\pvec u, \pvec x, \pvec v, \pvec w),
    \uInter{\Gamma}{\pvec u}{\pvec \gamma \pvec u \pvec x \pvec v \pvec w}, 
    \uInter{A}{\pvec x}{\pvec a \pvec u \pvec x \pvec v \pvec w}
    \cwedge
    \uInter{B}{\pvec v}{\pvec b \pvec u \pvec x \pvec v \pvec w}
        \proves \uInter{C}{\pvec c \pvec u \pvec x \pvec v}{\pvec w}
    \using (\cwedge\textup{L})
\]
\justifies
\bang \Wit(\pvec u, \pvec x, \pvec v, \pvec w),
\uInter{\Gamma}{\pvec u}{\pvec \gamma \pvec u \pvec x \pvec v \pvec w}, 
\uInter{A \cwedge B}{\pvec x, \pvec v}{\pvec a \pvec u \pvec x \pvec v \pvec w, \pvec b \pvec u \pvec x \pvec v \pvec w}
    \proves \uInter{C}{\pvec c \pvec u \pvec x \pvec v}{\pvec w}
\using (\textup{D}\ref{inter})
\end{prooftree}
\]
(\emph{$\lto\!\textup{R}$}). By induction hypothesis we have closed terms $\pvec \gamma, \pvec a, \pvec b$ in $\Wit$ such that
\begin{enumerate}
	\item[(\rm{IH})] $\bang \Wit(\pvec u, \pvec x, \pvec w), 
                    	\uInter{\Gamma}{\pvec u}{\pvec \gamma \pvec u \pvec x \pvec w}, 
			\uInter{A}{\pvec x}{\pvec a \pvec u \pvec x \pvec w} 
                    	\proves \uInter{B}{\pvec b \pvec u \pvec x}{\pvec w}$
\end{enumerate}
We claim that the terms $\pvec \gamma$ and $\pvec a$ and $\pvec b$ witness the conclusion of this rule. We have

\[
\begin{prooftree}
    \[
        \[
            \justifies
            \bang \Wit(\pvec u, \pvec x, \pvec w), 
            \uInter{\Gamma}{\pvec u}{\pvec \gamma \pvec u \pvec x \pvec w}, 
            \uInter{A}{\pvec x}{\pvec a \pvec u \pvec x \pvec w} 
                \proves \uInter{B}{\pvec b \pvec u \pvec x}{\pvec w}
            \using \mbox{(\rm{IH})} 
        \]
        \justifies
        \bang \Wit(\pvec u, \pvec x, \pvec w), 
        \uInter{\Gamma}{\pvec u}{\pvec \gamma \pvec u \pvec x \pvec w}
            \proves \uInter{A}{\pvec x}{\pvec a \pvec u \pvec x \pvec w} \lto \uInter{B}{\pvec b \pvec u \pvec x}{\pvec w}
        \using (\lto\!\textup{R})
    \]
    \justifies
    \bang \Wit(\pvec u, \pvec x, \pvec w), 
    \uInter{\Gamma}{\pvec u}{\pvec \gamma \pvec u \pvec x \pvec w}
        \proves \uInter{A \lto B}{\pvec a \pvec u, \pvec b \pvec u}{\pvec x, \pvec w}
    \using (\textup{D}\ref{inter})
\end{prooftree}
\]
(\emph{$\lto\!\textup{L}$}). By induction hypothesis we have closed terms $\pvec a, \pvec b, \pvec c, \pvec \gamma, \pvec \delta$ such that
\begin{enumerate}
	\item[\IH{(i)}] $\proves \Wit(\pvec a)$ and $\proves \Wit(\pvec b)$ and $\proves \Wit(\pvec c)$ and $\proves \Wit(\pvec \gamma)$ and $\proves \Wit(\pvec \delta)$
	\item[\IH{(ii)}] $\bang \Wit(\pvec u, \pvec y),
                                \uInter{\Gamma}{\pvec u}{\pvec \gamma \pvec u \pvec y} 
                                \proves \uInter{A}{\pvec a \pvec u}{\pvec y}$ and
                                $\bang \Wit(\pvec w, \pvec v, \pvec z), 
                                \uInter{\Delta}{\pvec w}{\pvec \delta \pvec w \pvec v \pvec z}, 
                                \uInter{B}{\pvec v}{\pvec b \pvec w \pvec v \pvec z} 
                                    \proves \uInter{C}{\pvec c \pvec w \pvec v}{\pvec z}$
\end{enumerate}
Using \IH{(i)} and \IH{(ii)} we have
\[
\begin{prooftree}
    \[
        \[
            \justifies
            \bang \Wit(\pvec u, \pvec y),
            \uInter{\Gamma}{\pvec u}{\pvec \gamma \pvec u \pvec y} 
                \proves \uInter{A}{\pvec a \pvec u}{\pvec y}
            \using \mbox{\IH{(ii)}} 
        \]
        \justifies
        \bang \Wit(\pvec u, \pvec g (\pvec a \pvec u) (\pvec b \pvec w (\pvec f(\pvec a \pvec u)) \pvec z)),
        \uInter{\Gamma}{\pvec u}{\pvec \gamma \pvec u (\pvec g (\pvec a \pvec u) 
        (\pvec b \pvec w (\pvec f(\pvec a \pvec u)) \pvec z))} 
            \proves \uInter{A}{\pvec a \pvec u}{\pvec g (\pvec a \pvec u) (\pvec b \pvec w (\pvec f(\pvec a \pvec u)) \pvec z)}
        \using {[\frac{\pvec g (\pvec a \pvec u) (\pvec b \pvec w (\pvec f(\pvec a \pvec u)) \pvec z)}{\pvec y}]}
    \]
    \justifies
    \bang \Wit(\pvec u, \pvec w, \pvec g, \pvec f,  \pvec z),
    \uInter{\Gamma}{\pvec u}{\pvec \gamma \pvec u (\pvec g (\pvec a \pvec u) 
    (\pvec b \pvec w (\pvec f(\pvec a \pvec u)) \pvec z))} 
        \proves \uInter{A}{\pvec a \pvec u}{\pvec g (\pvec a \pvec u) (\pvec b \pvec w (\pvec f(\pvec a \pvec u)) \pvec z)}
        \using \mbox{\IH{(i)}} 
\end{prooftree}
\]
and
\[
\begin{prooftree}
    \[
        \[
            \justifies
            \bang \Wit(\pvec w, \pvec v, \pvec z), 
            \uInter{\Delta}{\pvec w}{\pvec \delta \pvec w \pvec v \pvec z}, 
            \uInter{B}{\pvec v}{\pvec b \pvec w \pvec v \pvec z} 
                \proves \uInter{C}{\pvec c \pvec w \pvec v}{\pvec z}
            \using \mbox{\IH{(ii)}} 
        \]  
        \justifies
        \bang \Wit(\pvec w, \pvec f(\pvec a \pvec u), \pvec z), 
        \uInter{\Delta}{\pvec w}{\pvec \delta \pvec w (\pvec f(\pvec a \pvec u)) \pvec z}, 
        \uInter{B}{\pvec f(\pvec a \pvec u)}{\pvec b \pvec w (\pvec f(\pvec a \pvec u)) \pvec z} 
            \proves \uInter{C}{\pvec c \pvec w (\pvec f(\pvec a \pvec u))}{\pvec z}
        \using {[\frac{\pvec f (\pvec a \pvec u)}{\pvec v}]}
    \]
    \justifies
    \bang \Wit(\pvec w, \pvec f, \pvec u, \pvec z), 
    \uInter{\Delta}{\pvec w}{\pvec \delta \pvec w (\pvec f(\pvec a \pvec u)) \pvec z}, 
    \uInter{B}{\pvec f(\pvec a \pvec u)}{\pvec b \pvec w (\pvec f(\pvec a \pvec u)) \pvec z} 
        \proves \uInter{C}{\pvec c \pvec w (\pvec f(\pvec a \pvec u))}{\pvec z}
    \using \mbox{\IH{(i)}}     
\end{prooftree}
\]
Let us call the two derivations above $\pi_1$ and $\pi_2$, and let $\tilde{\pvec \delta} = \pvec \delta \pvec w (\pvec f(\pvec a \pvec u)) \pvec z$ and $\tilde{\pvec \gamma} = \pvec \gamma \pvec u (\pvec g (\pvec a \pvec u) (\pvec b \pvec w (\pvec f(\pvec a \pvec u)) \pvec z))$. Then:
\[
\begin{prooftree}
    \[
        \pi_1
        \qquad \qquad
        \pi_2
        \justifies
        \bang \Wit(\pvec u, \pvec w, \pvec g, \pvec f, \pvec z),
        \uInter{\Gamma}{\pvec u}{\tilde{\pvec \gamma}},
        \uInter{\Delta}{\pvec w}{\tilde{\pvec \delta}},
        \uInter{A}{\pvec a \pvec u}{\pvec g (\pvec a \pvec u) (\pvec b \pvec w (\pvec f(\pvec a \pvec u)) \pvec z)} \lto
        \uInter{B}{\pvec f(\pvec a \pvec u)}{\pvec b \pvec w (\pvec f(\pvec a \pvec u)) \pvec z} 
            \proves \uInter{C}{\pvec c \pvec w (\pvec f(\pvec a \pvec u))}{\pvec z}
        \using (\lto\textup{L})
    \]
    \justifies
    \bang \Wit(\pvec u, \pvec w, \pvec g, \pvec f, \pvec z),
    \uInter{\Gamma}{\pvec u}{\tilde{\pvec \gamma}},
    \uInter{\Delta}{\pvec w}{\tilde{\pvec \delta}},
    \uInter{A \lto B}{\pvec f, \pvec g}{\pvec a \pvec u, \pvec b \pvec w (\pvec f(\pvec a \pvec u)) \pvec z}
        \proves \uInter{C}{\pvec c \pvec w (\pvec f(\pvec a \pvec u))}{\pvec z}
    \using (\textup{D}\ref{inter})
\end{prooftree}
\]
\emph{Quantifiers}. As the quantifiers are treated uniformly, the witnessing terms of the premises of the rules are also witnessing terms for the conclusions: universal ($\forall$)
\[
\begin{prooftree}
\[
     \[
     \justifies
    \bang \Wit(\pvec u, \pvec x, \pvec w),
    \uInter{\Gamma}{\pvec u}{\pvec \gamma \pvec u \pvec x \pvec w}, 
    \uInter{A(t)}{\pvec x}{\pvec a \pvec u \pvec x \pvec w} 
        \proves \uInter{B}{\pvec b \pvec u \pvec x}{\pvec w}
	\using \mbox{(IH)}
\]
    \justifies
    \bang \Wit(\pvec u, \pvec x, \pvec w),
    \uInter{\Gamma}{\pvec u}{\pvec \gamma \pvec u \pvec x \pvec w}, 
    \forall z \uInter{A(z)}{\pvec x}{\pvec a \pvec u \pvec x \pvec w} 
        \proves \uInter{B}{\pvec b \pvec u \pvec x}{\pvec w}
    \using (\forall \textup{L})
\]
\justifies
\bang \Wit(\pvec u, \pvec x, \pvec w),
\uInter{\Gamma}{\pvec u}{\pvec \gamma \pvec u \pvec x \pvec w}, 
\uInter{\forall z A(z)}{\pvec x}{\pvec a \pvec u \pvec x \pvec w} 
    \proves \uInter{B}{\pvec b \pvec u \pvec x}{\pvec w}
\using (\textup{D}\ref{inter})
\end{prooftree}
\quad
\begin{prooftree}
    \[
		\[
            \justifies
       		\bang \Wit(\pvec u, \pvec y),
        		\uInter{\Gamma}{\pvec u}{\pvec \gamma \pvec u \pvec y} \proves \uInter{A(z)}{\pvec a \pvec u}{\pvec y}
			\using \mbox{(IH)} 
		\]
       		\justifies
        		\bang \Wit(\pvec u, \pvec y),
       		\uInter{\Gamma}{\pvec u}{\pvec \gamma \pvec u \pvec y} \proves \forall z \uInter{A(z)}{\pvec a \pvec u}{\pvec y}
       		\using {(\forall \textup{R})}
    \]
    \justifies
    \bang \Wit(\pvec u, \pvec y),
    \uInter{\Gamma}{\pvec u}{\pvec \gamma \pvec u \pvec y} \proves \uInter{\forall z A(z)}{\pvec a \pvec u}{\pvec y}
    \using (\textup{D}\ref{inter})
\end{prooftree}
\]
and existential ($\exists$)
\[
\begin{prooftree}
\[
   \[
    \justifies
    \bang \Wit(\pvec u, \pvec x,\pvec w),
    \uInter{\Gamma}{\pvec u}{\pvec \gamma \pvec u \pvec x \pvec w}, 
    \uInter{A(z)}{\pvec x}{\pvec a \pvec u \pvec x \pvec w} 
        \proves \uInter{B}{\pvec b \pvec u \pvec x}{\pvec w}
    \using \mbox{(IH)}
     \]
    \justifies
    \bang \Wit(\pvec u, \pvec x,\pvec w),
    \uInter{\Gamma}{\pvec u}{\pvec \gamma \pvec u \pvec x \pvec w}, 
    \exists z \uInter{A(z)}{\pvec x}{\pvec a \pvec u \pvec x \pvec w} 
        \proves \uInter{B}{\pvec b \pvec u \pvec x}{\pvec w}
    \using (\exists \textup{L})
\]
\justifies
\bang \Wit(\pvec u, \pvec x,\pvec w),
\uInter{\Gamma}{\pvec u}{\pvec \gamma \pvec u \pvec x \pvec w}, 
\uInter{\exists z A(z)}{\pvec x}{\pvec a \pvec u \pvec x \pvec w} 
    \proves \uInter{B}{\pvec b \pvec u \pvec x}{\pvec w}
\using (\textup{D}\ref{inter})
\end{prooftree}
\quad
\begin{prooftree}
    \[
	    \[
         \justifies
        \bang \Wit(\pvec u, \pvec y),
        \uInter{\Gamma}{\pvec u}{\pvec \gamma \pvec u \pvec y} 
            \proves \uInter{A(t)}{\pvec a \pvec u}{\pvec y}
		\using \mbox{(IH)}
        \]
        \justifies
         \bang \Wit(\pvec u, \pvec y),
        \uInter{\Gamma}{\pvec u}{\pvec \gamma \pvec u \pvec y} 
            \proves \exists z \uInter{A(z)}{\pvec a \pvec u}{\pvec y}
        \using {(\exists \textup{R})}
    \]
    \justifies
     \bang \Wit(\pvec u, \pvec y),
     \uInter{\Gamma}{\pvec u}{\pvec \gamma \pvec u \pvec y} 
         \proves \uInter{\exists z A(z)}{\pvec a \pvec u}{\pvec y}
    \using (\textup{D}\ref{inter})
\end{prooftree}
\]
\emph{Weakening}. By induction hypothesis the premise of the weakening rule is witnessable, i.e. we have closed terms $\pvec \gamma, \pvec b$ in $\Wit$ such that
\begin{enumerate}
	\item[(\rm{IH})] $\bang \Wit(\pvec u, \pvec w),
                            \uInter{\Gamma}{\pvec u}{\pvec \gamma \pvec u \pvec w} 
                                \proves \uInter{B}{\pvec b \pvec u}{\pvec w}$
\end{enumerate}
Let $\pvec 0$ be an arbitrary closed terms of the appropriate type. We claim that the terms $\lambda \pvec u, \pvec x, \pvec w . \pvec \gamma \pvec u \pvec w$ and $\pvec 0$ and $\lambda \pvec u, \pvec x . \pvec b \pvec u$ witness the conclusion of the weakening rule:
\[
\begin{prooftree}
    \[
        \[
            \justifies
            \bang \Wit(\pvec u,\pvec w), 
            \uInter{\Gamma}{\pvec u}{\pvec \gamma \pvec u \pvec w} 
                \proves \uInter{B}{\pvec b \pvec u}{\pvec w}
            \using \mbox{(\rm{IH})}
        \]
        \justifies
        \bang \Wit(\pvec u,\pvec w, \pvec x), 
        \uInter{\Gamma}{\pvec u}{\pvec \gamma \pvec u \pvec w},
        \bang \ubq{\pvec \tau^-_A}{\pvec y}{\pvec 0 \pvec u \pvec x \pvec w} \uInter{A}{\pvec x}{\pvec y}
            \proves \uInter{B}{\pvec b \pvec u}{\pvec w}
        \using(\textup{wkn})
    \]
    \justifies
    \bang \Wit(\pvec u, \pvec x, \pvec w),
    \uInter{\Gamma}{\pvec u}{(\lambda \pvec u, \pvec x, \pvec w . \pvec \gamma \pvec u \pvec w) \pvec u \pvec x \pvec w},
    \uInter{\bang A}{\pvec x}{\pvec 0 \pvec u \pvec x \pvec w}
        \proves \uInter{B}{(\lambda \pvec u, \pvec x . \pvec b \pvec u) \pvec u \pvec x}{\pvec w}
    \using {(\textup{D}\ref{inter})}
\end{prooftree}
\]
\emph{Contraction}. By induction hypothesis we have closed terms $\pvec \gamma, \pvec a_0, \pvec a_1, \pvec b$ such that
\begin{enumerate}
	\item[\IH{(i)}] $\proves \Wit(\pvec \gamma)$ and $\proves \Wit(\pvec a_i)$, for $i \in \{0,1\}$, and $\proves \Wit(\pvec b)$
	\item[\IH{(ii)}] $\bang \Wit(\pvec u, \pvec x_0, \pvec x_1, \pvec w),
                            \uInter{\Gamma}{\pvec u}{\pvec \gamma \pvec u \pvec x_0 \pvec x_1 \pvec w}, 
                            \uInter{\bang A}{\pvec x_0}{\pvec a_0  \pvec u \pvec x_0 \pvec x_1 \pvec w}, 
                            \uInter{\bang A}{\pvec x_1}{\pvec a_1 \pvec u \pvec x_0 \pvec x_1 \pvec w}
                                \proves \uInter{B}{\pvec b \pvec u \pvec x_0 \pvec x_1}{\pvec w}$
\end{enumerate}
We claim that the terms $\lambda \pvec u, \pvec x, \pvec w . \pvec \gamma \pvec u \pvec x \pvec x \pvec w$ and $\lambda \pvec u, \pvec x, \pvec w . \join{\pvec a_0 \pvec u \pvec x \pvec x \pvec w}{\pvec a_1 \pvec u \pvec x \pvec x \pvec w}$ and $\lambda \pvec u, \pvec x . \pvec b \pvec u \pvec x \pvec x$ witness the conclusion of the contraction rule:
\[
\begin{prooftree}
    \[
        \[
            \[
                \[
                    \[
                        \[
                            \[
                                \justifies
                                \bang \Wit(\pvec u, \pvec x_0, \pvec x_1, \pvec w),
                                \uInter{\Gamma}{\pvec u}{\pvec \gamma \pvec u \pvec x_0 \pvec x_1 \pvec w}, 
                                \uInter{\bang A}{\pvec x_0}{\pvec a_0  \pvec u \pvec x_0 \pvec x_1 \pvec w}, 
                                \uInter{\bang A}{\pvec x_1}{\pvec a_1 \pvec u \pvec x_0 \pvec x_1 \pvec w}
                                    \proves \uInter{B}{\pvec b \pvec u \pvec x_0 \pvec x_1}{\pvec w}                            
                                \using \mbox{\IH{(ii)}}
                            \]
                            \justifies
                            \bang \Wit(\pvec u, \pvec x, \pvec x, \pvec w),
                            \uInter{\Gamma}{\pvec u}{\pvec \gamma \pvec u \pvec x \pvec x \pvec w}, 
                            \uInter{\bang A}{\pvec x}{\pvec a_0  \pvec u \pvec x \pvec x \pvec w}, 
                            \uInter{\bang A}{\pvec x}{\pvec a_1 \pvec u \pvec x \pvec x \pvec w}
                                \proves \uInter{B}{\pvec b \pvec u \pvec x \pvec x}{\pvec w}                            
                            \using[\frac{\pvec x}{\pvec x_0}, \frac{\pvec x}{\pvec x_1}]
                        \]
                        \justifies
                        \bang \Wit(\pvec u, \pvec x, \pvec w),
                        \uInter{\Gamma}{\pvec u}{\pvec \gamma \pvec u \pvec x \pvec x \pvec w}, 
                        \uInter{\bang A}{\pvec x}{\pvec a_0  \pvec u \pvec x \pvec x \pvec w}, 
                        \uInter{\bang A}{\pvec x}{\pvec a_1 \pvec u \pvec x \pvec x \pvec w}
                            \proves \uInter{B}{\pvec b \pvec u \pvec x \pvec x}{\pvec w}                            
                        \using[\frac{\pvec x}{\pvec x_0}, \frac{\pvec x}{\pvec x_1}]
                        \using(\textup{con})
                    \]
                    \justifies
                    \bang \Wit(\pvec u, \pvec x, \pvec w),
                    \uInter{\Gamma}{\pvec u}{\pvec \gamma \pvec u \pvec x \pvec x \pvec w}, 
                    \bang \ubq{\pvec \tau^-_A}{\pvec y}{\pvec a_0 \pvec u \pvec x \pvec x \pvec w} \uInter{A}{\pvec x}{\pvec y}, 
                    \bang \ubq{\pvec \tau^-_A}{\pvec y}{\pvec a_1 \pvec u \pvec x \pvec x \pvec w} \uInter{A}{\pvec x}{\pvec y} 
                        \proves \uInter{B}{\pvec b \pvec u \pvec x \pvec x}{\pvec w}                            
                    \using {(\textup{D}\ref{inter})}
                \]
                \justifies
                \bang \Wit(\pvec u, \pvec x, \pvec w),
                \uInter{\Gamma}{\pvec u}{\pvec \gamma \pvec u \pvec x \pvec x \pvec w}, 
                \bang \ubq{\pvec \tau^-_A}{\pvec y}{\pvec a_0 \pvec u \pvec x \pvec x \pvec w} \uInter{A}{\pvec x}{\pvec y} \otimes
                \bang \ubq{\pvec \tau^-_A}{\pvec y}{\pvec a_1 \pvec u \pvec x \pvec x \pvec w} \uInter{A}{\pvec x}{\pvec y} 
                    \proves \uInter{B}{\pvec b \pvec u \pvec x \pvec x}{\pvec w}                            
                \using {(\otimes\textup{L})}
            \]
            \justifies
            \bang \Wit(\pvec u, \pvec x, \pvec w, \pvec a_0 \pvec u \pvec x \pvec x \pvec w, \pvec a_1 \pvec u \pvec x \pvec x \pvec w),
            \uInter{\Gamma}{\pvec u}{\pvec \gamma \pvec u \pvec x \pvec x \pvec w}, 
            \bang \ubq{\pvec \tau^-_A}{\pvec y}{\join{\pvec a_0 \pvec u \pvec x \pvec x \pvec w}{\pvec a_1 \pvec u \pvec x \pvec x \pvec w}} \uInter{A}{\pvec x}{\pvec y}
                \proves \uInter{B}{\pvec b \pvec u \pvec x \pvec x}{\pvec w}                            
            \using $\ConStrength$
        \]
        \justifies
        \bang \Wit(\pvec u, \pvec x, \pvec w),
        \uInter{\Gamma}{\pvec u}{\pvec \gamma \pvec u \pvec x \pvec x \pvec w}, 
        \bang \ubq{\pvec \tau^-_A}{\pvec y}{\join{\pvec a_0 \pvec u \pvec x \pvec x \pvec w}{\pvec a_1 \pvec u \pvec x \pvec x \pvec w}} \uInter{A}{\pvec x}{\pvec y}
            \proves \uInter{B}{\pvec b \pvec u \pvec x \pvec x}{\pvec w}                            
        \using \mbox{\IH{(i)}}       
    \]
    \justifies 
    \bang \Wit(\pvec u, \pvec x, \pvec w),
    \uInter{\Gamma}{\pvec u}{(\lambda \pvec u, \pvec x, \pvec w . \pvec \gamma \pvec u \pvec x \pvec x \pvec w) \pvec u \pvec x \pvec w}, 
    \uInter{\bang A}{\pvec x}{(\lambda \pvec u, \pvec x, \pvec w . \join{\pvec a_0 \pvec u \pvec x \pvec x \pvec w}
        {\pvec a_1 \pvec u \pvec x \pvec x \pvec w}) \pvec u \pvec x \pvec w}
        \proves \uInter{B}{(\lambda \pvec u, \pvec x . \pvec b \pvec u \pvec x \pvec x) \pvec u \pvec x}{\pvec w}
    \using {(\textup{D}\ref{inter})}
\end{prooftree}
\]
(\emph{$\bang\textup{L}$}). By induction hypothesis the premise of the dereliction rule is witnessable, i.e. we have closed terms $\pvec \gamma, \pvec a, \pvec b$ such that
\begin{enumerate}
	\item[\IH{(i)}] $\proves \Wit(\pvec \gamma)$ and $\proves \Wit(\pvec a)$ and $\proves \Wit(\pvec b)$
	\item[\IH{(ii)}] $\bang \Wit(\pvec u, \pvec x, \pvec w),
                            \uInter{\Gamma}{\pvec u}{\pvec \gamma \pvec u \pvec x \pvec w}, 
                            \uInter{A}{\pvec x}{\pvec a  \pvec u \pvec x \pvec w}
                                \proves \uInter{B}{\pvec b \pvec u \pvec x}{\pvec w}$
\end{enumerate}
We claim that the terms $\pvec \gamma$ and $\lambda \pvec u, \pvec x, \pvec w . \singleton{(\pvec a \pvec u \pvec x \pvec w)}$ and $\pvec b$ witness the conclusion of the contraction rule:
\[
\begin{prooftree}
    \[
        \[
            \[
                \justifies
                \bang \Wit(\pvec u, \pvec x, \pvec w),
                \uInter{\Gamma}{\pvec u}{\pvec \gamma \pvec u \pvec x \pvec w}, 
                \uInter{A}{\pvec x}{\pvec a  \pvec u \pvec x \pvec w}
                    \proves \uInter{B}{\pvec b \pvec u \pvec x}{\pvec w}
                \using \mbox{\IH{(ii)}}
            \]
            \justifies
            \bang \Wit(\pvec u, \pvec x, \pvec w, \pvec a  \pvec u \pvec x \pvec w),
            \uInter{\Gamma}{\pvec u}{\pvec \gamma \pvec u \pvec x \pvec w},
            \bang \ubq{\pvec \tau^-_A}{\pvec y}{\singleton{(\pvec a  \pvec u \pvec x \pvec w)}} \uInter{A}{\pvec x}{\pvec y}
                \proves \uInter{B}{\pvec b \pvec u \pvec x}{\pvec w}
            \using $\ConUnit$
        \]
        \justifies
        \bang \Wit(\pvec u, \pvec x, \pvec w),
        \uInter{\Gamma}{\pvec u}{\pvec \gamma \pvec u \pvec x \pvec w}, 
        \bang \ubq{\pvec \tau^-_A}{\pvec y}{\singleton{(\pvec a  \pvec u \pvec x \pvec w)}} \uInter{A}{\pvec x}{\pvec y}
            \proves \uInter{B}{\pvec b \pvec u \pvec x}{\pvec w}
        \using \using \mbox{\IH{(i)}}
    \]
    \justifies
    \bang \Wit(\pvec u, \pvec x, \pvec w),
    \uInter{\Gamma}{\pvec u}{\pvec \gamma \pvec u \pvec x \pvec w}, 
    \uInter{\bang A}{\pvec x}{(\lambda \pvec u, \pvec x, \pvec w . \singleton{(\pvec a \pvec u \pvec x \pvec w)}) \pvec u \pvec x \pvec w}
        \proves \uInter{B}{\pvec b \pvec u \pvec x}{\pvec w}
    \using {(\textup{D}\ref{inter})}
\end{prooftree} 
\]
(\emph{$\bang\textup{R}$}). By induction hypothesis the premise of the promotion rule is witnessable, i.e. we have closed terms $\pvec \gamma, \pvec a$ such that
\begin{enumerate}
	\item[\IH{(i)}] $\proves \Wit(\pvec \gamma)$ and $\proves \Wit(\pvec a)$
	\item[\IH{(ii)}] $\bang \Wit(\pvec u, \pvec y'),
                            \uInter{\bang \Gamma}{\pvec u}{\pvec \gamma \pvec u \pvec y'} 
                                \proves \uInter{A}{\pvec a \pvec u}{\pvec y'}$
\end{enumerate}
We claim that the terms $\lambda \pvec u, \pvec y . \comp{(\pvec \gamma \pvec u)}{\pvec y}$ and $\pvec a$ witness the conclusion of the promotion rule. If $\Gamma = B_1, \ldots, B_n$, we write $\ubq{\pvec \tau^-_\Gamma}{\pvec w}{\pvec \gamma \pvec u \pvec y'} \uInter{\Gamma}{\pvec u}{\pvec w}$ as an abbreviation for the tuple $\ubq{\pvec \tau^-_{B_1}}{\pvec w_1}{\pvec \gamma_1 \pvec u \pvec y'} \uInter{B_1}{\pvec u_1}{\pvec w_1}, \ldots, \ubq{\pvec \tau^-_{B_n}}{\pvec w_n}{\pvec \gamma_n \pvec u \pvec y'} \uInter{B_n}{\pvec u_n}{\pvec w_n}$, assuming each $B_i$ has interpretation $\uInter{B_i}{\pvec u_i}{\pvec w_i}$, :
\[
\begin{prooftree}
    \[
        \[
            \[
                \[
                    \[
                        \[
                            \[
                                \[
                                    \justifies
                                    \bang \Wit(\pvec u, \pvec y'), \uInter{\bang \Gamma}{\pvec u}{\pvec \gamma \pvec u \pvec y'}
                                        \proves \uInter{A}{\pvec a \pvec u}{\pvec y'}
                                    \using \mbox{\IH{(ii)}}
                                \]
                                \justifies
                                \bang \Wit(\pvec u, \pvec y'), 
                                \bang \ubq{\pvec \tau^-_\Gamma}{\pvec w}{\pvec \gamma \pvec u \pvec y'} \uInter{\Gamma}{\pvec u}{\pvec w} 
                                    \proves \uInter{A}{\pvec a \pvec u}{\pvec y'}
                                \using {(\textup{D}\ref{inter})}
                            \]
                            \justifies
                            \bang \Wit(\pvec u), 
                            \ubq{\pvec \tau^-_A}{\pvec y'}{\pvec y} \bang \Wit(\pvec y'),
                            \ubq{\pvec \tau^-_A}{\pvec y'}{\pvec y} \bang \ubq{\pvec \tau^-_\Gamma}{\pvec w'}{\pvec \gamma \pvec u \pvec y'} \uInter{\Gamma}{\pvec u}{\pvec w'}
                                \proves \ubq{\pvec \tau^-_A}{\pvec y'}{\pvec y} \uInter{A}{\pvec a \pvec u}{\pvec y'}
                            \using \mbox{\Quant{1}}
                        \]
                        \justifies
                        \bang \Wit(\pvec u), 
                        \bang \ubq{\pvec \tau^-_A}{\pvec y'}{\pvec y} \bang \Wit(\pvec y'),
                        \bang \ubq{\pvec \tau^-_A}{\pvec y'}{\pvec y} 
                        \bang \ubq{\pvec \tau^-_\Gamma}{\pvec w'}{\pvec \gamma \pvec u \pvec y'} \uInter{\Gamma}{\pvec u}{\pvec w'}
                            \proves \ubq{\pvec \tau^-_A}{\pvec y'}{\pvec y} \uInter{A}{\pvec a \pvec u}{\pvec y'}
                        \using (\bang\textup{L})
                    \]
                    \justifies
                    \bang \Wit(\pvec u), 
                    \bang \ubq{\pvec \tau^-_A}{\pvec y'}{\pvec y} \bang \Wit(\pvec y'),
                    \bang \ubq{\pvec \tau^-_A}{\pvec y'}{\pvec y} 
                    \bang \ubq{\pvec \tau^-_\Gamma}{\pvec w'}{\pvec \gamma \pvec u \pvec y'} \uInter{\Gamma}{\pvec u}{\pvec w'}
                        \proves~\bang \ubq{\pvec \tau^-_A}{\pvec y'}{\pvec y} \uInter{A}{\pvec a \pvec u}{\pvec y'}
                    \using (\bang\textup{R})
                \]
                \justifies
                \bang \Wit(\pvec u), 
                \bang \ubq{\pvec \tau^-_A}{\pvec y'}{\pvec y} 
                \bang \ubq{\pvec \tau^-_\Gamma}{\pvec w'}{\pvec \gamma \pvec u \pvec y'} \uInter{\Gamma}{\pvec u}{\pvec w'}
                    \proves \,\bang \ubq{\pvec \tau^-_A}{\pvec y'}{\pvec y} \uInter{A}{\pvec a \pvec u}{\pvec y'}
                \using \mbox{\Quant{2}}
            \]
            \justifies
            \bang \Wit(\pvec u, \pvec y, \pvec \gamma \pvec u), 
            \bang \ubq{\pvec \tau^-_\Gamma}{\pvec w'}{\comp{(\pvec \gamma \pvec u)}{\pvec y}} \uInter{\Gamma}{\pvec u}{\pvec w'} 
            \proves \, \bang \ubq{\pvec \tau^-_A}{\pvec y'}{\pvec y} \uInter{A}{\pvec a \pvec u}{\pvec y'}
            \using $\ConApp$
        \]
        \justifies
        \bang \Wit(\pvec u, \pvec y), 
        \bang \ubq{\pvec \tau^-_\Gamma}{\pvec w'}{\comp{(\pvec \gamma \pvec u)}{\pvec y}} \uInter{\Gamma}{\pvec u}{\pvec w'} 
            \proves \, \bang \ubq{\pvec \tau^-_A}{\pvec y'}{\pvec y} \uInter{A}{\pvec a \pvec u}{\pvec y'}
        \using \mbox{\IH{(i)}}
    \]
    \justifies 
    \bang \Wit(\pvec u, \pvec y), 
    \uInter{\bang \Gamma}{\pvec u}{\comp{(\pvec \gamma \pvec u)}{\pvec y}}
        \proves \uInter{\bang A}{\pvec a \pvec u}{\pvec y}
    \using {(\textup{D}\ref{inter})}
\end{prooftree}
\]
That concludes the proof. 
\end{proof}

\subsection{Parametrised interpretation of $\AL^\BB$ (interpreting disjunction)}

The parametrised interpretation of $\AL$-theories presented above deals uniformly with the multiplicative connectives ($A \otimes B$ and $A \lto B$), the modality ($\bang A$) and the quantifiers ($\exists z A$ and $\forall z A$), but it leaves anything else to the non-logical axioms of the theory. In particular, the reader will notice that we have not committed to any particular interpretation of the additive connectives ($A \oplus B$ and $A \& B$). As described in Subsection \ref{Subsection_theories}, we aim to capture these via the theory of booleans $\AL^\BB$, where these additive connectives are definable. 

Hence, in this section we show how, with some extra assumptions on the interpretation of the predicate symbol $\BB(x)$, the parametrised interpretation of $\AL$ can be extended to a parametrised interpretation of $\AL^\BB$, and hence, a parametrised interpretation of disjunction. 

Let $\Asource$ be an extension of $\AL^\BB$ which is obtained as the Girard translation of an $\IL^\BB$-theory, and $\Atarget$ be an extension of $\AL^\omega$ such that assumptions \Assumption{1} -- \Assumption{5} hold. In this paper we will always assume that the equality predicate is non-computational, so that $\uInter{s = t}{\varepsilon}{\varepsilon} \pdefin s = t$. We then consider the following extra assumption on the interpretation $\bound{z}{b}{\BB}$ of the boolean predicate $\BB$:

\begin{enumerate}
    \item[\Assumption{6}] Assume that for each predicate symbol $P$ of $\Asource$ we have a term in $\Atarget$ 
    \[ \ifW{P} \colon \wtype{\BB} \to \wtype{P} \to \wtype{P} \to \wtype{P} \]
    such that 
    \begin{itemize}
    	\item[$(i)$] $\proves_{\Atarget} \Wit_{\wtype{P} \to \wtype{P} \to \wtype{\BB} \to \wtype{P}}(\ifW{P})$, and 
	\item[$(ii)$] $\bang \Wit(x_1, x_2), \bang (\bound{z}{b}{\BB}), \bang (\bound{\pvec x}{{\rm if}(z, x_1, x_2)}{P}) \proves_{\Atarget} \bound{\pvec x}{\ifW{P}(b, x_1, x_2)}{P}$. 
    \end{itemize}
    We also assume that for some terms $\tilde{\true}$ and $\tilde{\false}$ of $\Atarget$ we have $\proves_{\Atarget} \bound{\true}{\tilde{\true}}{\BB}$ and $\proves_{\Atarget} \bound{\false}{\tilde{\false}}{\BB}$.
\end{enumerate}

Our assumption that $\Asource$ is the Girard translation of an $\IL^\BB$-theory is crucial for the following definition and lemma, where we only quantify over \emph{intuitionistic} formulas:

\begin{definition} For each \emph{intuitionistic} formula $A$, define the tuple of terms
\[  \ifW{A} \colon \wtype{\BB} \to  \pvec \tau^+_{\lTrans{A}} \to \pvec \tau^+_{\lTrans{A}} \to \pvec \tau^+_{\lTrans{A}} \]
by induction on $A$, as follows:
\[
\begin{array}{lcl}
	\ifW{A \wedge B}(b, \pvec x_1, \pvec v_1, \pvec x_2, \pvec v_2) & = & 
		\ifW{A}(b, \pvec x_1, \pvec x_2), \ifW{B}(b, \pvec v_1, \pvec v_2) \\[1mm]
	\ifW{A \to B}(b, \pvec f_1, \pvec g_1, \pvec f_2, \pvec g_2) & = & 
		\lambda \pvec x^{\pvec \tau^+_{\lTrans{A}}} . \ifW{B}(b, \pvec f_1 \pvec x, \pvec f_2 \pvec x),
		\lambda \pvec x^{\pvec \tau^+_{\lTrans{A}}} 
		\lambda \pvec w^{\pvec \tau^-_{\lTrans{B}}} . \join{\pvec g_1 \pvec x \pvec w}{\pvec g_2 \pvec x \pvec w} \\[1mm]
	\ifW{\exists z A}(b, \pvec x_1, \pvec x_2) & = & \ifW{A}(b, \pvec x_1, \pvec x_2) \\[1mm]
	\ifW{\forall z A}(b, \pvec x_1, \pvec x_2) & = & \ifW{A}(b, \pvec x_1, \pvec x_2)
\end{array}
\]
where, for predicate symbols $P$, we let $\ifW{P}$ be the term assumed to exist in \Assumption{6}.
\end{definition}

\begin{lemma}[Monotonicity lemma for $\BB$] \label{monotonicity-lemma} Under assumption \Assumption{6}, for each \emph{intuitionistic} formula $A$, we have:
\begin{enumerate}
	\item[(i)] $\proves_{\Atarget} \Wit_{\tau^+_{\lTrans{A}} \to \tau^+_{\lTrans{A}} \to \wtype{\BB} \to \pvec \tau^+_{\lTrans{A}}}(\ifW{A})$
	\item[(ii)] $\bang \Wit(\pvec x_1, \pvec x_2, \pvec y), \bang (\bound{z}{b}{\BB}), \bang \uInter{\lTrans{A}}{{\rm if}(z, \pvec x_1, \pvec x_2)}{\pvec y} \proves_{\Atarget} \uInter{\lTrans{A}}{\ifW{A}(b, \pvec x_1, \pvec x_2)}{\pvec y}$
\end{enumerate}
\end{lemma}

\begin{proof} By induction on the complexity of the formula $A$. $(i)$ follows directly from assumptions \Assumption{6} $(i)$ and \ConStrength, and Lemma \ref{lem-W-closure}. $(ii)$ The case where $A$ is a predicate symbol also follows directly from \Assumption{6} $(ii)$. The cases of conjunction, existential and universal quantifiers are easy to verify. Let us focus on the case of implication. Assume $\Wit(\pvec f_1, \pvec g_1, \pvec f_2, \pvec g_2, \pvec x, \pvec w)$ and $\bang (\bound{z}{b}{\BB})$ (and hence $\BB(z)$). Noting that
\[ 
(*) \quad \BB(z),
	\ubq{\pvec \tau^-_A}{\pvec y}{\pvec g_1 \pvec x \pvec w} \uInter{\lTrans{A}}{\pvec x}{\pvec y},
	\ubq{\pvec \tau^-_A}{\pvec y}{\pvec g_1 \pvec x \pvec w} \uInter{\lTrans{A}}{\pvec x}{\pvec y}
	\proves \ubq{\pvec \tau^-_A}{\pvec y}{{\rm if}(z, \pvec g_1 \pvec x \pvec w, \pvec g_2 \pvec x \pvec w)} \uInter{\lTrans{A}}{\pvec x}{\pvec y}
\]
we have
\[
\begin{array}{rcl}
    &  & \uInter{\lTrans{(A \to B)}}{{\rm if}(z, \pvec f_1, \pvec g_1, \pvec f_2, \pvec g_2)}{\pvec x, \pvec w} \\[1mm]
    & \equiv & \bang \ubq{\pvec \tau^-_A}{\pvec y}{{\rm if}(z, \pvec g_1 \pvec x \pvec w, \pvec g_2 \pvec x \pvec w)} \uInter{\lTrans{A}}{\pvec x}{\pvec y} 
    		     \lto \uInter{\lTrans{B}}{{\rm if}(z, \pvec f_1 \pvec x, \pvec f_2 \pvec x)}{\pvec w}\\[1mm]
    & \stackrel{(*)}{\Rightarrow} &
    	\bang \ubq{\pvec \tau^-_A}{\pvec y}{\pvec g_1 \pvec x \pvec w} \uInter{\lTrans{A}}{\pvec x}{\pvec y}
	~\otimes~
	\bang \ubq{\pvec \tau^-_A}{\pvec y}{\pvec g_2 \pvec x \pvec w} \uInter{\lTrans{A}}{\pvec x}{\pvec y} 
    		     \lto \uInter{\lTrans{B}}{{\rm if}(z, \pvec f_1 \pvec x, \pvec f_2 \pvec x)}{\pvec w} \\[1mm]
    & \stackrel{\textup{\ConStrength}}{\Rightarrow} &
    	\bang \ubq{\pvec \tau^-_A}{\pvec y}{\join{\pvec g_1 \pvec x \pvec w}{\pvec g_2 \pvec x \pvec w}} \uInter{\lTrans{A}}{\pvec x}{\pvec y} 
    		     \lto \uInter{\lTrans{B}}{{\rm if}(z, \pvec f_1 \pvec x, \pvec f_2 \pvec x)}{\pvec w} \\[1mm]
    & \stackrel{(\textup{IH})}{\Rightarrow} &
    	\bang \ubq{\pvec \tau^-_A}{\pvec y}{\join{\pvec g_1 \pvec x \pvec w}{\pvec g_2 \pvec x \pvec w}} \uInter{\lTrans{A}}{\pvec x}{\pvec y} 
    		     \lto \uInter{\lTrans{B}}{\ifW{B}(b, \pvec f_1 \pvec x, \pvec f_2 x)}{\pvec w} \\[2mm]
    & \equiv &
    	\uInter{\lTrans{(A \to B)}}{\ifW{A \to B}(b, \pvec f_1, \pvec g_1, \pvec f_2, \pvec g_2)}{\pvec x, \pvec w}
\end{array}
\]
since the assumptions imply $\Wit(\pvec f_1 \pvec x, \pvec f_2 \pvec x, \pvec w)$ -- needed for the (IH).
\end{proof}

Using the monotonicity lemma above (Lemma \ref{monotonicity-lemma}) we can then show that the axioms $A[\true/z], A[\false/z], \BB(z) \proves A$ of $\AL^\BB$ are interpretable when $A$ is the Girard translation of an intuitionistic formula:

\begin{theorem} Under assumption \Assumption{6}, for any intuitionistic formula $A$ the sequent $\bang \lTrans{A}[\true/z], \bang \lTrans{A}[\false/z], \bang \BB(z) \proves \lTrans{A}$ is witnessable in $\Atarget$.
\end{theorem}

\begin{proof} Starting with the following consequence of $\BB$\emph{L}
\[ \bang \Wit(\pvec x_1, \pvec x_2, \pvec y),  
   \bang \uInter{\lTrans{A}[\true/z]}{{\rm if}(\true, \pvec x_1, \pvec x_2)}{\pvec y}, 
   \bang \uInter{\lTrans{A}[\false/z]}{{\rm if}(\false, \pvec x_1, \pvec x_2)}{\pvec y}, \bang \BB(z)
    \proves_{\Atarget} \uInter{\lTrans{A}}{{\rm if}(z, \pvec x_1, \pvec x_2)}{\pvec y} \]
we can apply Lemma \ref{monotonicity-lemma} $(ii)$ to obtain
\[ \bang \Wit(\pvec x_1, \pvec x_2, \pvec y), 
  \bang \uInter{\lTrans{A}[\true/z]}{\pvec x_1}{\pvec y}, 
  \bang \uInter{\lTrans{A}[\false/z]}{\pvec x_2}{\pvec y}, 
  \bang(\bound{z}{b}{\BB}) \proves_{\Atarget} \uInter{\lTrans{A}}{\ifW{\lTrans{A}}(b, \pvec x_1, \pvec x_2)}{\pvec y} \]
and by \ConUnit, assuming $\bang \Wit(\pvec x_1, \pvec x_2, \pvec y)$, we have
\[ 
  \bang \ubq{\pvec \tau^-_{\lTrans{A}}}{\pvec y'}{\singleton(\pvec y)} \uInter{\lTrans{A}[\true/z]}{\pvec x_1}{\pvec y'}, 
  \bang \ubq{\pvec \tau^-_{\lTrans{A}}}{\pvec y'}{\singleton(\pvec y)} \uInter{\lTrans{A}[\false/z]}{\pvec x_2}{\pvec y'}, 
  \bang(\bound{z}{b}{\BB}) \proves_{\Atarget} \uInter{\lTrans{A}}{\ifW{\lTrans{A}}(b, \pvec x_1, \pvec x_2)}{\pvec y} \]
The witnessing terms are clearly in $\Wit$ (referring to Lemma \ref{monotonicity-lemma} $(i)$).
\end{proof}

That the axioms $\BB$\emph{R} are witnessable follows from the second part of \Assumption{6}, and $\proves \neg (\true = \false)$ is interpreted by itself, since we are assuming that the equality predicate is non-computational. Therefore, assumptions \Assumption{1} -- \Assumption{6} guarantee the soundness of the interpretation of $\Asource$ into $\Atarget$.

\subsection{Parametrised interpretation of $\AHAomega$ (interpreting induction)}

As we have done for the theory of booleans $\BB$ above, we can also prove a general monotonicity lemma, and a corresponding parametric interpretation, for the theory of natural numbers $\NN$ and finite types $\Type$. Hence, here we take $\Asource = \AHAomega = \lTrans{(\HAomega)}$ and $\Atarget = \NAHA^\omega = \lTrans{(\nHAomega)}$. On top of the assumptions \Assumption{1} -- \Assumption{6} we must add assumptions that ensure the soundness of the induction axiom schema, and also the typing axioms:
\begin{enumerate}
\item[\Assumption{7}] Assume that for each predicate symbol $P$ of $\Asource$ we have a term in $\Atarget$
\[ \apW{P} \colon (\NN \to \wtype{P}) \to \wtype{\NN} \to \wtype{P} \]
such that
\begin{itemize}
	\item $\proves_{\Atarget} \Wit_{\wtype{P} \to \wtype{P} \to \wtype{\NN} \to \wtype{P}}(\apW{P})$
	\item $\bang \forall n^\NN \Wit(fn), \bang (\bound{n}{a}{\NN}), \bang (\bound{\pvec x}{f n}{P}) \proves_{\Atarget} \bound{\pvec x}{\apW{P}(f)(a)}{P}$
	\item for each $a \in \wtype{\NN}$ there exists an $N_a \in \NN$ such that $\bound{n}{a}{\NN} \proves_{\Atarget} n \leq N_a$
\end{itemize}

\item[\Assumption{8}] Assume that in $\Atarget$ we have a family of terms $\tilde{\ap}$ such that
\[ \bound{f}{\tilde{f}}{\sigma \to \tau}, \bound{x}{\tilde{x}}{\sigma} \proves_{\Atarget} \bound{\ap(f, x)}{\tilde{\ap}(\tilde{f}, \tilde{x})}{\tau} \]
and that for each constant $c^\tau$ of $\Asource$ we have a term $\tilde{c}$ of $\Atarget$ such that
\begin{itemize}
	\item $\proves_{\Atarget} \bound{c}{\tilde{c}}{\tau}$
\end{itemize}
In particular, it follows that 
\begin{itemize}
	\item for each numeral $n \in \NN$ there exists a term $a_n$ such that $\proves_{\Atarget} \bound{n}{a_n}{\NN}$.
\end{itemize}
\end{enumerate}

Let us see how these assumptions imply the soundness of the non-logical axioms of $\AHAomega$.

\begin{definition} For each \emph{intuitionistic} formula $A$, define the tuple of terms
\[ 
\apW{A} \colon \pvec (\NN \to \pvec \tau^+_{\lTrans{A}}) \to \wtype{\NN} \to \pvec \tau^+_{\lTrans{A}}
\]
by induction on $A$, as follows:
\[
\begin{array}{lcl}
	\apW{A \wedge B}(\pvec f_A, \pvec f_B)(a) & = & 
		\apW{A}(\pvec f_A)(a), \apW{B}(\pvec f_B)(a) \\[1mm]
	\apW{A \to B}(\pvec f, \pvec g)(a) & = & 
		\lambda \pvec x^{\pvec \tau^+_{\lTrans{A}}} . \apW{B}(\lambda i . \pvec f i \pvec x)(a);
		\lambda \pvec x^{\pvec \tau^+_{\lTrans{A}}}, \pvec w^{\pvec \tau^-_{\lTrans{B}}} . \pvec g^{\sqcup} N_a \pvec x \pvec w \\[1mm]
	\apW{\exists x A}(\pvec f)(a) & = & \apW{A}(\pvec f)(a) \\[1mm]
	\apW{\forall x A}(\pvec f)(a) & = & \apW{A}(\pvec f)(a)
\end{array}
\]
where, for predicate symbols $P$ we let $\apW{P}$ be the term assumed to exist in \Assumption{7}, and $\pvec g^{\sqcup}$ is defined recursively as
\[
\pvec g^{\sqcup} n \pvec x \pvec w = 
\left\{
	\begin{array}{lcl}
		\pvec g 0 \pvec x \pvec w & {\rm if} \; n = 0 \\[2mm]
		\join{\pvec g n \pvec x \pvec w}{\pvec g^{\sqcup} (n-1) \pvec x \pvec w} & {\rm if} \; n > 0
	\end{array}
\right.
\]
\end{definition}

\begin{lemma}[Monotonicity lemma for $\NN$] \label{monotonicity-lemma-int} Under assumption \Assumption{7}, for each \emph{intuitionistic} formula $A$, we have:
\begin{enumerate}
	\item[$(i)$] $\proves_{\Atarget} \Wit_{(\NN \to \pvec \tau^+_{\lTrans{A}}) \to \wtype{\NN} \to \pvec \tau^+_{\lTrans{A}}}(\apW{P})$
	\item[$(ii)$] $\bang \forall n^\NN \Wit(\pvec f n), \bang \Wit(\pvec y), \bang (\bound{n}{a}{\NN}), \uInter{\lTrans{A}}{\pvec f n}{\pvec y} \proves_{\Atarget} \uInter{\lTrans{A}}{\apW{A}{(\pvec f)(a)}}{\pvec y}$
\end{enumerate}
\end{lemma}

\begin{proof} The proof is very similar to that of the monotonicity lemma for $\BB$ (Lemma \ref{monotonicity-lemma}). In the case of implication we first need to prove 
\[ (*) \quad 
	\bang \ubq{\pvec \tau^-_A}{\pvec y}{\pvec g^{\sqcup} N_a \pvec x \pvec w} \uInter{\lTrans{A}}{\pvec x}{\pvec y}, 
	\bang (\bound{n}{a}{\NN}) 
	\proves_{\Atarget} \ubq{\pvec \tau^-_A}{\pvec y}{\pvec g n \pvec x \pvec w} \uInter{\lTrans{A}}{\pvec x}{\pvec y}
\]
which we can do by first proving, using induction on $k$ and then taking $k = N_a$,  
\[ \bang \ubq{\pvec \tau^-_A}{\pvec y}{\pvec g^{\sqcup} k \pvec x \pvec w} \uInter{\lTrans{A}}{\pvec x}{\pvec y}, \bang (n \leq k) 
	\proves_{\Atarget} \ubq{\pvec \tau^-_A}{\pvec y}{\pvec g n \pvec x \pvec w} \uInter{\lTrans{A}}{\pvec x}{\pvec y}
\]
The above also uses applications of \ConStrength~for the induction step. Hence, assuming $\bang \forall n^\NN \Wit(\pvec f n, \pvec g n)$ and $\bang \Wit(\pvec x, \pvec w)$ and $\bang (\bound{n}{a}{\NN})$ we have
\[
\begin{array}{rcl}
    \uInter{\lTrans{(A \to B)}}{\pvec f n, \pvec g n}{\pvec x, \pvec w} 
    & \equiv & \bang \ubq{\pvec \tau^-_A}{\pvec y}{\pvec g n \pvec x \pvec w} \uInter{\lTrans{A}}{\pvec x}{\pvec y} 
    		     \lto \uInter{\lTrans{B}}{\pvec f n \pvec x}{\pvec w}\\[1mm]
    & \stackrel{(*)}{\Rightarrow} &
    	\bang \ubq{\pvec \tau^-_A}{\pvec y}{\pvec g^{\sqcup} N_a \pvec x \pvec w} \uInter{\lTrans{A}}{\pvec x}{\pvec y} 
    		     \lto \uInter{\lTrans{B}}{(\lambda n . \pvec f n \pvec x)(n)}{\pvec w} \\[1mm]
    & \stackrel{(\textup{IH})}{\Rightarrow} &
    	\bang \ubq{\pvec \tau^-_A}{\pvec y}{\pvec g^{\sqcup} N_a \pvec x \pvec w} \uInter{\lTrans{A}}{\pvec x}{\pvec y} 
    		     \lto \uInter{\lTrans{B}}{\apW{B}(\lambda n . \pvec f n \pvec x)(a)}{\pvec w} \\[2mm]
    & \equiv &
    	\uInter{\lTrans{(A \to B)}}{\apW{A \to B}(\pvec f, \pvec g)(a)}{\pvec x, \pvec w}
\end{array}
\]
since the assumptions imply $\bang \forall n^\NN \Wit(\pvec f n \pvec x)$ -- needed for the (IH). 
\end{proof}

\begin{proposition}\label{lTrans:induction} Under assumption \Assumption{7}, the induction rule
\[ 
\begin{prooftree}
    \proves \lTrans{A}(0)
    \quad
    \bang \lTrans{A}(n), \bang \NN(n) \proves \lTrans{A}(n+1)
    \justifies
    \bang \NN(n) \proves \lTrans{A}(n)
\end{prooftree}
\]
is witnessable in $\Atarget$.
\end{proposition}

\begin{proof} We must show that if the two premises are witnessable then the conclusion is also witnessable. Let $\pvec s, \pvec r, \pvec t$ in $\Wit$ be realisers for the premises:
\begin{enumerate}
	\item[$(i)$] $\bang \Wit(\pvec y) \proves_{\Atarget} \uInter{\lTrans{A}(0)}{\pvec s}{\pvec y}$
	\item[$(ii)$] $\bang \Wit(\pvec x, \pvec y),
	  \bang \ubq{\tau^-_{\lTrans{A}}}{\pvec y'}{\pvec r a \pvec x \pvec y} \uInter{\lTrans{A}(n)}{\pvec x}{\pvec y'}, 
	  \bang (\bound{n}{a}{\NN}) 
	  \proves_{\Atarget}
	  \uInter{\lTrans{A}(n+1)}{\pvec t a \pvec x}{\pvec y}$
\end{enumerate}
From the assumption that for each $n \in \NN$ there exists an $a_n \in \wtype{\NN}$ such that $\bound{n}{a_n}{\NN}$, we have
\begin{enumerate}
	\item[$(iii)$] $\bang \Wit(\pvec x, \pvec y),
	  \bang \ubq{\tau^-_{\lTrans{A}}}{\pvec y'}{\pvec r a_n \pvec x \pvec y} \uInter{\lTrans{A}(n)}{\pvec x}{\pvec y'}, 
	  \bang \NN(n)
	  \proves_{\Atarget}
	  \uInter{\lTrans{A}(n+1)}{\pvec t a_n \pvec x}{\pvec y}$
\end{enumerate}
which implies
\begin{enumerate}
	\item[$(iv)$] $\bang \Wit(\pvec x),
	  \bang \forall \pvec y^{\Wit{}} \uInter{\lTrans{A}(n)}{\pvec x}{\pvec y}, 
	  \bang \NN(n)
	  \proves_{\Atarget} \,
	  \bang \forall \pvec y^{\Wit{}} \uInter{\lTrans{A}(n+1)}{\pvec t a_n \pvec x}{\pvec y}$
\end{enumerate}
Let $\pvec f n$ be defined by primitive recursion on $n$ as 
\[
\pvec f n = 
\left \{
	\begin{array}{lcl}
		\pvec s & {\rm if} \; n = 0 \\[2mm]
		\pvec t a_n (\pvec f (n-1)) & {\rm if} \; n > 0
	\end{array}
\right.
\]
Since $\Wit(\pvec s, \pvec t, a_n)$, it follows that $\bang \forall n^\NN \Wit(\pvec f n)$. From $(i)$ and $(iv)$, by induction on $n$ we have
\begin{enumerate}
	\item[$(v)$] 
	$\bang \Wit(\pvec y),
	  \bang \NN(n)
	  \proves_{\Atarget}
	  \uInter{\lTrans{A}(n)}{\pvec f n}{\pvec y}$
\end{enumerate}
and, by Lemma \ref{monotonicity-lemma-int}, $\bang \Wit(\pvec y), \bang (\bound{n}{a}{\NN}) \proves_{\Atarget} \uInter{\lTrans{A}(n)}{\apW{A}(\pvec f) (a)}{\pvec y}$.
\end{proof}

Using \Assumption{8} we can also guarantee the soundness of the typing axioms of $\AHAomega$. The remaining non-logical axioms of $\AHAomega$ only involve equations on natural numbers, and are easily seen to be witnessable. It follows that assumptions \Assumption{1} -- \Assumption{8} guarantee the soundness of the interpretation of $\AHAomega$ into $\NAHA^\omega$.

\section{Parametrised Interpretations of $\IL$}
\label{IL-interpretations}

We now describe how the parametrised interpretation of $\AL$-theories gives rise to two parametrised interpretations of $\IL$-theories. 

\begin{definition}[$\AL$-parameters from $\IL$-parameters] \label{def-parameters-translation} In the context of $\IL$-theories,  when we refer to assumptions \Assumption{1} -- \Assumption{8}, we mean the \emph{forgetful translation} of original $\AL$-assumptions, by considering the following translation of $\IL$-parameters:
\[
\begin{array}{lcl}
	\lWit_\tau(x) & \pdefin & \lTrans{(\Wit_\tau(x))} \\[2mm]
	\lbound{\pvec x}{a}{P} & \pdefin & \lTrans{(\bound{\pvec x}{a}{P})} \\[2mm]
	\lubq{\pvec \tau}{\pvec x}{\pvec a}{A} & \pdefin & \lTrans{(\ubq{\pvec \tau}{\pvec x}{\pvec a} \fTrans{A})} \\[2mm]
\end{array}
\]
\end{definition}

For instance, \Assumption{1} would say that when interpreting an $\IL$-theory $\Isource$ into an $\IL$-theory $\Itarget$, the theory $\Itarget$ must be an extension of $\IL^\omega$. What is important is that from the assumptions that $\Isource$ and $\Itarget$ satisfy\footnote{By Proposition \ref{prop:lb-equivalence} it follows that $\bTrans{\Isource}$ and $\bTrans{\Itarget}$ also satisfy the $\AL$-version of \Assumption{1} -- \Assumption{8} for the parameters
\[
\begin{array}{lcl}
	\Wit_\tau^\circ(x) & \pdefin & \bTrans{(\Wit_\tau(x))} \\[2mm]
	\bbound{\pvec x}{a}{P} & \pdefin & \bTrans{(\bound{\pvec x}{a}{P})} \\[2mm]
	\bubq{\pvec \tau}{\pvec x}{\pvec a}{A} & \pdefin & \bTrans{(\ubq{\pvec \tau}{\pvec x}{\pvec a} \fTrans{A})} \\[2mm]
\end{array}	
\]
} the $\IL$-version of \Assumption{1} -- \Assumption{8} for the parameters $\{\bound{\pvec x}{a}{P}\}_{P \in \Pred{\Isource}^c}$, $\{\Wit_\tau(x)\}_{\tau \in \Type}$ and $\{\ubq{\pvec \tau}{\pvec x}{\pvec a} A\}_{\pvec \tau \in \Type; A \in \Formulas{\Itarget}}$, it must follow that $\lTrans{\Isource}$ and $\lTrans{\Itarget}$ satisfy the $\AL$-version of \Assumption{1} -- \Assumption{8} for the parameters $\{\lbound{\pvec x}{a}{P}\}_{P \in \Pred{\lTrans{\Isource}}^c}$, $\{\lWit_\tau(x)\}_{\tau \in \Type}$ and $\{\lubq{\pvec \tau}{\pvec x}{\pvec a} A\}_{\pvec \tau \in \Type; A \in \Formulas{\Itarget}}$.

\begin{figure}
\[
\begin{tikzcd}[row sep=huge, column sep = huge]
\Isource \arrow[r, "\lInter{\cdot}{\pvec x}{\pvec y}; \; \bInter{\cdot}{\pvec x}{\pvec y}"] \arrow[d,  "\lTrans{(\cdot)};\, \bTrans{(\cdot)}"]
& \Itarget \\
\lTrans{\Isource} \simeq \bTrans{\Isource} \arrow[r, "\uInter{\,\cdot\,}{\pvec x}{\pvec y}"]
& \lTrans{\Itarget} \simeq \bTrans{\Itarget} \arrow[u, "\fTrans{(\cdot)}" ] 
\end{tikzcd}
\]
\caption{Parametrised interpretations of $\Isource$ into $\Itarget$} 
\label{Diagram_interpret}
\end{figure}

We will normally omit the type parameter $\tau$ in the formulas above when this can be easily inferred from the context. For an $\IL$-theory, we write $\forall \bound{\pvec x}{a}{P} \, A$ and $\exists \bound{\pvec x}{a}{P} \, A$ as abbreviations for $\forall \pvec x (\bound{\pvec x}{a}{P} \to A)$ and $\exists \pvec x (\bound{\pvec x}{a}{P} \wedge A)$, respectively.

\begin{definition}[$\IL$-interpretations] \label{def:combined-inter} Given the choice of $\IL$ parameters in $\Itarget$, consider the $\lTrans{(\cdot)}$-translation of these parameters in the $\AL$-theory $\lTrans{\Itarget}$ with the corresponding $\AL$-interpretation $A \mapsto \uInter{A}{\pvec x}{\pvec y}$ of $\lTrans{\Isource}$ into $\lTrans{\Itarget}$. This gives rise to the following $\IL$-interpretation of $\Isource$ into $\Itarget$
\[  A \quad \mapsto \quad \fTrans{(\uInter{\lTrans{A}}{\pvec x}{\pvec y})}  \]
which we will abbreviate as $\lInter{A}{\pvec x}{\pvec y} \,\pdefin\, \fTrans{(\uInter{\lTrans{A}}{\pvec x}{\pvec y})}$. Similarly, consider the $\bTrans{(\cdot)}$-translation of the  parameters in $\bTrans{\Itarget}$ with the corresponding $\AL$-interpretation $A \mapsto \uInter{A}{\pvec x}{\pvec y}$ of $\bTrans{\Isource}$ into $\bTrans{\Itarget}$. This gives rise to the following $\IL$-interpretation of $\Isource$ into $\Itarget$
\[  A \quad \mapsto \quad \fTrans{(\uInter{\bTrans{A}}{\pvec x}{\pvec y})}  \]
which we will abbreviate as $\bInter{A}{\pvec x}{\pvec y} \,\pdefin\, \fTrans{(\uInter{\bTrans{A}}{\pvec x}{\pvec y})}$.
\end{definition}

The relation between the parametrised interpretation of $\AL$-theories, and the two parametrised interpretations of $\IL$-theories is summarised in Figure \ref{Diagram_interpret}.

\begin{proposition} \label{prop-lInter} The following equivalences are provable in $\Itarget$
\[
\begin{array}{lcl}
\lInter{P(\pvec x)}{a}{} & \Leftrightarrow &  \bound{\pvec x}{a}{P} \quad \mbox{if $P \in \Pred{\Isource}^{c}$} \\[2mm]
\lInter{P(\pvec x)}{}{} & \Leftrightarrow & P(\pvec x) \quad \mbox{if $P \in \Pred{\Isource}^{nc}$} \\[2mm]
\lInter{A \to B}{\pvec f, \pvec g}{\pvec x, \pvec w} & \Leftrightarrow & \ubq{\pvec \tau^-_{\lTrans{A}}}{\pvec y}{\pvec g \pvec x \pvec w} \lInter{A}{\pvec x}{\pvec y} \to \lInter{B}{\pvec f \pvec x}{\pvec w} \\[2mm]
\lInter{A \wedge B}{\pvec x, \pvec v}{\pvec y, \pvec w} & \Leftrightarrow & \lInter{A}{\pvec x}{\pvec y} \wedge \lInter{B}{\pvec v}{\pvec w} \\[2mm]
\lInter{A \vee B}{b, \pvec x, \pvec v}{\pvec y, \pvec w} & \Leftrightarrow & \exists \bound{z}{b}{\BB} (\pcond{z}{\ubq{\pvec \tau^-_{\lTrans{A}}}{\pvec y'}{\pvec y} \lInter{A}{\pvec x}{\pvec y'}}{\ubq{\pvec \tau^-_{\lTrans{B}}}{\pvec w'}{\pvec w}\lInter{B}{\pvec v}{\pvec w'}}) \\[2mm]
\lInter{\exists z A}{\pvec x}{\pvec y} & \Leftrightarrow & \exists z \ubq{\pvec \tau^-_{\lTrans{A}}}{\pvec y'}{\pvec y} \lInter{A}{\pvec x}{\pvec y'} \\[2mm]
\lInter{\forall z A}{\pvec x}{\pvec y} & \Leftrightarrow & \forall z \lInter{A}{\pvec x}{\pvec y} 
\end{array}
\]
In particular, we have that for computational predicate symbols $P$:
\eqleft{
\begin{array}{lcl}
\lInter{\exists z^P A}{c, \pvec x}{\pvec y} & \Leftrightarrow & \exists \bound{z}{c}{P} \ubq{\pvec \tau^-_{\lTrans{A}}}{\pvec y'}{\pvec y} \lInter{A}{\pvec x}{\pvec y'} \\[2mm]
\lInter{\forall z^P A}{\pvec f}{b, \pvec y} & \Leftrightarrow & \forall \bound{z}{b}{P} \, \lInter{A}{\pvec f b}{\pvec y}
\end{array}
}
\end{proposition}

\begin{proof} See Appendix A. \end{proof}

\begin{proposition} \label{prop-bInter} The following equivalences are provable in $\Itarget$
\eqleft{
\begin{array}{lcl}
\bInter{P(\pvec x)}{a}{} & \Leftrightarrow &  \bound{\pvec x}{a}{P} \quad \mbox{if $P \in \Pred{\Isource}^{c}$} \\[2mm]
\bInter{P(\pvec x)}{}{} & \Leftrightarrow & P(\pvec x) \quad \mbox{if $P \in \Pred{\Isource}^{nc}$} \\[2mm]
\bInter{A \to B}{\pvec f, \pvec g}{\pvec x, \pvec w} & \Leftrightarrow & \ubq{\pvec \tau^+_{\bTrans{A}}, \pvec \tau^-_{\bTrans{B}}}{\pvec x', \pvec w'}{\pvec x, \pvec w} (\bInter{A}{\pvec x'}{\pvec g \pvec x' \pvec w'} \to \bInter{B}{\pvec f \pvec x'}{\pvec w'}) \\[2mm]
\bInter{A \wedge B}{\pvec x, \pvec v}{\pvec y, \pvec w} & \Leftrightarrow & \bInter{A}{\pvec x}{\pvec y} \wedge \bInter{B}{\pvec v}{\pvec w} \\[2mm]
\bInter{A \vee B}{\pvec x, \pvec v, b}{\pvec y, \pvec w} & \Leftrightarrow & \exists \bound{z}{b}{\BB} (\pcond{z}{\bInter{A}{\pvec x}{\pvec y}}{\bInter{B}{\pvec v}{\pvec w}}) \\[2mm]
\bInter{\exists z A}{\pvec x}{\pvec y} & \Leftrightarrow & \exists z \bInter{A}{\pvec x}{\pvec y} \\[2mm]
\bInter{\forall z A}{\pvec x}{\pvec y} & \Leftrightarrow & \ubq{\pvec \tau^-_{\bTrans{A}}}{\pvec y'}{\pvec y} \forall z \bInter{A}{\pvec x}{\pvec y'} 
\end{array}
}
In particular, we have that for computational predicate symbols $P$
\eqleft{
\begin{array}{lcl}
\bInter{\exists z^P A}{\pvec x, c}{\pvec y} & \Leftrightarrow & \exists \bound{z}{c}{P} \; \bInter{A}{\pvec x}{\pvec y} \\[2mm]
\bInter{\forall z^P A}{\pvec f}{c, \pvec y} & \Leftrightarrow & \ubq{\wtype{P}, \btype{\pvec \tau}^-_{\bTrans{A}}}{c', \pvec y'}{c, \pvec y} \ubq{\pvec \tau^-_{\bTrans{A}}}{c'', \pvec y''}{c', \pvec y'} \forall \bound{z}{c''}{P} \, \bInter{A}{\pvec f c''}{\pvec y''}   
\end{array}
}
\end{proposition}

\begin{proof} See Appendix B. \end{proof}

\subsection{Soundness of the $\lInter{\cdot}{}{}$-interpretation}

Let us now see how we can derive the soundness for the $\IL$-interpretation $\lInter{\cdot}{}{}$ from the soundness of the corresponding $\AL$-interpretation $\uInter{\cdot}{}{}$. 

\begin{definition}[$\lInter{\cdot}{}{}$-witnessable $\IL$-sequents] A sequent $\Gamma \proves A$ of $\Isource$ is said to be $\lInter{\cdot}{}{}$-\emph{witnessable} in $\Itarget$ if there are closed terms $\pvec \gamma, \pvec a$ of $\Itarget$ such that 
\begin{enumerate}
	\item[(i)] $\proves_{\Itarget} \Wit_{\pvec \tau^+_{\lTrans{\Gamma}} \to \pvec \tau^-_{\lTrans{A}} \to \wtype{\tau^-_{\lTrans{\Gamma}}}}(\pvec \gamma)$ and $\proves_{\Itarget} \Wit_{\pvec \tau^+_{\lTrans{\Gamma}} \to \pvec \tau^+_{\lTrans{A}}}(\pvec a)$, and
	\item[(ii)] $\Wit_{\pvec \tau^+_{\lTrans{\Gamma}}, \pvec \tau^-_{\lTrans{A}}}(\pvec x,\pvec w), \ubq{\pvec \tau^-_{\lTrans{\Gamma}}}{\pvec y}{\pvec \gamma \pvec x \pvec w} \lInter{\Gamma}{\pvec x}{\pvec y} \proves_{\Itarget} \lInter{A}{\pvec a \pvec x}{\pvec w}$.
\end{enumerate}
\end{definition}

\begin{definition}[Sound $\IL$-interpretation] An $\IL$-interpretation of $\Isource$ into $\Itarget$ is said to be \emph{sound} if the provable sequents of $\Isource$ are witnessable in $\Itarget$.
\end{definition} 

\begin{lemma} \label{trans-forget-trans} $\lTrans{(\fTrans{(\uInter{\lTrans{A}}{\pvec x}{\pvec y})})} \equiv \uInter{\lTrans{A}}{\pvec x}{\pvec y}$.
\end{lemma}

\begin{proof} By a simple induction on the structure of $A$.
\end{proof}

\begin{lemma} \label{c-wit-wit} If $\Gamma \proves A$ is $\lInter{\cdot}{}{}$-witnessable in $\Itarget$, then $\bang \lTrans{\Gamma} \proves \lTrans{A}$ is witnessable in $\lTrans{\Itarget}$.
\end{lemma}

\begin{proof} By assumption we have closed terms $\pvec \gamma, \pvec a$ of $\Itarget$ such that $\proves_{\Itarget} \Wit_{\pvec \tau^+_{\lTrans{\Gamma}} \to \pvec \tau^-_{\lTrans{A}} \to \wtype{\tau^-_{\lTrans{\Gamma}}}}(\pvec \gamma)$ and $\proves_{\Itarget} \Wit_{\pvec \tau^+_{\lTrans{\Gamma}} \to \pvec \tau^+_{\lTrans{A}}}(\pvec a)$ and 
\[ 
\Wit_{\pvec \tau^+_{\lTrans{\Gamma}}, \pvec \tau^-_{\lTrans{A}}}(\pvec x,\pvec w),
\ubq{\pvec \tau^-_{\lTrans{\Gamma}}}{\pvec y}{\pvec \gamma \pvec x \pvec w} \fTrans{(\uInter{\lTrans{\Gamma}}{\pvec x}{\pvec y})} 
	\proves_{\Itarget} \fTrans{(\uInter{\lTrans{A}}{\pvec a \pvec x}{\pvec w})}. \]
By Proposition \ref{ltrans-prop} we have 
\[ 
\lWit_{\pvec \tau^+_{\lTrans{\Gamma}}, \pvec \tau^-_{\lTrans{A}}}(\pvec x,\pvec w),
\bang \lubq{\pvec \tau^-_{\lTrans{\Gamma}}}{\pvec y}{\pvec \gamma \pvec x \pvec w} \uInter{\lTrans{\Gamma}}{\pvec x}{\pvec y} 
	\proves_{\lTrans{\Itarget}} \lTrans{(\fTrans{(\uInter{\lTrans{A}}{\pvec a \pvec x}{\pvec w})})} .\]
By Lemma \ref{trans-forget-trans} this implies
\[
\lWit_{\pvec \tau^+_{\lTrans{\Gamma}}, \pvec \tau^-_{\lTrans{A}}}(\pvec x,\pvec w),
\bang \lubq{\pvec \tau^-_{\lTrans{\Gamma}}}{\pvec y}{\pvec \gamma \pvec x \pvec w} \uInter{\lTrans{\Gamma}}{\pvec x}{\pvec y} 
\proves_{\lTrans{\Itarget}} \uInter{\lTrans{A}}{\pvec a \pvec x}{\pvec w}
\]
Hence $\bang \lTrans{\Gamma} \proves \lTrans{A}$ is witnessable in $\lTrans{\Itarget}$.
\end{proof}

\begin{theorem}[Soundness of $\lInter{\cdot}{}{}$-interpretation] \label{soundness2} Assume a fixed choice of $\IL$-parameters $\{\bound{\pvec x}{a}{P}\}_{P \in \Pred{\Isource}^c}$, $\{\Wit_\tau(x)\}_{\tau \in \Type}$ and $\{\ubq{\pvec \tau}{\pvec x}{\pvec a} A\}_{\pvec \tau \in \Type; A \in \Formulas{\Itarget}}$ in $\Itarget$. If
\begin{enumerate}
	\item[$(i)$] the corresponding $\AL$-parameters, $\{\lbound{\pvec x}{a}{P}\}_{P \in \Pred{\lTrans{\Isource}}^c}$, $\{\lWit_\tau(x)\}_{\tau \in \Type}$ and $\{\lubq{\pvec \tau}{\pvec x}{\pvec a} A\}_{\pvec \tau \in \Type; A \in \Formulas{\lTrans{\Itarget}}}$, are an adequate choice for the formulas $\uInter{A}{\pvec x}{\pvec y}$, for all $A$ in $\lTrans{\Isource}$, and
	\item[$(ii)$] the non-logical axioms of $\Isource$ are $\lInter{\cdot}{}{}$-witnessable in $\Itarget$, 
\end{enumerate}
then the $\lInter{\cdot}{}{}$-interpretation of $\Isource$ into $\Itarget$ is sound.
\end{theorem}

\begin{proof} The second assumption and Lemma \ref{c-wit-wit} imply that all the non-logical axioms of $\lTrans{\Isource}$ are witnessable in $\lTrans{\Itarget}$. Therefore, given the $\AL$-parameters $\{\lbound{\pvec x}{a}{P}\}_{P \in \Pred{\lTrans{\Isource}}^c}$, $\{\lWit_\tau(x)\}_{\tau \in \Type}$ and $\{\lubq{\pvec \tau}{\pvec x}{\pvec a} A\}_{\pvec \tau \in \Type; A \in \Formulas{\lTrans{\Itarget}}}$, by the first assumption and the Soundness Theorem (Theorem~\ref{soundness1}) we have that this instance of the parametrised interpretation $\uInter{\cdot}{}{}$ is a sound interpretation of $\lTrans{\Isource}$ into $\lTrans{\Itarget}$. Now, assume $\Gamma \proves_{\Isource} A$. By Proposition \ref{ltrans-prop} we have $\bang \lTrans{\Gamma} \proves_{\lTrans{\Isource}} \lTrans{A}$. By the soundness of $A \mapsto \uInter{A}{\pvec x}{\pvec y}$ we have closed terms $\pvec \gamma, \pvec a$ of $\lTrans{\Itarget}$ such that 
\begin{enumerate}
	\item $\proves_{\lTrans{\Itarget}} \lWit_{\pvec \tau^+_{\lTrans{\Gamma}} \to \pvec \tau^-_{\lTrans{A}} \to \wtype{\tau^-_{\lTrans{\Gamma}}}}(\pvec \gamma)$ and $\proves_{\lTrans{\Itarget}} \lWit_{\pvec \tau^+_{\lTrans{\Gamma}} \to \pvec \tau^+_{\lTrans{A}}}(\pvec a)$
	\item $\bang \lWit_{\pvec \tau^+_{\lTrans{\Gamma}}, \pvec \tau^-_{\lTrans{A}}}(\pvec x,\pvec w), 
		\bang \lubq{\pvec \tau^-_{\lTrans{\Gamma}}}{\pvec y}{\pvec \gamma \pvec x \pvec w} \uInter{\lTrans{\Gamma}}{\pvec x}{\pvec y} 
                    	\proves_{\lTrans{\Itarget}} \uInter{\lTrans{A}}{\pvec a \pvec x}{\pvec w}$.
\end{enumerate}
Hence, by the forgetful translation (Definition \ref{forget}), we have
\begin{enumerate}
	\item $\proves_{\Itarget} \Wit_{\pvec \tau^+_{\lTrans{\Gamma}} \to \pvec \tau^-_{\lTrans{A}} \to \wtype{\tau^-_{\lTrans{\Gamma}}}}(\pvec \gamma)$ and $\proves_{\Itarget} \Wit_{\pvec \tau^+_{\lTrans{\Gamma}} \to \pvec \tau^+_{\lTrans{A}}}(\pvec a)$
	\item $\Wit_{\pvec \tau^+_{\lTrans{\Gamma}}, \pvec \tau^-_{\lTrans{A}}}(\pvec x,\pvec w), 
		\ubq{\pvec \tau^-_{\lTrans{\Gamma}}}{\pvec y}{\pvec \gamma \pvec x \pvec w} \lInter{\Gamma}{\pvec x}{\pvec y} 
                    	\proves_{\Itarget} \lInter{A}{\pvec a \pvec x}{\pvec w}$.
\end{enumerate}
\end{proof}

\subsection{Soundness of the $\bInter{\cdot}{}{}$-interpretation}

In a similar manner we can derive the soundness for the $\bInter{\cdot}{}{}$-interpretation from the soundness of the $\AL$-interpretation $\uInter{\cdot}{}{}$. 

\begin{definition}[$\bInter{\cdot}{}{}$-witnessable $\IL$-sequents] A sequent $\Gamma \proves A$ of $\Isource$ is said to be $\bInter{\cdot}{}{}$-\emph{witnessable} in $\Itarget$ if there are closed terms $\pvec \gamma, \pvec a$ of $\Itarget$ such that 
\begin{enumerate}
	\item[(i)] $\proves_{\Itarget} \Wit_{\pvec \tau^+_{\bTrans{\Gamma}} \to \pvec \tau^-_{\bTrans{A}} \to  \tau^-_{\bTrans{\Gamma}}}(\pvec \gamma)$ and $\proves_{\Itarget} \Wit_{\pvec \tau^+_{\bTrans{\Gamma}} \to \pvec \tau^+_{\bTrans{A}}}(\pvec a)$, and
	\item[(ii)] $\Wit_{\pvec \tau^+_{\bTrans{\Gamma}}, \pvec \tau^-_{\bTrans{A}}}(\pvec x,\pvec w), \bInter{\Gamma}{\pvec x}{\pvec \gamma \pvec x \pvec w} \proves_{\Itarget} \bInter{A}{\pvec a \pvec x}{\pvec w}$.
\end{enumerate}
\end{definition}

\begin{lemma} \label{btrans-forget-btrans} $\bTrans{(\fTrans{(\uInter{\bTrans{A}}{\pvec x}{\pvec y})})} \equiv \uInter{\bTrans{A}}{\pvec x}{\pvec y}$.
\end{lemma}

\begin{proof} By a simple induction on the structure of $A$.
\end{proof}

\begin{lemma} \label{o-wit-wit} If $\Gamma \proves A$ is $\bInter{\cdot}{}{}$-witnessable in $\Itarget$, then $\bTrans{\Gamma} \proves_{\bTrans{\Itarget}} \bTrans{A}$ is witnessable in $\bTrans{\Itarget}$.
\end{lemma}

\begin{proof}  The proof is similar to the proof of Lemma~\ref{c-wit-wit}.
\end{proof}

\begin{theorem}[Soundness of the $\bInter{\cdot}{}{}$-interpretation] \label{soundness3} Assume a fixed choice of $\IL$-parameters $\{\bound{\pvec x}{a}{P}\}_{P \in \Pred{\Isource}^c}$, $\{\Wit_\tau(x)\}_{\tau \in \Type}$ and $\{\ubq{\pvec \tau}{\pvec x}{\pvec a} A\}_{\pvec \tau \in \Type; A \in \Formulas{\Itarget}}$ in $\Itarget$. If
\begin{enumerate}
	\item[$(i)$] the corresponding $\AL$-parameters, $\{\bbound{\pvec x}{a}{P}\}_{P \in \Pred{\lTrans{\Isource}}^c}$, $\{\Wit^\circ_\tau(x)\}_{\tau \in \Type}$ and $\{\bubq{\pvec \tau}{\pvec x}{\pvec a} A\}_{\pvec \tau \in \Type; A \in \Formulas{\lTrans{\Itarget}}}$, are an adequate choice for the formulas $\uInter{A}{\pvec x}{\pvec y}$, for all $A$ in $\bTrans{\Isource}$, and
	\item[$(ii)$] all the non-logical axioms of $\Isource$ are $\bInter{\cdot}{}{}$-witnessable in $\Itarget$, 
\end{enumerate}
then the $\bInter{\cdot}{}{}$-interpretation of $\Isource$ into $\Itarget$ is sound.
\end{theorem}

\begin{proof} Similar to the proof of Theorem~\ref{soundness2}.
\end{proof}

\subsection{Comparing the interpretations $\bInter{\cdot}{}{}$ and $\lInter{\cdot}{}{}$ }

Given that $\bang \lTrans{A}$ is equivalent to $\bTrans{A}$ over $\AL$ (cf. Proposition \ref{prop:lb-equivalence}), one should expect that the interpretations $\bInter{\cdot}{}{}$ and $\lInter{\cdot}{}{}$ are also, in some sense, equivalent. In order to prove such relation it seems that we need to place the following extra assumptions on the parameter $\ubq{\pvec \tau}{\pvec y}{\pvec a} A$ in order to ensure that it behaves as a bounded universal quantifier:
\begin{enumerate}
    \item[\Quant{3}] $\ubq{\pvec \tau}{\pvec x}{\pvec a}{(A \cwedge B(\pvec x))} \Leftrightarrow_{\Atarget} A \cwedge \ubq{\pvec \tau}{\pvec x}{\pvec a}{B(\pvec x)}$, if $\pvec x \not \in \FV{(A)}$
    \item[\Quant{4}] $\forall z \ubq{\pvec \tau}{\pvec x}{\pvec a}{A} \Leftrightarrow_{\Atarget} \ubq{\pvec \tau}{\pvec x}{\pvec a}{\forall z A}$
    \item[\Quant{5}] $\exists z \ubq{\pvec \tau}{\pvec x}{\pvec a}{A} \Rightarrow_{\Atarget} \ubq{\pvec \tau}{\pvec x}{\pvec a} \exists z A$
\end{enumerate}

We then obtain a relationship between $\lInter{A}{\pvec x}{\pvec y}$ and $\bInter{A}{\pvec x}{\pvec y}$ as follows:

\begin{theorem} \label{compare-interpretations} For each formula $A$ there are tuples of closed terms $\pvec s_1, \pvec t_1$ and $\pvec s_2, \pvec t_2$  such that
\begin{enumerate}
	\item[$(i)$] $\Wit_{\pvec \tau^+_{\lTrans{A}}, \pvec \tau^-_{\bTrans{A}}}(\pvec x, \pvec y), \ubq{\pvec \tau^-_{\lTrans{A}}}{\pvec y'}{\pvec s_1 \pvec x \pvec y} \lInter{A}{\pvec x}{\pvec y'} \proves_{\IL^\omega} \bInter{A}{\pvec t_1 \pvec x}{\pvec y}$
	\item[$(ii)$] $\Wit_{\pvec \tau^+_{\bTrans{A}}, \pvec \tau^-_{\lTrans{A}}}(\pvec x, \pvec y), \bInter{A}{\pvec x}{\pvec s_2 \pvec x \pvec y} \proves_{\IL^\omega} \ubq{\pvec \tau^-_{\lTrans{A}}}{\pvec y'}{\pvec y} \lInter{A}{\pvec t_2 \pvec x}{\pvec y'}$
	\item[$(iii)$] $\proves_{\IL^\omega} \Wit(\pvec s_1) \wedge \Wit(\pvec s_2) \wedge \Wit(\pvec t_1) \wedge \Wit(\pvec t_2)$, of appropriate types.
\end{enumerate}
\end{theorem}

\begin{proof} See Appendix C. \end{proof}

It follows from Theorem~\ref{compare-interpretations} that the interpretations $\lInter{A}{\pvec x}{\pvec y}$ and $\bInter{A}{\pvec x}{\pvec y}$ are in fact two different presentations of the same interpretation, in the sense that these two interpretations must necessarily have the same set of characterising principles, and in particular will validate the same set of principles.

\section{Concrete Interpretations of $\HAomega$}
\label{sec-instances}

Let us conclude by considering several instances of the parametrised $\IL$ interpretations $\lInter{\cdot}{}{}$ and $\bInter{\cdot}{}{}$ (Definition \ref{def:combined-inter}). By Theorems \ref{soundness2} and \ref{soundness3}, in order to prove the soundness of the derived $\IL$ interpretation of $\Isource$ into $\Itarget$, it is enough to check that the choice of parameters is adequate for the formulas in the image of the interpretation, and that the non-logical axioms of $\Isource$ are witnessable in $\Itarget$. 

For simplicity, for all instantiations considered here we always take the source theory to be $\Isource = \HAomega$ (equality only for type level $0$, quantifier-free formulas decidable) and the target theory to be $\Itarget = \nHAomega$ (equality available for all types). This means that the predicate symbols we must consider are typing assertions $\rho(x)$, for each finite type $\rho$, and we will consider all of them as computational, i.e. in $\Pred{\Isource}^c$. In this case $\lTrans{\Isource} = \AHAomega$ and $\lTrans{\Itarget} = \NAHA^\omega$. That deals with assumptions \Assumption{1} and \Assumption{2}.

We consider three groups of instantiations, depending on the choice of the parameter $\bound{x}{a}{\tau}$, which we will take to be either $x =_\tau a$ (equality), $x \in_\tau a$ (set inclusion) or $x \leq^*_\tau a$ (majorizability). In each of these cases, and for the corresponding instances of $\Wit$ that we will consider, it should be straightforward to verify that assumptions \Assumption{6}, \Assumption{7} and \Assumption{8} are satisfied (disjunction, induction and finite types). For instance, in the case of majorizability, when $\Wit_\tau(x) = x \leq^*_\tau x$, assumption \Assumption{7} becomes $\forall n^\NN (fn \leq^*_\tau fn), n \leq a, x \leq^*_\tau f n \proves x \leq^*_\tau \apW{P}(f)(a)$ which is satisfied for $\apW{P}(f)(a) = \max_{n \leq a} f n$; and \Assumption{8} follows directly from the definition of majorizability, taking $\tilde{\ap} = \ap$. Note that self-majorizability $x \leq^*_{\tau^*} x$ for the sequence type $\tau^*$ boils down to every element of the sequence $x$ being self-majorizing, i.e. $\forall i < |x| (x_i \leq^*_\tau x_i)$.

\subsection{Interpretations where $\bound{x}{a}{\tau} \; \pdefin \; x =_\tau a$}

The instances where $\bound{x}{a}{\tau}$ is chosen to be $\tau(x) \wedge (x = a)$, with $\wtype{\tau} = \tau$ and $m^{\BB}_\tau(f)(a) = m^{\NN}_\tau(f)(a) = f(a)$, which we call \emph{precise interpretations}, include the seminal interpretations such as G\"odel's Dialectica interpretation, its Diller-Nahm variant, and Kreisel's modified realizability. In these cases the soundness of the interpretation is already known, so we will simply show in detail how the parameters are instantiated to obtain these interpretations, without duelling too much on their soundness. 

\paragraph {\bf Modified realizability interpretation} Consider the following instantiation of the parameters:
\[
\begin{array}{ccccccc}
\bound{x}{a}{\tau} & \wtype{\tau} & \Wit_\tau(x) & \ubq{\pvec \tau}{\pvec x}{\varepsilon} A & \btype{\pvec \tau} & \ifW{\tau} & \apW{\tau} \\
\hline
\tau(x) \wedge (x = a) & \tau & \rm{true} & \forall \pvec x^{\pvec \tau} \, A  & \varepsilon & {\rm if}_{\tau} & \lambda f . f
\end{array}
\]
where $\varepsilon$ denotes the empty tuple of terms or types. Since $\bTrans{A}$ is equivalent to $\bang \lTrans{A}$, we see that $\bTrans{A}$ has interpretation $\bInter{A}{\pvec x}{\varepsilon}$, and we normally omit the empty tuple symbol. Using Proposition \ref{prop-bInter}, we then see that this instantiation leads to the following interpretation:

\begin{proposition}[Kreisel's modified realizability of $\HAomega$] With the parameters instantiated as above we have:
\eqleft{
\begin{array}{lcl}
\bInter{A \to B}{\pvec f}{} & \Leftrightarrow & \forall \pvec x (\bInter{A}{\pvec x}{} \to \bInter{B}{\pvec f \pvec x}{}) \\[2mm]
\bInter{A \wedge B}{\pvec x, \pvec v}{} & \Leftrightarrow & \bInter{A}{\pvec x}{} \wedge \bInter{B}{\pvec v}{} \\[2mm]
\bInter{A \vee B}{\pvec x, \pvec v, b}{} & \Leftrightarrow & \BB(b) \wedge \pcond{b}{\bInter{A}{\pvec x}{}}{\bInter{B}{\pvec v}{}} \\[2mm]
\bInter{\exists z^\tau A}{\pvec x, c}{} & \Leftrightarrow & \tau(c) \wedge \bInter{A[c/z]}{\pvec x}{} \\[2mm]
\bInter{\forall z^\tau A}{\pvec f}{} & \Leftrightarrow & \forall z^\tau \bInter{A}{\pvec f z}{}
\end{array}
}
so that $\bInter{A}{\pvec x}{}$ can be seen to correspond to $\pvec x \mr A$.
\end{proposition}

\begin{proof} Direct from Proposition \ref{prop-bInter}.
\end{proof}

\paragraph {\bf Dialectica interpretation} Consider this instantiation of the parameters:
\[
\begin{array}{ccccccc}
\bound{x}{a}{\tau} & \wtype{\tau} & \Wit_\tau(x) & \ubq{\pvec \tau}{\pvec x}{\pvec a} A & \btype{\pvec \tau} & \ifW{\tau} & \apW{\tau} \\
\hline
\tau(x) \wedge (x = a) & \tau & \rm{true} & A[\pvec a / \pvec x]  & \pvec \tau & {\rm if}_{\tau} & \lambda f . f
\end{array}
\]

\begin{proposition}[G\"odel's Dialectica interpretation of $\HAomega$] With the parameters instantiated as above we have:
\eqleft{
\begin{array}{lcl}
\lInter{A \to B}{\pvec f, \pvec g}{\pvec x, \pvec w} & \Leftrightarrow & \lInter{A}{\pvec x}{\pvec g \pvec x \pvec w} \to \lInter{B}{\pvec f \pvec x}{\pvec w} \\[2mm]
\lInter{A \wedge B}{\pvec x, \pvec v}{\pvec y, \pvec w} & \Leftrightarrow & \lInter{A}{\pvec x}{\pvec y} \wedge \lInter{B}{\pvec v}{\pvec w} \\[2mm]
\lInter{A \vee B}{b, \pvec x, \pvec v}{\pvec y, \pvec w} & \Leftrightarrow & \BB(b) \wedge \pcond{b}{\lInter{A}{\pvec x}{\pvec y}}{\lInter{B}{\pvec v}{\pvec w}} \\[2mm]
\lInter{\exists z^\tau A}{c, \pvec x}{\pvec y} & \Leftrightarrow & \tau(c) \wedge \lInter{A[c/z]}{\pvec x}{\pvec y} \\[2mm]
\lInter{\forall z^\tau A}{\pvec f}{\pvec y, b} & \Leftrightarrow & \tau(b) \to \lInter{A[b/z]}{\pvec f b}{\pvec y}
\end{array}
}
so that $\lInter{A}{\pvec x}{\pvec y}$ can be seen to correspond to $A_D(\pvec x; \pvec y)$.
\end{proposition}

\begin{proof} Direct from Proposition \ref{prop-lInter}, using the above instantiation of the parameters. We are using here the equivalences
\[ \exists x^\tau (x = t \wedge A(x)) \Leftrightarrow \tau(t) \wedge A(t) \quad \mbox{and} \quad \forall x^\tau (x = t \to A(x)) \Leftrightarrow \tau(t) \to A(t) \]
which are valid in $\nHAomega$, in order to remove equality on higher-types. In this way, since our source theory is $\HAomega$, we also have that for all formulas $A$ of $\HAomega$, the formula $\lInter{A}{\pvec x}{\pvec y}$ will be decidable. This property is essential for satisfying condition \ConStrength, where $\join{\pvec y_1}{\pvec y_2}$ is defined via a case distinction involving the formula $\lInter{A}{\pvec x}{\pvec y}$. This is also the only place where we make use of the assumption that the choice of parameters only needs to be adequate for the formulas in the image of the interpretation.
\end{proof}

\paragraph {\bf Diller-Nahm interpretation} Consider this instantiation of the parameters:
\[
\begin{array}{ccccccc}
\bound{x}{a}{\tau} & \wtype{\tau} & \Wit_\tau(x) & \ubq{\pvec \tau}{\pvec x}{\pvec a} A & \btype{\pvec \tau} & \ifW{\tau} & \apW{\tau} \\
\hline
\tau(x) \wedge (x = a) & \tau & \rm{true} & \forall \pvec x \in_{\pvec \tau} \pvec a \, A  & \pvec \tau^* & {\rm if}_{\tau} & \lambda f . f
\end{array}
\]

\begin{proposition} \label{dn-formulation} With the parameters instantiated as above we have:
\[
\begin{array}{lcl}
\lInter{A \to B}{\pvec f, \pvec g}{\pvec x, \pvec w} & \Leftrightarrow & \forall \pvec y \in_{\pvec 	\tau^-_{\lTrans{A}}} \pvec g \pvec x \pvec w \, \lInter{A}{\pvec x}{\pvec y } \to \lInter{B}{\pvec f \pvec x}{\pvec w} \\[2mm]
\lInter{A \wedge B}{\pvec x, \pvec v}{\pvec y, \pvec w} & \Leftrightarrow & \lInter{A}{\pvec x}{\pvec y} \wedge \lInter{B}{\pvec v}{\pvec w} \\[2mm]
\lInter{A \vee B}{b, \pvec x, \pvec v}{\pvec y, \pvec w} & \Leftrightarrow & \BB(b) \wedge \pcond{b}{\forall \pvec y' \in_{\pvec 	\tau^-_{\lTrans{A}}} \pvec y \lInter{A}{\pvec x}{\pvec y'}}{\forall \pvec w' \in_{\pvec 	\tau^-_{\lTrans{B}}} \pvec w {\lInter{B}{\pvec v}{\pvec w'}}} \\[2mm]
\lInter{\exists z^\tau A}{c, \pvec x}{\pvec y} & \Leftrightarrow & \tau(c) \wedge \forall \pvec y' \in_{\pvec \tau^-_{\lTrans{A}}} \pvec y \, \lInter{A[c/z]}{\pvec x}{\pvec y'} \\[2mm]
\lInter{\forall z^\tau A}{\pvec f}{c, \pvec y} & \Leftrightarrow & \tau(c) \to \lInter{A[c/z]}{\pvec f c}{\pvec y}
\end{array}
\]
\end{proposition}

\begin{proof} Direct from Proposition \ref{prop-lInter}, using the above instantiation of the parameters.
\end{proof}

The treatment of disjunction and existential quantifier in the instance above appears to diverge from the standard Diller-Nahm interpretation, but the following proposition shows that this is in fact an equivalent way of presenting the Diller-Nahm interpretation.

\begin{proposition}[Correspondence with Diller-Nahm interpretation of $\HAomega$] The interpretation above $\lInter{A}{\pvec x}{\pvec y}$ can be seen to correspond to $\dnInter{A}{\pvec x}{\pvec y}$, in the sense that for each $A$ there are terms $\pvec s_1, \pvec t_1$ and $\pvec s_2,\pvec t_2$ such that
\begin{enumerate}
	\item[(i)] $\forall \pvec y' \in \pvec s_1 \pvec x \pvec y \lInter{A}{\pvec x}{\pvec y'} \proves \dnInter{A}{\pvec t_1 \pvec x}{\pvec y}$
 	\item[(ii)] $\forall \pvec y' \in \pvec s_2 \pvec x \pvec y \, \dnInter{A}{\pvec x}{\pvec y'} \proves \lInter{A}{\pvec t_2 \pvec x}{\pvec y}$
\end{enumerate}
\end{proposition}

\begin{proof} By induction on $A$. \\[2mm]
Existential quantifier $(i)$. Let $\pvec s_1, \pvec t_1$ be given by induction hypothesis. Then:
\eqleft{
\begin{array}{lcl}
	& & \forall \pvec  y' \in \{\pvec s_1 \pvec x \pvec y\} \lInter{\exists z^\tau A}{c, \pvec x}{\pvec y'} \\[2mm]
	& \stackrel{\mathrm{P}~\ref{dn-formulation}}{\Leftrightarrow} & \forall \pvec  y' \in \{\pvec s_1 \pvec x \pvec y\} ( \tau(c) 
			\wedge \forall \pvec y'' \in \pvec y' \, \lInter{A[c/z]}{\pvec x}{\pvec y''} ) \\[2mm]
	& \Rightarrow&  \tau(c) \wedge \forall \pvec y'' \in \pvec s_1 \pvec x \pvec y \lInter{A[c/z]}{\pvec x}{\pvec y''} \\[2mm]
	& \stackrel{(\mathrm{IH}_{(i)})}{\Rightarrow} & \tau(c) \wedge \dnInter{(A[c/z])}{\pvec t_1 \pvec x}{\pvec y} \\[2mm]
	& \stackrel{\textrm{DN def.}}{\equiv} & \dnInter{(\exists z^\tau A)}{c, \pvec t_1 \pvec x}{\pvec y} 
\end{array}
}
Existential quantifier $(ii)$. Let $\pvec s_2, \pvec t_2$ be given by induction hypothesis. Then:
\eqleft{
\begin{array}{lcl}
	& & \forall \pvec y'' \in \bigcup_{\pvec y' \in \pvec y} \pvec s_2 \pvec x \pvec y' \, \dnInter{(\exists z^\tau A)}{c, \pvec x}{\pvec y''} \\[2mm]
	& \stackrel{\textrm{DN def.}}{\equiv} & \forall \pvec y'' \in \bigcup_{\pvec y' \in \pvec y} \pvec s_2 \pvec x \pvec y' \, (\tau(c) \wedge \dnInter{A[c/z]}{\pvec x}{\pvec y''}) \\[2mm]
	& \Rightarrow & \tau(c) \wedge \forall \pvec y'' \in \bigcup_{\pvec y' \in \pvec y} \pvec s_2 \pvec x \pvec y' \, \dnInter{A[c/z]}{\pvec x}{\pvec y''} \\[2mm]
	& \Rightarrow & \tau(c) \wedge \forall \pvec y' \in \pvec y \forall \pvec y'' \in \pvec s_2 \pvec x \pvec y' \, \dnInter{A[c/z]}{\pvec x}{\pvec y''} \\[2mm]
	& \stackrel{(\mathrm{IH}_{(ii)})}{\Rightarrow} & \tau(c) \wedge \forall \pvec y' \in \pvec y \lInter{A[c/z]}{\pvec t_2 \pvec x}{\pvec y'} \\[2mm]
	& \stackrel{\mathrm{P}~\ref{dn-formulation}}{\Leftrightarrow} & \lInter{(\exists z^\tau A)}{c, \pvec t_2 \pvec x}{\pvec y''}
\end{array}
}
Implication $(i)$. Let $\pvec s_1^B, \pvec t_1^B$ and $\pvec s_2^A, \pvec t_2^A$ be given by induction hypothesis. Then:
\eqleft{
\begin{array}{lcl}
	& & \forall \pvec x' \in \{ \pvec t_2^A \pvec x \} \forall \pvec w' \in \pvec s_1^B (\pvec f \pvec x') \pvec w \lInter{A \to B}{\pvec f, \pvec g}{\pvec x', \pvec w'} \\[2mm] 
	& \Rightarrow &
		\forall \pvec w' \in \pvec s_1^B (\pvec f (\pvec t_2^A \pvec x)) \pvec w \lInter{A \to B}{\pvec f, \pvec g}{\pvec t_2^A \pvec x, \pvec w'} \\[2mm] 
	& \stackrel{\mathrm{P}~\ref{dn-formulation}}{\Leftrightarrow} & 
		\forall \pvec w' \in \pvec s_1^B (\pvec f (\pvec t_2^A \pvec x)) \pvec w (\forall \pvec y \in \pvec g (\pvec t_2^A \pvec x) \pvec w' \, \lInter{A}{\pvec t_2^A \pvec x}{\pvec y} \to \lInter{B}{\pvec f (\pvec t_2^A \pvec x)}{\pvec w'}) \\[2mm]
	& \Rightarrow & 
		\forall \pvec w' \in \pvec s_1^B (\pvec f (\pvec t_2^A \pvec x)) \pvec w \forall \pvec y \in \pvec g (\pvec t_2^A \pvec x) \pvec w' \, \lInter{A}{\pvec t_2^A \pvec x}{\pvec y} \\[2mm]
	& & \hspace{30mm} \to \forall \pvec w' \in \pvec s_1^B (\pvec f (\pvec t_2^A \pvec x)) \pvec w \lInter{B}{\pvec f (\pvec t_2^A \pvec x)}{\pvec w'} \\[2mm]
	& \stackrel{(\mathrm{IH}_{(i)}, \mathrm{IH}_{(ii)})}{\Rightarrow} & \forall \pvec w' \in \pvec s_1^B (\pvec f (\pvec t_2^A \pvec x)) \pvec w \forall \pvec y \in \pvec g (\pvec t_2^A \pvec x) \pvec w'  \forall \pvec y' \in \pvec s_2^A \pvec x \pvec y \, \dnInter{A}{\pvec x}{\pvec y'} \\[2mm]
	& & \hspace{30mm} \to \dnInter{B}{\pvec t_1^B(\pvec f (\pvec t_2^A \pvec x))}{\pvec w} \\[2mm]
	& \Rightarrow & 
		\forall \pvec y' \in \bigcup_{\pvec w' \in \pvec s_1^B (\pvec f (\pvec t_2^A \pvec x)) \pvec w} \bigcup_{\pvec y \in \pvec g (\pvec t_2^A \pvec x) \pvec w'} \pvec s_2^A \pvec x \pvec y  \, \dnInter{A}{\pvec x}{\pvec y'} \\[2mm]
	& & \hspace{30mm} \to \dnInter{B}{\pvec t_1^B(\pvec f (\pvec t_2^A \pvec x))}{\pvec w}  \\[2mm]
	& \stackrel{\textrm{DN def.}}{\equiv} & \dnInter{(A \to B)}{\pvec t_1^{A \to B}[\pvec f, \pvec g]}{\pvec x, \pvec w}
\end{array}
}
where $\pvec t_1^{A \to B}[\pvec f, \pvec g] \equiv \lambda \pvec x, \pvec w . \bigcup_{\pvec w' \in \pvec s_1^B (\pvec f (\pvec t_2^A \pvec x)) \pvec w} \bigcup_{\pvec y \in \pvec g (\pvec t_2^A \pvec x) \pvec w'} \pvec s_2^A \pvec x \pvec y, \lambda \pvec x . \pvec t_1^B(\pvec f (\pvec t_2^A \pvec x))$. \\[2mm]
Implication $(ii)$. Let $\pvec s_1^A, \pvec t_1^A$ and $\pvec s_2^B, \pvec t_2^B$ be given by induction hypothesis. Then:
\eqleft{
\begin{array}{lcl}
	& & \forall \pvec x' \in \{ \pvec t_1^A \pvec x \} \forall \pvec w' \in \pvec s_2^B (\pvec f \pvec x') \pvec w \dnInter{(A \to B)}{\pvec f, \pvec g}{\pvec x', \pvec w'} \\[2mm] 
	& \Rightarrow &
		\forall \pvec w' \in \pvec s_2^B (\pvec f (\pvec t_1^A \pvec x)) \pvec w \dnInter{(A \to B)}{\pvec f, \pvec g}{\pvec t_1^A \pvec x, \pvec w'} \\[2mm] 
	& \stackrel{\textrm{DN def.}}{\equiv} & 
		\forall \pvec w' \in \pvec s_2^B (\pvec f (\pvec t_1^A \pvec x)) \pvec w (\forall \pvec y \in \pvec g (\pvec t_1^A \pvec x) \pvec w' \, \dnInter{A}{\pvec t_1^A \pvec x}{\pvec y} \\[2mm]
	& & \hspace{25mm} \to \dnInter{B}{\pvec f (\pvec t_1^A \pvec x)}{\pvec w'}) \\[2mm]
	& \Rightarrow & 
		\forall \pvec w' \in \pvec s_2^B (\pvec f (\pvec t_1^A \pvec x)) \pvec w \forall \pvec y \in \pvec g (\pvec t_1^A \pvec x) \pvec w' \, \dnInter{A}{\pvec t_1^A \pvec x}{\pvec y} \\[2mm]
	& & \hspace{25mm} \to \forall \pvec w' \in \pvec s_2^B (\pvec f (\pvec t_1^A \pvec x)) \pvec w \dnInter{B}{\pvec f (\pvec t_1^A \pvec x)}{\pvec w'} \\[2mm]
	& \stackrel{(\mathrm{IH}_{(i)}, \mathrm{IH}_{(ii)})}{\Rightarrow} & \forall \pvec w' \in \pvec s_2^B (\pvec f (\pvec t_1^A \pvec x)) \pvec w \forall \pvec y \in \pvec g (\pvec t_1^A \pvec x) \pvec w'  \forall \pvec y' \in \pvec s_1^A \pvec x \pvec y \, \lInter{A}{\pvec x}{\pvec y'} \\[2mm]
	& & \hspace{30mm} \to \lInter{B}{\pvec t_2^B(\pvec f (\pvec t_1^A \pvec x))}{\pvec w} \\[2mm]
	& \Rightarrow & 
		\forall \pvec y' \in \bigcup_{\pvec w' \in \pvec s_2^B (\pvec f (\pvec t_1^A \pvec x)) \pvec w} \bigcup_{\pvec y \in \pvec g (\pvec t_1^A \pvec x) \pvec w'} \pvec s_1^A \pvec x \pvec y  \, \lInter{A}{\pvec x}{\pvec y'} \to \lInter{B}{\pvec t_2^B(\pvec f (\pvec t_1^A \pvec x))}{\pvec w} \\[2mm]
	& \stackrel{\mathrm{P}~\ref{dn-formulation}}{\Leftrightarrow} & \lInter{A \to B}{\lambda \pvec x, \pvec w . \bigcup_{\pvec w' \in \pvec s_2^B (\pvec f (\pvec t_1^A \pvec x)) \pvec w} \bigcup_{\pvec y \in \pvec g (\pvec t_1^A \pvec x) \pvec w'} \pvec s_1^A \pvec x \pvec y, \lambda \pvec x . \pvec t_2^B(\pvec f (\pvec t_1^A \pvec x))}{\pvec x, \pvec w} 
\end{array}
}
\end{proof}

\begin{remark}[Stein's family of interpretations] In \cite{Stein(79)} Stein describes a family of interpretations parametrised by a number $n$. The idea is that when $\rho$ is a type of type level $\geq n$ we treat the contraction in a way similar to the Diller-Nahm interpretation and so $\btype{\rho}=\rho^*$ and $\ubq{\rho}{x}{a} A \pdefin \forall x \in_{\rho} \textrm{range}(a) \, A$, where $a$ is a function from the pure type $(n-1)$. But when $\rho$ is a type of type level $< n$ we treat it as modified realizability and therefore $\btype{\rho} = \varepsilon$ and $\ubq{\rho}{x}{\varepsilon} A \pdefin \forall x^\rho A$. Although we could consider combinations of this with the various interpretations of quantifiers, we will leave this for future work.
\end{remark}

\begin{remark}[Diller-Nahm with majorizability] One could also consider the following choice of parameters
\[
\begin{array}{ccccccc}
\bound{x}{a}{\tau} & \wtype{\tau} & \Wit_\tau(x) & \ubq{\pvec \tau}{\pvec x}{\pvec a} A & \btype{\pvec \tau} & \ifW{\tau} & \apW{\tau} \\
\hline
\tau(x) \wedge (x = a) & \tau & \rm{true} & \forall \pvec x \leq^*_{\pvec \tau} \pvec a \, A  & \pvec \tau & {\rm if}_{\tau} & \lambda f . f
\end{array}
\]
which corresponds to a version of the Diller-Nahm interpretation where set inclusion is replaced by majorizability. Unfortunately this does not seem to lead to a sound interpretation, and indeed we cannot satisfy condition \ConUnit, as there is no term $\singleton{}$ which satisfies:
\[ \forall \pvec y \leq^*_{\pvec \tau} \singleton{(\pvec z)} A[\pvec y] \proves A[\pvec z] \]
in $\Itarget = \nHAomega$ for an arbitrary $\pvec z$, since this would imply $\pvec z \leq^*_{\pvec \tau} \singleton{(\pvec z)}$. One could then try to take $\{\Wit_\tau(x)\}_{\tau \in \Type}$ to be ``$x$ is monotone of type $\tau$'', i.e. $W_\tau(x) = x \leq^*_\tau x$, but in this case the axiom \[ \proves \rec_{\rho} \in \rho \to (\rho \to \NN \to \rho) \to \NN \to \rho \] is no longer witnessable, since $\rec$ is not self-majorizing (monotone).
\end{remark}

\subsection{Interpretations where $\bound{x}{a}{\tau} \; \pdefin \; x \leq^*_\tau \! a$}

The instances where $\bound{x}{a}{\tau}$ is chosen to be $x \leq^*_\tau a$, with $\wtype{\tau} = \tau$, which we call \emph{bounded interpretations}, include the bounded functional interpretation \cite{FerreiraOliva2005}, and the bounded modified realizability \cite{FerreiraNunes2006}. In this case we will also discuss a new interpretation: the bounded Diller-Nahm interpretation. Let $\mforall x^\tau A$ be an abbreviation for $\forall x^\tau (x \leq_\tau^* x \to A)$.

\paragraph {\bf Bounded modified realizability} Consider this instantiation of the parameters:
\[
\begin{array}{ccccccc}
\bound{x}{a}{\tau} & \wtype{\tau} & \Wit_\tau(x) & \ubq{\pvec \tau}{\pvec x}{\varepsilon} A & \btype{\pvec \tau} & \ifW{\tau}(b, x, y) & \apW{\tau} \\
\hline\\[-12pt]
\tau(x) \wedge (x \leq^*_\tau a) & \tau & x \leq^*_\tau x & \mforall \pvec x^{\pvec \tau} A  & \varepsilon & \max_{\tau}(x, y) & \lambda f . f
\end{array}
\]

Again we see that in $\bInter{A}{\pvec x}{\pvec y}$ the tuple $\pvec y$ will be empty and the types $\btype{\pvec \tau}$ can be omitted. 

\begin{proposition}[Bounded modified realizability, \cite{FerreiraNunes2006}] With the parameters instantiated as above we have:
\eqleft{
\begin{array}{lcl}
\bInter{A \to B}{\pvec f}{} & \Leftrightarrow & \mforall \pvec x^{\tau^+_{\bTrans{A}}}  (\bInter{A}{\pvec x}{} \to \bInter{B}{\pvec f \pvec x}{}) \\[1mm]
\bInter{A \wedge B}{\pvec x, \pvec v}{} & \Leftrightarrow & \bInter{A}{\pvec x}{} \wedge \bInter{B}{\pvec v}{} \\[1mm]
\bInter{A \vee B}{\pvec x, \pvec v}{} & \Leftrightarrow & \bInter{A}{\pvec x}{} \vee \bInter{B}{\pvec v}{} \\[1mm]
\bInter{\exists z^\tau A}{\pvec x, c}{} & \Leftrightarrow & \exists z^\tau \! \leq^*_\tau\! c \, \bInter{A}{\pvec x}{} \\[1mm]
\bInter{\forall z^\tau A}{\pvec f}{} & \Leftrightarrow & \forall z^\tau \! \leq^*_\tau\! b \, \bInter{A}{\pvec f b}{}
\end{array}
}
so that $\bInter{A}{\pvec x}{}$ can be seen to correspond to $\pvec x \bmr A$.
\end{proposition}

\begin{proof} Direct from Proposition \ref{prop-bInter}, using the above instantiation of the parameters.
\end{proof}

\paragraph {\bf Bounded functional interpretation} Consider this instantiation of the parameters:
\[
\begin{array}{ccccccc}
\bound{x}{a}{\tau} & \wtype{\tau} & \Wit_\tau(x) & \ubq{\pvec \tau}{\pvec x}{\pvec a} A & \btype{\pvec \tau} & \ifW{\tau}(b, x, y) & \apW{\tau} \\
\hline\\[-12pt]
\tau(x) \wedge (x \leq^*_\tau a) & \tau & x \leq^*_\tau x & \mforall \pvec x \leq^*_{\pvec \tau} \pvec a \, A  & \pvec \tau & \max_{\tau}(x, y) & \lambda f . f
\end{array}
\]

\begin{proposition}[Bounded functional interpretation, \cite{FerreiraOliva2005}] With the parameters instantiated as above we have:
\eqleft{
    \begin{array}{lcl}
        \lInter{A \to B}{\pvec f, \pvec g}{\pvec x, \pvec w} & \Leftrightarrow & \mforall \pvec y \leq^*_{\pvec \tau^-_{\lTrans{A}}} \! \pvec g \pvec x \pvec w \, \lInter{A}{\pvec x}{\pvec y} \to \lInter{B}{\pvec f \pvec x}{\pvec w} \\[1mm]
        \lInter{A \wedge B}{\pvec x, \pvec v}{\pvec y, \pvec w} & \Leftrightarrow & \lInter{A}{\pvec x}{\pvec y} \wedge \lInter{B}{\pvec v}{\pvec w} \\[1mm]
        \lInter{A \vee B}{b, \pvec x, \pvec v}{\pvec y, \pvec w} & \Leftrightarrow & \mforall \pvec y' \leq^*_{\pvec \tau^-_{\lTrans{A}}} \!\pvec y \, \lInter{A}{\pvec x}{\pvec y'} \vee \mforall \pvec w' \leq^*_{\pvec \tau^-_{\lTrans{B}}} \! \pvec w \, \lInter{B}{\pvec v}{\pvec w'} \\[1mm]
        \lInter{\exists z^\tau A}{c, \pvec x}{\pvec y} & \Leftrightarrow & \exists z \leq^*_{\tau} c \mforall \pvec y' \leq^*_{\pvec \tau^-_{\lTrans{A}}} \! \pvec y \, \lInter{A}{\pvec x}{\pvec y'} \\[1mm]
        \lInter{\forall z^\tau A}{\pvec f}{b, \pvec y} & \Leftrightarrow & \forall z^\tau \leq^*_{\tau} b \, \lInter{A}{\pvec f b}{\pvec y}
    \end{array}
}
so that $\lInter{A}{\pvec x}{\pvec y}$ can be seen to correspond to $A_B(\pvec x; \pvec y)$.
\end{proposition}

\begin{proof} Direct from Proposition \ref{prop-lInter}. \end{proof}

\begin{remark} In order to extend the source theory with bounded quantifiers, in this case one must add an ``intensional'' majorizability relation $x \trianglelefteq y$, which satisfies
\[ f \trianglelefteq g \quad \Rightarrow \quad \forall x, y (x \trianglelefteq y \to (f x \trianglelefteq g y) \wedge (f x \trianglelefteq f y)) \]
with a rule-version of the other direction:
\[
\begin{prooftree}
\Gamma, x \trianglelefteq y \proves (f x \trianglelefteq g y) \wedge (f x \trianglelefteq f y)
\justifies 
\Gamma \proves f \trianglelefteq g 
\end{prooftree}
\]
Adding the other direction as an axiom would require us to produce a majorant for arbitrary $x$'s and $y$'s, which we do not have in $\Itarget = \nHAomega$.
\end{remark}

\paragraph {\bf Bounded Diller-Nahm interpretation} Let us consider now what we believe is another novel functional interpretation of $\HAomega$, where contraction is treated like in the Diller-Nahm interpretation (via finite sets), but the typing axioms are treated as in the bounded interpretations (via majorizability). As above, we are considering $\Isource = \HAomega$ and $\Itarget = \nHAomega$, but consider the following instantiation of the parameters:
\[
\begin{array}{ccccccc}
\bound{x}{a}{\tau} & \wtype{\tau} & \Wit_\tau(x) & \ubq{\pvec \tau}{\pvec x}{\pvec a} A & \btype{\pvec \tau} & \ifW{\tau}(b, x, y) & \apW{\tau}\\ 
\hline\\[-12pt]
\tau(x) \wedge (x \leq^*_\tau a) & \tau & x \leq^*_\tau x  & \mforall \pvec x \in_{\pvec \tau} \pvec a \, A  & \pvec \tau^* & \max_{\tau}(x, y) & \lambda f . f
\end{array}
\]
With these parameters the $\lInter{\cdot}{}{}$-interpretation becomes:
\[
\begin{array}{lcl}
\lInter{A \to B}{\pvec f, \pvec g}{\pvec x, \pvec w} 
	& \equiv & \mforall \pvec y \in_{\pvec \tau^-_{\lTrans{A}}} \pvec g \pvec x \pvec w \, \lInter{A}{\pvec x}{\pvec y} \to \lInter{B}{\pvec f \pvec x}{\pvec w} \\[1mm]
\lInter{A \wedge B}{\pvec x, \pvec v}{\pvec y, \pvec w} 
	& \equiv & \lInter{A}{\pvec x}{\pvec y} \wedge \lInter{B}{\pvec v}{\pvec w} \\[1mm]
\lInter{A \vee B}{\pvec x, \pvec v}{\pvec y, \pvec w} 
	& \equiv & \mforall \pvec y' \in_{\pvec \tau^-_{\lTrans{A}}} \pvec y \, \lInter{A}{\pvec x}{\pvec y'} \vee \mforall \pvec w' \in_{\pvec \tau^-_{\lTrans{B}}} \pvec w \, \lInter{B}{\pvec v}{\pvec w'} \\[1mm]
\lInter{\exists z^\tau A}{c, \pvec x}{\pvec y} 
	& \equiv & \exists z \leq^*_\tau c \mforall \pvec y' \in_{\pvec \tau^-_{\lTrans{A}}} \pvec y \lInter{A}{\pvec x}{\pvec y'} \\[1mm]
\lInter{\forall z^\tau A}{\pvec f}{b, \pvec y} 
	& \equiv & \forall z \leq^*_\tau b \, \lInter{A}{\pvec f b}{\pvec y}
\end{array}
\]

\begin{proposition}[Bounded Diller-Nahm interpretation] \label{BDN} The derived functional interpretation above is a sound interpretation of $\HAomega$.
\end{proposition}

\begin{proof} In order to prove the soundness for the interpretation it is enough to show that this choice of parameters satisfies the conditions of Theorem \ref{soundness2}, i.e. that 
\begin{enumerate}
	\item[$(i)$] the corresponding $\AL$ parameters, $\lbound{x}{a}{\tau}$, $\lWit_\tau(x)$ and $\lubq{\pvec \tau}{\pvec x}{\pvec a}{A}$, are an adequate choice for the formulas $\uInter{A}{\pvec x}{\pvec y}$, for all $A$ in $\lTrans{\Isource}$, and
	\item[$(ii)$] all the non-logical axioms of $\Isource$ are $\lInter{\cdot}{}{}$-witnessable in $\Itarget$. 
\end{enumerate}
Since $\lWit_\tau(x)$ is the assumption that $x$ is self-majorizing (i.e. monotone), conditions \WitS, \WitK, \WitApp~easily follow. \Quant{1} and \Quant{2} are also straightforward. The conditions for validating ``contraction" \ConUnit, \ConStrength, \ConApp~hold by taking $\singleton(x) = \{ x \}$ and $\join{\pvec y_1}{\pvec y_2} = \pvec y_1 \cup \pvec y_2$ and $\comp{\pvec f}{\pvec z} = \cup_{\pvec x \in \pvec z} \pvec f \pvec x$ as indeed we have:
\begin{enumerate}
    \item[\ConUnit] $(\pvec z \leq^* \pvec z), \mforall \pvec y \in \{ \pvec z \} A[\pvec y] \proves_{\Atarget} A[\pvec z]$
    \item[\ConStrength] $\mforall \pvec y \in \pvec y_1 \cup \pvec y_2 A[\pvec y] 
        \proves_{\Atarget} \mforall \pvec y \in \pvec y_1 A[\pvec y] \wedge \mforall \pvec y \in \pvec y_2 A[\pvec y]$
    \item[\ConApp] $\mforall \pvec y \in \cup_{\pvec x \in \pvec z} \pvec f \pvec x \, A[\pvec y] 
        \proves_{\Atarget} \mforall \pvec x \in \pvec z \, \mforall \pvec y \in \pvec f \pvec x \, A[\pvec y]$
\end{enumerate}
That concludes the proof that the choice of parameters is an adequate choice for the formulas $\uInter{A}{\pvec x}{\pvec y}$, for all $A$ in $\lTrans{\Isource}$. We must also show that the non-logical axioms of $\Isource$ are $\lInter{\cdot}{}{}$-witnessable in $\Itarget$. This can be done by verifying that the assumptions \Assumption{6}, \Assumption{7} and \Assumption{8} can be satisfied by appropriate terms, which is straightforward (similar to the bounded interpretations).
\end{proof}

\begin{remark} \label{remark-possible-equiv} It could turn out, however, that this ``Bounded Diller-Nahm interpretation'' is actually equivalent (in the sense of having the same characterising principles) as the Diller-Nahm interpretation or the bounded functional interpretation. This is still open. But we suspect this will not be the case, since being a member of a finite set is strictly stronger than being majorized by some element. More precisely, from $x \in_\tau a$ we indeed have $x \leq^* \max a$. But from the assumption $x \leq_\tau^* a$ we cannot in general find a finite set $\tilde{a}$ (depending only on $a$) such that $x \in \tilde{a}$. This should be settled once we have investigated the characterising principles of this new interpretation, which we plan to do in a follow up paper.
\end{remark}

\subsection{Interpretations where $\bound{x}{a}{\tau} \; \pdefin \; x \in_{\tau} a$}

The instances where $\bound{x}{a}{\tau}$ is chosen to be $x \in_{\tau} a$, with $\wtype{\tau} = \tau^*$ and $\apW{P}(f)(a) = \bigcup \{ f z : z \in a \}$, which we call \emph{Herbrand interpretations}, give some new interpretations for $\HAomega$ which are related with recently developed functional interpretations for nonstandard arithmetic \cite{BergBriseidSafarik2012}. In fact, to obtain the latter interpretations one needs to consider two types of predicate symbols, as explained in Section \ref{sectionfinalremarks}.

\paragraph {\bf Herbrand realizability (for IL)} Consider the following instantiation of the parameters:
\[
\begin{array}{ccccccc}
\bound{x}{a}{\tau} & \wtype{\tau} & \Wit_\tau(x) & \ubq{\pvec \tau}{\pvec x}{\varepsilon} A & \btype{\pvec \tau} & \ifW{\tau}(b, x, y) & \apW{\tau}(f)(a) \\
\hline
\tau(x) \wedge (x \in_\tau a) & \tau^* & \mathrm{true} & \forall \pvec x^{\pvec \tau} A  & \varepsilon & x \cup y & \bigcup\limits_{z \in a} f z
\end{array}
\]
Again we see that in $\bInter{A}{\pvec x}{\pvec y}$ the tuple $\pvec y$ will be empty and hence we omit it.

\begin{proposition}[Herbrand realizability]\label{Herbrandrealizability} With the parameters instantiated as above we have:
\eqleft{
\begin{array}{lcl}
\bInter{A \to B}{\pvec f}{} & \Leftrightarrow & \forall \pvec x^{\tau^+_{\bTrans{A}}}  (\bInter{A}{\pvec x}{} \to \bInter{B}{\pvec f \pvec x}{}) \\[2mm]
\bInter{A \wedge B}{\pvec x, \pvec v}{} & \Leftrightarrow & \bInter{A}{\pvec x}{} \wedge \bInter{B}{\pvec v}{} \\[2mm]
\bInter{A \vee B}{\pvec x, \pvec v, b}{} & \Leftrightarrow & \BB(b) \wedge (\pcond{b}{\bInter{A}{\pvec x}{}}{\bInter{B}{\pvec v}{}}) \\[2mm]
%
\bInter{\exists z^\tau A}{\pvec x, c}{} & \Leftrightarrow & \exists z \in_\tau c \bInter{A}{\pvec x}{} \\[2mm]
\bInter{\forall z^\tau A}{\pvec f}{} & \Leftrightarrow & \forall z \in_\tau b  \bInter{A}{\pvec f b}{}\end{array}
}
\end{proposition}

\begin{proof} Direct from Proposition \ref{prop-bInter}, using the above instantiation of the parameters.
\end{proof}

\paragraph {\bf Herbrand Diller-Nahm interpretation} Consider the following instantiation of the parameters:
\[
\begin{array}{ccccccc}
\bound{x}{a}{\tau} & \wtype{\tau} & \Wit_\tau(x) & \ubq{\pvec \tau}{\pvec x}{\pvec a} A & \btype{\pvec \tau} & \ifW{\tau}(b, x, y) & \apW{\tau}(f)(a) \\
\hline
\tau(x) \wedge (x \in_\tau a) & \tau^* & \mathrm{true} & \forall \pvec x \in_{\pvec \tau} \pvec a \, A  & \pvec \tau^* & x \cup y & \bigcup\limits_{z \in a} f z
\end{array}
\]

\begin{proposition}[Herbrand Diller-Nahm interpretation] \label{PropHFI} With the parameters instantiated as above we have:
\eqleft{
\begin{array}{lcl}
\lInter{A \to B}{\pvec f, \pvec g}{\pvec x, \pvec w} & \Leftrightarrow & \forall \pvec y \in_{\tau^+_{\lTrans{A}}} \! \pvec g \pvec x \pvec w \, \lInter{A}{\pvec x}{\pvec y} \to \lInter{B}{\pvec f \pvec x}{\pvec w} \\[2mm]
\lInter{A \wedge B}{\pvec x, \pvec v}{\pvec y, \pvec w} & \Leftrightarrow & \lInter{A}{\pvec x}{\pvec y} \wedge \lInter{B}{\pvec v}{\pvec w} \\[2mm]
\lInter{A \vee B}{b, \pvec x, \pvec v}{\pvec y, \pvec w} & \Leftrightarrow & \forall \pvec y' \in_{\tau^+_{\lTrans{A}}} \! \pvec y \, \lInter{A}{\pvec x}{\pvec y'} \vee \forall \pvec w' \in_{\tau^+_{\lTrans{B}}} \! \pvec w \, \lInter{B}{\pvec v}{\pvec w'} \\[2mm]
%
\lInter{\exists z^\tau A}{\pvec x}{\pvec y} & \Leftrightarrow & \exists z \in_\tau c \forall \pvec y' \in_{\tau^+_{\lTrans{A}}} \! \pvec y \, \lInter{A}{\pvec x}{\pvec y'} \\[2mm]
\lInter{\forall z^\tau A}{\pvec x}{b, \pvec y} & \Leftrightarrow & \forall z \in_\tau b \lInter{A}{\pvec x}{\pvec y}
\end{array}
}
and this is a sound interpretation of $ \HAomega$.
\end{proposition}

\begin{proof} Direct from Proposition \ref{prop-lInter}, using the above instantiation of the parameters.
\end{proof}

The Herbrand realizability for intuitionistic logic given by Proposition~\ref{Herbrandrealizability} and the Herbrand Diller-Nahm interpretation given by Proposition~\ref{PropHFI} are in a sense "rediscovered" interpretations. In fact, these interpretations are closely related with the interpretations given in \cite{BergBriseidSafarik2012} for $\HAomega_{\st}$. They are also closely connected with the interpretation for ``pure logic'' considered by Gilda Ferreira and Fernando Ferreira in the paper \cite{Ferreira2017}. Moreover, Fernando Ferreira has a preprint \cite{FF(ta)} where he considers essentially the Herbrand Diller-Nahm interpretation for $\HAomega$ as well as an extension to second order arithmetic.

Our parametrised interpretations allow us to consider also a Herbrand version of the bounded functional interpretation. 

\paragraph {\bf Herbrandized bfi} We conclude this list of instantiations with what we believe is yet another novel functional interpretation of $\HAomega$, where contraction is treated like the Herbrandized interpretations, but the typing axioms are treated as in the bounded interpretations:
\[
\begin{array}{ccccccc}
\bound{x}{a}{\tau} & \wtype{\tau} & \Wit_\tau(x) & \ubq{\pvec \tau}{\pvec x}{\pvec a} A & \btype{\pvec \tau} & \ifW{\tau}(b, x, y) & \apW{\tau}(f)(a) \\
\hline\\[-12pt]
\tau(x) \wedge (x \in_\tau a) & \tau^* & x \leq_\tau^* x & \mforall \pvec x \leq_{\pvec \tau}^* \! \pvec a \, A  & \pvec \tau & x \cup y & \bigcup\limits_{z \in a} f z
\end{array}
\]

\begin{proposition}[Herbrandized bounded functional interpretation of $\HAomega$] \label{h-bfi} With the parameters instantiated as above we have:
\eqleft{
\begin{array}{lcl}
\lInter{A \to B}{\pvec f, \pvec g}{\pvec x, \pvec w} & \Leftrightarrow & \mforall \pvec y \leq^*_{\pvec \tau^-_{\lTrans{A}}} \pvec g \pvec x \pvec w \, \lInter{A}{\pvec x}{\pvec y} \to \lInter{B}{\pvec f \pvec x}{\pvec w} \\[2mm]
\lInter{A \wedge B}{\pvec x, \pvec v}{\pvec y, \pvec w} & \Leftrightarrow & \lInter{A}{\pvec x}{\pvec y} \wedge \lInter{B}{\pvec v}{\pvec w} \\[2mm]
\lInter{A \vee B}{\pvec x, \pvec v}{\pvec y, \pvec w} & \Leftrightarrow & \mforall \pvec y' \leq^*_{\pvec \tau^-_{\lTrans{A}}} \pvec y \, \lInter{A}{\pvec x}{\pvec y'}  \vee \mforall \pvec w' \leq^*_{\pvec \tau^-_{\lTrans{B}}} \pvec w\, \lInter{B}{\pvec v}{\pvec w'} \\[2mm]
%
\lInter{\exists z^\tau A}{c, \pvec x}{\pvec y} & \Leftrightarrow & \exists z \in_\tau c \mforall \pvec y' \leq^*_{\pvec \tau^-_{\lTrans{A}}} \pvec y \, \lInter{A}{\pvec x}{\pvec y'}\\[2mm]
\lInter{\forall z^\tau A}{\pvec f}{b, \pvec y} & \Leftrightarrow & \forall z \in_\tau b \, \lInter{A}{\pvec f b}{\pvec y}\\[2mm]
\end{array}
}
and this is a sound interpretation of $ \HAomega$.
\end{proposition}

\begin{proof} It is easy to check that these equivalences hold, by applying Proposition~\ref{prop-lInter} using the above instantiation of the parameters. That this is a sound interpretation, then follows by verifying assumption \Assumption{1}-\Assumption{8}, which is quite straightforward and follows the same patterns as in the previous instances.
\end{proof}

A remark similar to Remark~\ref{remark-possible-equiv} also applies here. We suspect that this is a new interpretation, but will only be certain once we have investigated its characterising principles.

\begin{remark}[A Herbrandized Dialectica] One might also consider an instantiation of the parameters as follows:
\[
\begin{array}{ccccccc}
\bound{x}{a}{\tau} & \wtype{\tau} & \Wit_\tau(x) & \ubq{\pvec \tau}{\pvec x}{\pvec a} A & \btype{\pvec \tau} & \ifW{\tau}(b, x, y) & \apW{\tau}(f)(a) \\
\hline
\tau(x) \wedge (x \in_\tau a) & \tau^* & \rm{true} & A[\pvec a / \pvec x]  & \pvec \tau & x \cup y & \bigcup\limits_{z \in a} f z
\end{array}
\]
which would correspond to a ``Herbrandized" version of the Dialectica interpretation. In this case contraction is dealt with in a precise way, but quantifiers are approximated by finite sets. In $\HAomega$, however, where definition by cases is available, it's easy to check that this would give rise to an interpretation which is equivalent to the original Dialectica, since the interpretation of the quantifiers
\eqleft{
\begin{array}{lcl}
\lInter{\exists z^\tau A}{c, \pvec x}{\pvec y} & \Leftrightarrow & \exists z \in_\tau \! c \, \lInter{A}{\pvec x}{\pvec y} \\[2mm]
\lInter{\forall z^\tau A}{\pvec f}{\pvec y, c} & \Leftrightarrow & \forall z \in_\tau \! c \, \lInter{A}{\pvec f c}{\pvec y}
\end{array}
}
can be effectively replaced by precise witnesses
\eqleft{
\begin{array}{lcl}
\lInter{\exists z^\tau A}{c, \pvec x}{\pvec y} & \Leftrightarrow & \lInter{A[c/z]}{\pvec x}{\pvec y} \\[2mm]
\lInter{\forall z^\tau A}{\pvec f}{\pvec y, c} & \Leftrightarrow & \lInter{A[c/z]}{\pvec f c}{\pvec y}
\end{array}
}
\end{remark}

\section{Final Remarks}
\label{sectionfinalremarks}

We have described above a general framework for unifying several functional interpretation, which we then used to discover new interpretations. These are summarised in Figure \ref{summary-table}.

\begin{figure}[t]
\begin{center}
\resizebox{\columnwidth}{!}{%
\begin{tabular}{c|c|c|c|c|c|c}
    $\wtype{\tau}$ & $\bound{x}{a}{\tau}$ & $ \ifW{\tau}(b, x, y) \; / \; \apW{\tau}(f, a)$ & $\btype{\pvec \tau}$ & $\ubq{\pvec \tau}{\pvec x}{\pvec a} A$ & $\Wit_\tau(x)$ & {\bf Interpretation} \\[2mm]
    \hline
    $\tau$ & $x =_\tau a$ &  ${\rm if}_\tau(b, x, y) \; / \; f(a)$ & $\pvec \tau$ & $A[\pvec a / \pvec x]$ & true & Dialectica  \\[2mm]
    $\tau$ & $x =_\tau a$ &  ${\rm if}_\tau(b, x, y) \; / \; f(a)$ & $\varepsilon$ & $\forall \pvec x^{\pvec \tau} A$ & true & Modified realizability \\[2mm]
    $\tau$ & $x =_\tau a$ &  ${\rm if}_\tau(b, x, y) \; / \; f(a)$ & $\pvec \tau$ & $\forall \pvec x \leq^*_{\pvec \tau} \! \pvec a \, A$ & true / $x \leq^*_\tau x$  & (\emph{combination not sound}) \\[2mm]
    $\tau$ & $x =_\tau a$ & ${\rm if}_\tau(b, x, y) \; / \; f(a)$ & $\pvec \tau^*$ & $\forall \pvec x \in_{\pvec \tau} \! \pvec a \, A$ & true & Diller-Nahm  \\[2mm]
    \hline
    $\tau$ & $x \leq^*_\tau a$ & $\max_\tau(x, y) \; / \; f(a)$ & $\pvec \tau$ & $A[\pvec a / \pvec x]$ & $x \leq^*_\tau x$ & (\emph{combination not sound}) \\[2mm]
    $\tau$ & $x \leq^*_\tau a$ & $\max_\tau(x, y) \; / \; f(a)$ & $\varepsilon$ & $\mforall \pvec x^{\pvec \tau} A$ & $x \leq^*_\tau x$ & Bounded modified realizability \\[2mm]
    $\tau$ & $x \leq^*_\tau a$ & $\max_\tau(x, y) \; / \; f(a)$ & $\pvec \tau$ & $\mforall \pvec x \leq^*_{\pvec \tau} \! \pvec a \, A$ & $x \leq^*_\tau x$ & Bounded functional interpretation \\[2mm]
    $\tau$ & $x \leq^*_\tau a$ & $\max_\tau(x, y) \; / \; f(a)$ & $\pvec \tau^*$ & $\mforall \pvec x \in_{\pvec \tau} \! \pvec a \, A$ & $x \leq^* x$ & {\bf Bounded Diller-Nahm} \\[2mm]
    \hline
    $\tau^*$ & $x \in_\tau a$ & $x \cup y \; / \; \bigcup\limits_{z \in a} f z$ & $\pvec \tau$ & $A[\pvec a / \pvec x]$ & true & Herbrand Dialectica ( $\simeq$ Dialectica) \\[2mm]
    $\tau^*$ & $x \in_\tau a$ &   $x \cup y \; / \; \bigcup\limits_{z \in a} f z$ & $\varepsilon$ & $\forall \pvec x^{\pvec \tau} A$ & true & Herbrand realizability (for IL) \\[2mm]
    $\tau^*$ & $x \in_\tau a$ &   $x \cup y \; / \; \bigcup\limits_{z \in a} f z$ & $\pvec \tau$ & $\mforall \pvec x \leq^*_{\pvec \tau} \! \pvec a \, A$ & $x \leq^*_\tau x$ & {\bf Herbrandized bfi} \\[2mm]
    $\tau^*$ & $x \in_\tau a$ &   $x \cup y \; / \; \bigcup\limits_{z \in a} f z$ & $\pvec \tau^*$ & $\forall \pvec x \in_{\pvec \tau} \! \pvec a \, A$ & true & Herbrand Diller-Nahm
\end{tabular}
}
\end{center}
\caption{Summary of instantiations (with the two novel interpretations in \textbf{bold})}
\label{summary-table}
\end{figure}

A notable family of functional interpretations that we are not covering in this paper is Kohlenbach's \emph{monotone functional interpretations} (see \cite{K(96),K(08)}). We focus here on the different ways a \emph{formula} can be given a functional interpretation. The monotone functional interpretation in fact makes use of these same interpretations of formulas, but with a different interpretation of \emph{proofs}. More precisely, given the interpretation of a formula $A$ as $\uInter{A}{\pvec x}{\pvec y}$, we are focusing here on the soundness theorem that guarantees the existence of terms $\pvec t$ such that $\uInter{A}{\pvec t}{\pvec y}$ for provable $A$. In the monotone functional interpretation a different soundness proof is used, which, for provable $A$, guarantees the existence of terms $\tilde{\pvec t}$ such that $\exists \pvec x \! \leq^* \! \tilde{\pvec t} \, \uInter{A}{\pvec x}{\pvec y}$, where $\leq^*$ is Bezem's strong majorizability relation. Hence, one could consider ``monotone" soundness theorems for each of the interpretations discussed here, but we leave this to future work.

As shown in the previous section, the parametrised interpretations presented in this paper can be used as a way to discover new interpretations. The instances that we considered are by no means exhaustive. For instance, we think that the interpretations for nonstandard arithmetic from \cite{BergBriseidSafarik2012,DG(18),FerreiraGaspar2015} should also fit in our framework. The idea is to consider not just the typing predicate symbols $\tau(x)$, but also the standard predicate symbol $\st(x)$ as computational symbols, giving rise to the parameters $\bound{x}{a}{\tau}$ and $\bound{x}{a}{\st}$, which each can be given a different interpretation (on top of the choice of interpreting contraction via $\ubq{}{x}{a} A$). 
Suitable choices for these should lead to the known interpretations of nonstandard arithmetic, but might also give rise to new  interpretations for nonstandard arithmetic. This study, however, goes behind the scope of this paper.

Another question concerns variants with truth \cite{GO(10)}. We think that it may be possible to obtain the existing interpretations with truth, and maybe to find new ones, using our parametrised interpretations, by changing the interpretation of $\bang A$ in Definition \ref{inter} to $\uInter{\bang A}{\pvec x}{\pvec a} \pdefin \bang \ubq{}{\pvec y}{\pvec a} \uInter{A}{\pvec x}{\pvec y} \, \otimes \, \bang A$. We also leave this to future work. 

Usually, functional interpretations are accompanied by a characterisation theorem where one shows the equivalence between a formula and its interpretation. In order to show such equivalence one requires some principles -- typically, a form of Choice and of Markov's principle are among such principles -- which are called the characteristic principles of the interpretation. In the  case of our parametrised interpretation we do not know if such a (parametric) theorem holds. We were able to define parametrised characteristic principles and obtain the result but only assuming that the characteristic principles are interpretable (by themselves). This does not solve the problem since it may happen that the theory with the principles may not be consistent. However, for each particular instantiation described in this paper the parametrised characteristic principles indeed correspond to the actual characteristic principles of the interpretation obtained with that instantiation. So, it seems that if the resulting theory is consistent, then the parametrised interpretation admits a characterisation theorem.

Finally, it is well-known that intuitionistic functional interpretations are related with classical ones by means of a negative translation. For example, as shown in \cite{Avigad2006,StreicherKohlenbach2007}, Jean-Louis Krivine's negative translation is the correct tool to connect G\"{o}del's \emph{Dialectica} with Shoenfield's interpretation. Other factorisations were obtained in \cite{DG(ta),Gaspar2009,StreicherKohlenbach2007,OlivaStreicher2008}. It is our impression that composing our intuitionistic parametrised interpretation with various negative translations would entail parametrised classical interpretations that allows one to obtain all the standard interpretations for classical logic, showing factorisations are a general feature among functional interpretations. We also leave this to a future study.

\bibliographystyle{plain} 

\bibliography{dblogic}

\newpage

\section*{Appendix A: Proof of Proposition \ref{prop-lInter}}

\noindent Straightforward by simply unfolding definitions: \\[2mm]
If $A = P(\pvec x)$ and $P$ is a computational predicate symbol then:
\[
\lInter{P(\pvec x)}{a}{} 
	\stackrel{\textup{D}\ref{def:combined-inter}}{\Leftrightarrow} \fTrans{(\uInter{\lTrans{(P(\pvec x))}}{a}{})}
	\stackrel{\textup{D}\ref{g-trans}}{\Leftrightarrow} \fTrans{(\uInter{P(\pvec x)}{a}{})}
	\stackrel{\textup{D}\ref{inter}}{\Leftrightarrow} \fTrans{(\lbound{\pvec x}{a}{P})}
	\stackrel{\textup{D}\ref{forget}}{\Leftrightarrow} \bound{\pvec x}{a}{P}	
\]
If $A = P(\pvec x)$ and $P$ is a non-computational predicate symbol then:
\[ \lInter{P(\pvec x)}{}{} 
	\stackrel{\textup{D}\ref{def:combined-inter}}{\Leftrightarrow} \fTrans{(\uInter{\lTrans{(P(\pvec x))}}{}{})}
	\stackrel{\textup{D}\ref{g-trans}}{\Leftrightarrow} \fTrans{(\uInter{P(\pvec x)}{}{})} 
	\stackrel{\textup{D}\ref{inter}}{\Leftrightarrow} \fTrans{(P(\pvec x))}
	\stackrel{\textup{D}\ref{forget}}{\Leftrightarrow} P(\pvec x)	
\]
Implication.
\eqleft{
\begin{array}{lcl}
	\lInter{A \to B}{\pvec f, \pvec g}{\pvec x,\pvec w} 
	& \stackrel{\textup{D}\ref{def:combined-inter}}{\Leftrightarrow} & \fTrans{(\uInter{\lTrans{(A \to B)}}{\pvec f,\pvec g}{\pvec x,\pvec w})}  \\[1mm]
	& \stackrel{\textup{D}\ref{g-trans}}{\Leftrightarrow} & \fTrans{(\uInter{\bang \lTrans{A} \lto \lTrans{B}}{\pvec f,\pvec g}{\pvec x,\pvec w})}    \\[1mm]
	& \stackrel{\textup{D}\ref{inter}}{\Leftrightarrow} &  \fTrans{(\bang \lubq{\pvec \tau^-_{\lTrans{A}}}{\pvec y}{\pvec g \pvec x \pvec w}{\uInter{ \lTrans{A}}{\pvec x}{\pvec y}} 
		\lto \uInter{ \lTrans{B}}{ \pvec f \pvec x}{\pvec w})} 	 \\[1mm]
	& \stackrel{\textup{D}\ref{def-parameters-translation}}{\Leftrightarrow} &  \fTrans{(\bang \lTrans{( \ubq{\pvec \tau^-_{\lTrans{A}}}{\pvec y}{\pvec g \pvec x \pvec w} \fTrans{(\uInter{\lTrans{A}}{\pvec x}{\pvec y})} )}
		\lto \uInter{ \lTrans{B}}{ \pvec f \pvec x}{\pvec w})} 	 \\[1mm]
	& \stackrel{\textup{D}\ref{forget}}{\Leftrightarrow} &  \ubq{\pvec \tau^-_{\lTrans{A}}}{\pvec y}{\pvec g \pvec x \pvec w} \fTrans{(\uInter{ \lTrans{A}}{\pvec x}{\pvec y})} \to \fTrans{(\uInter{ \lTrans{B}}{\pvec f \pvec x}{\pvec w})}  \\[1mm]
	& \stackrel{\textup{D}\ref{def:combined-inter}}{\Leftrightarrow} & \ubq{\pvec \tau^-_{\lTrans{A}}}{\pvec y}{\pvec g \pvec x \pvec w} \lInter{A }{\pvec x}{\pvec y} \to \lInter{B}{\pvec f \pvec x}{\pvec w}  
\end{array}
}
Conjunction.
\eqleft{
\begin{array}{lcl}
	\lInter{A \wedge B}{\pvec x, \pvec v}{\pvec y,\pvec w} 
	& \stackrel{\textup{D}\ref{def:combined-inter}}{\Leftrightarrow}& \fTrans{(\uInter{\lTrans{(A \wedge B)}}{\pvec x,\pvec v}{\pvec y,\pvec w})}   \\[1mm]
	& \stackrel{\textup{D}\ref{g-trans}}{\Leftrightarrow} & \fTrans{(\uInter{\lTrans{A} \cwedge \lTrans{B}}{\pvec x,\pvec v}{\pvec y,\pvec w})}   \\[1mm]
	& \stackrel{\textup{D}\ref{inter}}{\Leftrightarrow} &  \fTrans{( \uInter{ \lTrans{A}}{\pvec x}{\pvec y} \cwedge \uInter{ \lTrans{B}}{  \pvec v}{\pvec w})} 
	\stackrel{\textup{D}\ref{forget}}{\Leftrightarrow} \fTrans{(\uInter{ \lTrans{A}}{\pvec x}{\pvec y})} \wedge \fTrans{(\uInter{ \lTrans{B}}{\pvec v}{\pvec w})}
	\stackrel{\textup{D}\ref{def:combined-inter}}{\Leftrightarrow} \lInter{A}{\pvec x}{\pvec y} \wedge \lInter{B}{\pvec v}{\pvec w} 
\end{array}
}
Disjunction. Recall that $A \vee B$ is defined as $\exists z^\BB (((z = \true) \to A) \wedge ((z = \false) \to B))$ (Proposition \ref{prop-def-disjunction}). Hence
\[
\footnotesize{\begin{array}{lcl}
	\lInter{A \vee B}{b, \pvec x, \pvec v}{\pvec y, \pvec w} 
	& \stackrel{\textup{D}\ref{def:combined-inter}}{\Leftrightarrow} & \fTrans{(\uInter{\lTrans{(A \vee B)}}{b, \pvec x,\pvec v}{\pvec y,\pvec w})}
	\stackrel{\textup{D}\ref{g-trans}}{\Leftrightarrow} \fTrans{(\uInter{\exists z^\BB (\lpcond{z}{\bang \lTrans{A}}{\bang\lTrans{B}})}{b, \pvec x,\pvec v}{\pvec y,\pvec w})}  \\[1mm]
	& \stackrel{\textup{D}\ref{inter}}{\Leftrightarrow} & \fTrans{(\exists \lbound{z}{b}{\BB} (\lpcond{z}{\bang \lubq{\pvec \tau^-_{\lTrans{A}}}{\pvec y'}{\pvec y}{\uInter{\lTrans{A}}{\pvec x}{\pvec y'}}}{\bang \lubq{\pvec \tau^-_{\lTrans{B}}}{\pvec w'}{\pvec w}{\uInter{\lTrans{B}}{\pvec v}{\pvec w'}}}))}	 \\[1mm]
	& \stackrel{\textup{D}\ref{forget}}{\Leftrightarrow} &  \exists \bound{z}{b}{\BB} (\pcond{z}{\ubq{\pvec \tau^-_{\lTrans{A}}}{\pvec y'}{\pvec y} \fTrans{(\uInter{\lTrans{A}}{\pvec x}{\pvec y'})}}{\ubq{\pvec \tau^-_{\lTrans{B}}}{\pvec w'}{\pvec w} \fTrans{(\uInter{\lTrans{B}}{\pvec v}{\pvec w'})}}) \\[1mm]
	& \stackrel{\textup{D}\ref{def:combined-inter}}{\Leftrightarrow} & \exists \bound{z}{b}{\BB} (\pcond{z}{\ubq{\pvec \tau^-_{\lTrans{A}}}{\pvec y'}{\pvec y} \lInter{A}{\pvec x}{\pvec y'}}{\ubq{\pvec \tau^-_{\lTrans{B}}}{\pvec w'}{\pvec w} \lInter{B}{\pvec v}{\pvec w'}}) 
\end{array}}
\]
Existential quantifier.
\eqleft{
\begin{array}{lcl}
	\lInter{\exists z A}{\pvec x}{\pvec y} 
	& \stackrel{\textup{D}\ref{def:combined-inter}}{\Leftrightarrow}& \fTrans{(\uInter{\lTrans{(\exists z A)}}{\pvec x}{\pvec y})}   \\[1mm]
	& \stackrel{\textup{D}\ref{g-trans}}{\Leftrightarrow} & \fTrans{(\uInter{\exists z \bang \lTrans{A}}{\pvec x}{\pvec y})}    \\[1mm]
	& \stackrel{\textup{D}\ref{inter}}{\Leftrightarrow} & \fTrans{(\exists z \bang \lubq{\pvec \tau^-_{\lTrans{A}}}{\pvec y'}{\pvec y}{\uInter{\lTrans{A}}{\pvec x}{\pvec y'}})}	 \\[1mm]
	& \stackrel{\textup{D}\ref{forget}}{\Leftrightarrow} &  \exists z \ubq{\pvec \tau^-_{\lTrans{A}}}{\pvec y'}{\pvec y} \fTrans{(\uInter{\lTrans{A}}{\pvec x}{\pvec y'})}
	\stackrel{\textup{D}\ref{def:combined-inter}}{\Leftrightarrow} \exists z \ubq{\pvec \tau^-_{\lTrans{A}}}{\pvec y'}{\pvec y} \lInter{A}{\pvec x}{\pvec y'}
\end{array}
}
Universal quantifier.
\[
	\lInter{\forall z A}{\pvec x}{\pvec y} 
	\stackrel{\textup{D}\ref{def:combined-inter}}{\Leftrightarrow} \fTrans{(\uInter{\lTrans{(\forall z A)}}{\pvec x}{\pvec y})}
	\stackrel{\textup{D}\ref{g-trans}}{\Leftrightarrow} \fTrans{(\uInter{\forall z\lTrans{A}}{\pvec x}{\pvec y})}
	\stackrel{\textup{D}\ref{inter}}{\Leftrightarrow} \fTrans{(\forall z \uInter{\lTrans{A}}{\pvec x}{\pvec y})}
	\stackrel{\textup{D}\ref{forget}}{\Leftrightarrow}  \forall z \fTrans{(\uInter{\lTrans{A}}{\pvec x}{\pvec y})}
	\stackrel{\textup{D}\ref{def:combined-inter}}{\Leftrightarrow} \forall z \lInter{A}{\pvec x}{\pvec y} 
\]
Relativised (computational) existential quantifier.
\eqleft{
\begin{array}{lcl}
	\lInter{\exists z^P A}{c, \pvec x}{\pvec y} 
	& \stackrel{\textup{N}\ref{notation:qual-IL}}{\Leftrightarrow}& \lInter{\exists z (P(z) \wedge A)}{c, \pvec x}{\pvec y}   \\[1mm]
	& \stackrel{\textup{D}\ref{def:combined-inter}}{\Leftrightarrow}& \fTrans{(\uInter{\lTrans{(\exists z (P(z) \wedge A))}}{c, \pvec x}{\pvec y})}   \\[1mm]
	& \stackrel{\textup{D}\ref{g-trans}}{\Leftrightarrow} & \fTrans{(\uInter{\exists z \bang ((P(z))^* \cwedge \lTrans{A})}{c, \pvec x}{\pvec y})}    \\[1mm]
	& \stackrel{\textup{D}\ref{inter}}{\Leftrightarrow} & \fTrans{(\exists z \bang \lubq{\pvec \tau^-_{\lTrans{A}}}{\pvec y'}{\pvec y}{((\lbound{z}{c}{P}) \cwedge \uInter{\lTrans{A}}{\pvec x}{\pvec y'})})} \\[1mm]
	& \stackrel{\textup{D}\ref{forget}}{\Leftrightarrow} & \exists z \ubq{\pvec \tau^-_{\lTrans{A}}}{\pvec y'}{\pvec y} ((\bound{z}{c}{P}) \wedge \fTrans{(\uInter{\lTrans{A}}{\pvec x}{\pvec y})})  \\[1mm]
	& \stackrel{\textup{(IH)}}{\Leftrightarrow} & \exists z \ubq{\pvec \tau^-_{\lTrans{A}}}{\pvec y'}{\pvec y} ((\bound{z}{c}{P}) \wedge \lInter{A}{\pvec x}{\pvec y})         \\[1mm]
	& \stackrel{\textup{\Quant{3}}}{\Leftrightarrow} & \exists \bound{z}{c}{P} \ubq{\pvec \tau^-_{\lTrans{A}}}{\pvec y'}{\pvec y} \lInter{A}{\pvec x}{\pvec y} 
\end{array}
}
Relativised (computational) universal quantifier.
\eqleft{
\begin{array}{lcl}
	\lInter{\forall z^P A}{\pvec f}{\pvec y, b} 
	& \stackrel{\textup{N}\ref{notation:qual-IL}}{\Leftrightarrow}& \lInter{\forall z (P(z) \to A)}{\pvec f}{\pvec y, b}   \\[1mm]
	& \stackrel{\textup{D}\ref{def:combined-inter}}{\Leftrightarrow}& \fTrans{(\uInter{\lTrans{(\forall z (P(z) \to A))}}{\pvec f}{\pvec y, b})}   \\[1mm]
	& \stackrel{\textup{D}\ref{g-trans}}{\Leftrightarrow}& \fTrans{(\uInter{\forall z (\bang P(z) \lto \lTrans{A})}{\pvec f}{\pvec y, b})}    \\[1mm]
	& \stackrel{\textup{D}\ref{inter}}{\Leftrightarrow}& \fTrans{(\forall z (\bang(\lbound{z}{b}{P}) \lto \uInter{\lTrans{A}}{\pvec f b}{\pvec y})}	 \\[1mm]
	& \stackrel{\textup{D}\ref{forget}}{\Leftrightarrow}& \forall \bound{z}{b}{P} \, \fTrans{(\uInter{\lTrans{A}}{\pvec f b}{\pvec y})} 
	\stackrel{\textup{D}\ref{def:combined-inter}}{\Leftrightarrow} \forall \bound{z}{b}{P} \, \lInter{A}{\pvec f b}{\pvec y}
\end{array}
}

\newpage

\section*{Appendix B: Proof of Proposition \ref{prop-bInter}}

\noindent Straightforward by simply unfolding definitions: \\[2mm]
If $A = P(\pvec x)$ with $P$ a computational predicate symbol then
\eqleft{
\bInter{P(\pvec x)}{a}{} 
	\stackrel{\textup{D}\ref{def:combined-inter}}{\equiv} \fTrans{(\uInter{\bTrans{(P(\pvec x))}}{a}{})}
	\stackrel{\textup{D}\ref{g-trans}}{\equiv} \fTrans{(\uInter{\bang P(\pvec x)}{a}{})}
	\stackrel{\textup{D}\ref{inter}}{\equiv} \fTrans{(\bang (\bbound{\pvec x}{a}{P}))}
	\stackrel{\textup{D}\ref{forget}}{\equiv} \bound{\pvec x}{a}{P} 
}
If $A = P(\pvec x)$ with $P$ a non-computational predicate symbol then
\eqleft{
\bInter{P(\pvec x)}{}{} 
	\stackrel{\textup{D}\ref{def:combined-inter}}{\equiv} \fTrans{(\uInter{(\bTrans{P(\pvec x))}}{}{})}
	\stackrel{\textup{D}\ref{g-trans}}{\equiv} \fTrans{(\uInter{\bang P(\pvec x)}{}{})}
	\stackrel{\textup{D}\ref{inter}}{\equiv} \fTrans{(\bang P(\pvec x))}
	\stackrel{\textup{D}\ref{forget}}{\equiv} P(\pvec x)
}
Implication.
\eqleft{
\begin{array}{lcl}
	\bInter{A \to B}{\pvec f, \pvec g}{\pvec x,\pvec w} 
	& \stackrel{\textup{D}\ref{def:combined-inter}}{\equiv} & \fTrans{(\uInter{\bTrans{(A \to B)}}{\pvec f,\pvec g}{\pvec x,\pvec w})}   \\[1mm]
	& \stackrel{\textup{D}\ref{g-trans}}{\equiv} & \fTrans{(\uInter{\bang (\bTrans{A} \lto \bTrans{B})}{\pvec f,\pvec g}{\pvec x,\pvec w})}  \\[1mm]
	& \stackrel{\textup{D}\ref{inter}}{\equiv} &  \fTrans{(\bang \bubq{\pvec \tau^+_{\bTrans{A}}, \pvec \tau^-_{\bTrans{B}}}{\pvec x', \pvec w'}{\pvec x, \pvec w} (\uInter{ \bTrans{A}}{\pvec x'}{\pvec g \pvec x' \pvec w'} \lto \uInter{ \bTrans{B}}{ \pvec f \pvec x'}{\pvec w'}))} 	 \\[1mm]
	& \stackrel{\textup{D}\ref{forget}}{\equiv} &  \ubq{\pvec \tau^+_{\bTrans{A}}, \pvec \tau^-_{\bTrans{B}}}{\pvec x', \pvec w'}{\pvec x, \pvec w} (\fTrans{(\uInter{ \bTrans{A}}{\pvec x'}{\pvec g \pvec x' \pvec w'})} \to \fTrans{(\uInter{ \bTrans{B}}{\pvec f \pvec x'}{\pvec w'})}) \\[1mm]
	& \stackrel{\textup{D}\ref{def:combined-inter}}{\equiv} &  \ubq{\pvec \tau^+_{\bTrans{A}}, \pvec \tau^-_{\bTrans{B}}}{\pvec x', \pvec w'}{\pvec x, \pvec w} (\bInter{A}{\pvec x'}{\pvec g \pvec x' \pvec w'} \to \bInter{B }{\pvec f \pvec x'}{\pvec w'}) 
\end{array}
}
Conjunction.
\eqleft{
\begin{array}{lcl}
	\lInter{A \wedge B}{\pvec x, \pvec v}{\pvec y,\pvec w} 
	& \stackrel{\textup{D}\ref{def:combined-inter}}{\equiv}& \fTrans{(\uInter{\bTrans{(A \wedge B)}}{\pvec x,\pvec v}{\pvec y,\pvec w})} \\[1mm]
	& \stackrel{\textup{D}\ref{g-trans}}{\equiv} & \fTrans{(\uInter{\bTrans{A}\otimes \bTrans{B})}{\pvec x,\pvec v}{\pvec y,\pvec w})}   \\[1mm]
	& \stackrel{\textup{D}\ref{inter}}{\equiv} &  \fTrans{(\uInter{ \bTrans{A}}{\pvec x}{\pvec y}\otimes \uInter{ \bTrans{B}}{  \pvec v}{\pvec w})}
	\stackrel{\textup{D}\ref{forget}}{\equiv}  \fTrans{(\uInter{\bTrans{A}}{\pvec x}{\pvec y})} \wedge \fTrans{(\uInter{ \bTrans{B}}{  \pvec v}{\pvec w})} 
	\stackrel{\textup{D}\ref{def:combined-inter}}{\equiv} \bInter{A}{\pvec x}{\pvec y} \wedge \bInter{B}{\pvec v}{\pvec w} 
\end{array}
}
Disjunction.
\eqleft{
\begin{array}{lcl}
	\bInter{A \vee B}{\pvec x, \pvec v, b}{\pvec y, \pvec w} 
	& \stackrel{\textup{D}\ref{def:combined-inter}}{\equiv}& \fTrans{(\uInter{\bTrans{(A \vee B)}}{\pvec x,\pvec v, b}{\pvec y,\pvec w})}   \\[1mm]
	& \stackrel{\textup{D}\ref{g-trans}}{\equiv} & \fTrans{(\uInter{\exists z^\BB (\lpcond{z}{\bTrans{A}}{\bTrans{B}})}{ \pvec x,\pvec v, b}{\pvec y,\pvec w})}   \\[1mm]
	& \stackrel{\textup{D}\ref{inter}}{\equiv} & \fTrans{(\exists \bbound{z}{b}{\BB} (\lpcond{z}{\uInter{\bTrans{A}}{\pvec x}{\pvec y}}{ \uInter{\bTrans{B}}{\pvec v}{\pvec w}}))}	\\[1mm]
	& \stackrel{\textup{D}\ref{forget}}{\equiv} &  \exists \bound{z}{b}{\BB} (\pcond{z}{\fTrans{(\uInter{\bTrans{A}}{\pvec x}{\pvec y})}}{\fTrans{(\uInter{\bTrans{B}}{\pvec v}{\pvec w})}})  \\[1mm]
	& \stackrel{\textup{D}\ref{def:combined-inter}}{\equiv} &\exists \bound{z}{b}{\BB} (\pcond{z}{\bInter{A}{\pvec x}{\pvec y}}{\bInter{B}{\pvec v}{\pvec w}}) 
\end{array}
}
Existential quantifier.
\[
	\bInter{\exists z A}{\pvec x}{\pvec y} 
	\stackrel{\textup{D}\ref{def:combined-inter}}{\equiv} \fTrans{(\uInter{\bTrans{(\exists z A)}}{\pvec x}{\pvec y})}
	\stackrel{\textup{D}\ref{g-trans}}{\equiv} \fTrans{(\uInter{\exists z \bTrans{A}}{\pvec x}{\pvec y})}
	\stackrel{\textup{D}\ref{inter}}{\equiv} \fTrans{(\exists z \uInter{\bTrans{A}}{\pvec x}{\pvec y})} 
	\stackrel{\textup{D}\ref{forget}}{\equiv} \exists z \fTrans{(\uInter{\bTrans{A}}{\pvec x}{\pvec y})} 
	\stackrel{\textup{D}\ref{def:combined-inter}}{\equiv} \exists z \bInter{A}{\pvec x}{\pvec y} 
\]
Universal quantifier.
\eqleft{
\begin{array}{lcl}
	\bInter{\forall z A}{\pvec x}{\pvec y} 
	& \stackrel{\textup{D}\ref{def:combined-inter}}{\equiv}& \fTrans{(\uInter{\bTrans{(\forall z A)}}{\pvec x}{\pvec y})}   \\[1mm]
	& \stackrel{\textup{D}\ref{g-trans}}{\equiv} & \fTrans{(\uInter{\bang \forall z\bTrans{A}}{\pvec x}{\pvec y})}    \\[1mm]
	& \stackrel{\textup{D}\ref{inter}}{\equiv} & \fTrans{(\bang \bubq{\pvec \tau^-_{\bTrans{A}}}{\pvec y'}{\pvec y} \forall z \uInter{\bTrans{A}}{\pvec x}{\pvec y'})}	 \\[1mm]
	& \stackrel{\textup{D}\ref{forget}}{\equiv} & \ubq{\pvec \tau^-_{\bTrans{A}}}{\pvec y'}{\pvec y} \forall z \fTrans{(\uInter{\bTrans{A}}{\pvec x}{\pvec y'})}  \\[1mm]
	& \stackrel{\textup{D}\ref{def:combined-inter}}{\equiv} & \ubq{\pvec \tau^-_{\bTrans{A}}}{\pvec y'}{\pvec y} \forall z \bInter{A}{\pvec x}{\pvec y'}
\end{array}
}
Relativised (computational) existential quantifier.
\eqleft{
\begin{array}{lcl}
\bInter{\exists z^P A}{\pvec x, c}{\pvec y}
	& \stackrel{\textup{N}\ref{notation:qual-IL}}{\equiv}& \bInter{\exists z (P(z) \wedge A)}{\pvec x, c}{\pvec y}   \\[1mm]
	& \stackrel{\textup{D}\ref{def:combined-inter}}{\equiv}& \fTrans{(\uInter{\bTrans{(\exists z (P(z) \wedge A))}}{\pvec x, c}{\pvec y})}   \\[1mm]
	& \stackrel{\textup{D}\ref{g-trans}}{\equiv} & \fTrans{(\uInter{\exists z (\bang P(z) \otimes \bTrans{A})}{\pvec x, c}{\pvec y})}  \\[1mm]
	& \stackrel{\textup{D}\ref{inter}}{\equiv} & \fTrans{(\exists z ((\bbound{z}{c}{P}) \otimes \uInter{\bTrans{A}}{\pvec x}{\pvec y}))}  \\[1mm]
	& \stackrel{\textup{D}\ref{forget}}{\equiv} & \exists \bound{z}{c}{P} \, \fTrans{(\uInter{\bTrans{A}}{\pvec x}{\pvec y})}  \\[1mm]
	& \stackrel{\textup{D}\ref{def:combined-inter}}{\equiv} & \exists \bound{z}{c}{P} \, \bInter{A}{\pvec x}{\pvec y}
\end{array}
}
Relativised (computational) universal quantifier. 
\[
\begin{array}{lcl}
	\bInter{\forall z^P A}{\pvec f}{c, \pvec y} 
	& \stackrel{\textup{N}\ref{notation:qual-IL}}{\equiv}& \bInter{\forall z (P(z) \to A)}{\pvec f}{c, \pvec y}  \\[1mm]
	& \stackrel{\textup{D}\ref{def:combined-inter}}{\equiv}& \fTrans{(\uInter{\bTrans{(\forall z (P(z) \to A))}}{\pvec f}{c, \pvec y})}   \\[1mm]
	& \stackrel{\textup{D}\ref{g-trans}}{\equiv} & \fTrans{(\uInter{\bang \forall z \bang (\bang P(z) \lto \bTrans{A})}{\pvec f}{c, \pvec y})}    \\[1mm]
	& \stackrel{\textup{D}\ref{inter}}{\equiv} & \fTrans{(\bang \bubq{\wtype{P},\btype{\pvec \tau}^-_A}{c', \pvec y'}{c, \pvec y} \forall z \bang \bubq{\pvec \tau^-_{\bTrans{A}}}{c'', \pvec y''}{c', \pvec y'} (\bang (\bbound{z}{c}{P}) \lto \uInter{\bTrans{A}}{\pvec f c''}{\pvec y''}))}	 \\[1mm]
	& \stackrel{\textup{\Quant{4}}}{\equiv} & \fTrans{(\bang \bubq{\wtype{P},\btype{\pvec \tau}^-_A}{c', \pvec y'}{c, \pvec y} \bang \bubq{\pvec \tau^-_{\bTrans{A}}}{c'', \pvec y''}{c', \pvec y'} \forall z (\bang (\bbound{z}{c''}{P}) \lto \uInter{\bTrans{A}}{\pvec f c''}{\pvec y''}))}	  \\[1mm]
	& \stackrel{\textup{D}\ref{forget}}{\equiv} & \ubq{\wtype{P}, \btype{\pvec \tau}^-_{\bTrans{A}}}{c', \pvec y'}{c, \pvec y} \ubq{{\pvec \tau}^-_{\bTrans{A}}}{c'', \pvec y''}{c', \pvec y'} \forall \bound{z}{c''}{P} \, \fTrans{( \uInter{\bTrans{A}}{\pvec f c''}{\pvec y''} )}    \\[1mm]
	& \stackrel{\textup{D}\ref{def:combined-inter}}{\equiv} & \ubq{\wtype{P}, \btype{\pvec \tau}^-_{\bTrans{A}}}{c', \pvec y'}{c, \pvec y} \ubq{\pvec \tau^-_{\bTrans{A}}}{c'', \pvec y''}{c', \pvec y'} \forall \bound{z}{c''}{P} \, \bInter{A}{\pvec f c''}{\pvec y''}   
\end{array}
\]

\section*{Appendix C: Proof of Theorem \ref{compare-interpretations}}

\noindent We will first prove the following lemma:

\begin{lemma} \label{comp-mon} For each formula $A$ of $\IL^\omega$ there exists a tuple of closed terms $\pvec a$ such that 
\[
\proves_{\IL^\omega} \Wit_{\pvec \tau^+_{\bTrans{A}} \to \btype{\pvec \tau^-_{\bTrans{A}}} \to \pvec \tau^-_{\bTrans{A}}}(\pvec a) \quad \mbox{and} \quad 
\Wit_{\pvec \tau^+_{\bTrans{A}}, \btype{\pvec \tau^-_{\bTrans{A}}}}(\pvec x, \pvec y), \bInter{A}{\pvec x}{\pvec a \pvec x \pvec y} \proves_{\IL^\omega} \ubq{\pvec \tau^-_{\bTrans{A}}}{\pvec y'}{\pvec y} \bInter{A}{\pvec x}{\pvec y'}
\]
\end{lemma}
\begin{proof} By induction on the complexity of the formula $A$. The only non-trivial case is when $A$ is an existential formula $\exists z A$: Assume $\pvec a$ is the witness for $A$, and $\Wit_{\pvec \tau^+_{\bTrans{(\exists z A)}}, \btype{\pvec \tau^-_{\bTrans{(\exists z A)}}}}(\pvec x, \pvec y)$, then, by definition, we have $\Wit_{\pvec \tau^+_{\bTrans{A}}, \btype{\pvec \tau^-_{\bTrans{A}}}}(\pvec x, \pvec y)$ and hence
\[
\begin{array}{lcl}
	\bInter{\exists z A}{\pvec x}{\pvec a \pvec x \pvec y} 
		& \stackrel{\textup{P}\ref{prop-bInter}}{\Leftrightarrow} & \exists z \bInter{A}{\pvec x}{\pvec a \pvec x \pvec y} \\[1mm]
		& \stackrel{(\textup{IH})}{\Rightarrow} & \exists z \ubq{\pvec \tau^-_{\bTrans{A}}}{\pvec y'}{\pvec y} \bInter{A}{\pvec x}{\pvec y'} \\[1mm]
		& \stackrel{\textup{\Quant{5}}}{\Rightarrow} & \ubq{\pvec \tau^-_{\bTrans{A}}}{\pvec y'}{\pvec y} \exists z \bInter{A}{\pvec x}{\pvec y'}  \\[1mm]
		& \stackrel{\textup{P}\ref{prop-bInter}}{\Leftrightarrow} & \ubq{\pvec \tau^-_{\bTrans{(\exists z A)}}}{\pvec y'}{\pvec y} \bInter{\exists z A}{\pvec x}{\pvec y'}  
\end{array}
\]
All other cases follow directly from our assumption \ConApp. For instance, writing ${\rm id}$ for the identity function, for implication $A \to B$ we have:
\[
\begin{array}{lcl}
		& & \bInter{A \to B}{\pvec f, \pvec g}{\comp{{\rm id}}{(\pvec x, \pvec w)}} \\[1mm]
		& \stackrel{\textup{P}\ref{prop-bInter}}{\Leftrightarrow} & \ubq{\pvec \tau^+_{\bTrans{A}}, \pvec \tau^-_{\bTrans{B}}}{\pvec x'', \pvec w''}{\comp{{\rm id}}{(\pvec x, \pvec w)}} (\bInter{A}{\pvec x''}{\pvec g \pvec x'' \pvec w''} \to \bInter{B}{\pvec f \pvec x''}{\pvec w''}) \\[1mm]
		& \stackrel{\textup{\ConApp}}{\Rightarrow} & \ubq{\btype{\pvec \tau^+_{\bTrans{A}}, \pvec \tau^-_{\bTrans{B}}}}{\pvec x', \pvec w'}{\pvec x, \pvec w} \ubq{\pvec \tau^+_{\bTrans{A}}, \pvec \tau^-_{\bTrans{B}}}{\pvec x'', \pvec w''}{\pvec x', \pvec w'} (\bInter{A}{\pvec x''}{\pvec g \pvec x'' \pvec w''} \to \bInter{B}{\pvec f \pvec x''}{\pvec w''})  \\[1mm]
		& \stackrel{\textup{P}\ref{prop-bInter}}{\Leftrightarrow} & \ubq{\pvec \tau^-_{\bTrans{(A \to B)}}}{\pvec x', \pvec w'}{\pvec x, \pvec w} \bInter{A \to B}{\pvec f, \pvec g}{\pvec x', \pvec w'}  
\end{array}
\]
since $\tau^-_{\bTrans{(A \to B)}} = \btype{\tau^+_{\bTrans{A}}, \tau^-_{\bTrans{B}}}$.
\end{proof}

We then prove points $(i)$ and $(ii)$ simultaneously by induction on the complexity of the formula $A$, using Propositions~\ref{prop-lInter} and \ref{prop-bInter}. The only non-trivial cases are the cases of implication and the quantifiers, so we will focus on these cases. During the proof we will make use of our assumptions \ConApp~ and \ConUnit, which only hold when the terms in question are in $\Wit$. This will be the case, however, since by induction hypothesis, the terms we are working with are already in $\Wit$, and the free-variables are also assumed to be in $\Wit$. The constructed terms will then be guaranteed to be in $\Wit$ (point $(iii)$) by Lemma \ref{lem-W-closure}. \\[2mm]
Universal quantifier $(i)$. Let $\pvec s_1, \pvec t_1$ be given by induction hypothesis, and assume $\Wit_{\pvec \tau^+_{\lTrans{(\forall z A)}}, \pvec \tau^-_{\bTrans{(\forall z A)}}}(\pvec x, \pvec y)$, i.e. $\Wit_{\pvec \tau^+_{\lTrans{A}}, \btype{\pvec \tau^-_{\bTrans{A}}}}(\pvec x, \pvec y)$. Then:
\eqleft{
\begin{array}{lcl}
	\ubq{\pvec \tau^-_{\lTrans{(\forall z A)}}}{\pvec y'}{\comp{(\pvec s_1 \pvec x)}{\pvec y}} \lInter{\forall z A}{\pvec x}{\pvec y'} 
		& \stackrel{\textup{P}\ref{prop-lInter}}{\Leftrightarrow} & \ubq{\pvec \tau^-_{\lTrans{A}}}{\pvec y'}{\comp{(\pvec s_1 \pvec x)}{\pvec y}} \forall z \lInter{A}{\pvec x}{\pvec y'} \\[1mm]
		& \stackrel{\textup{\ConApp}}{\Rightarrow} & \ubq{\pvec \tau^-_{\bTrans{A}}}{\pvec y'}{\pvec y} \ubq{\pvec \tau^-_{\lTrans{A}}}{\pvec y''}{\pvec s_1 \pvec x \pvec y'} \forall z \lInter{A}{\pvec x}{\pvec y''} \\[1mm]
		& \stackrel{\textup{\Quant{4}}}{\Rightarrow} &  \ubq{\pvec \tau^-_{\bTrans{A}}}{\pvec y'}{\pvec y} \forall z \ubq{\pvec \tau^-_{\lTrans{A}}}{\pvec y''}{\pvec s_1 \pvec x \pvec y'} \lInter{A}{\pvec x}{\pvec y''}  \\[1mm]
		& \stackrel{(\textup{IH}_{(i)})}{\Rightarrow} & \ubq{\pvec \tau^-_{\bTrans{A}}}{\pvec y'}{\pvec y} \forall z \bInter{A}{\pvec t_1 \pvec x}{\pvec y'}
		\stackrel{\textup{P}\ref{prop-bInter}}{\Leftrightarrow} \bInter{\forall z A}{\pvec t_1 \pvec x}{\pvec y}
\end{array}
}
Universal quantifier $(ii)$. Let $\pvec s_2, \pvec t_2$ be given by induction hypothesis, and assume $\Wit_{\pvec \tau^+_{\bTrans{(\forall z A)}}, \pvec \tau^-_{\lTrans{(\forall z A)}}}(\pvec x, \pvec y)$. Then:
\eqleft{
\begin{array}{lcl}
	\bInter{\forall z A}{\pvec x}{\singleton{(\pvec s_2 \pvec x \pvec y)}} 
	& \stackrel{\textup{P}\ref{prop-bInter}}{\equiv} & \ubq{\pvec \tau^-_{\bTrans{A}}}{\pvec y'}{\singleton{(\pvec s_2 \pvec x \pvec y)}} \forall z \bInter{A}{\pvec x}{\pvec y'}  \\[1mm]
	& \stackrel{\textup{\ConUnit}}{\Rightarrow} & \forall z \bInter{A}{\pvec x}{\pvec s_2 \pvec x \pvec y}\\[1mm]
	& \stackrel{(\textup{IH}_{(ii)})}{\Rightarrow} & \forall z \ubq{\pvec \tau^-_{\lTrans{A}}}{\pvec y'}{\pvec y} \lInter{A}{\pvec t_2 \pvec x}{\pvec y'}  \\[1mm]
	& \stackrel{\textup{\Quant{4}}}{\Rightarrow} & \ubq{\pvec \tau^-_{\lTrans{A}}}{\pvec y'}{\pvec y} \forall z \lInter{A}{\pvec t_2 \pvec x}{\pvec y'} 
	\stackrel{\textup{P}\ref{prop-lInter}}{\Rightarrow} \ubq{\pvec \tau^-_{\lTrans{(\forall z A)}}}{\pvec y'}{\pvec y} \lInter{\forall z A}{\pvec t_2 \pvec x}{\pvec y}
\end{array}
}
Existential quantifier $(i)$. Let $\pvec s_1, \pvec t_1$ be given by induction hypothesis and assume $\Wit_{\pvec \tau^+_{\lTrans{(\exists z A)}}, \pvec \tau^-_{\bTrans{(\exists z A)}}}(\pvec x, \pvec y)$. Then:
\eqleft{
\begin{array}{lcl}
	\ubq{\pvec \tau^-_{\lTrans{(\exists z A)}}}{\pvec y'}{\eta(\pvec s_1 \pvec x \pvec y)} \lInter{\exists z A}{\pvec x}{\pvec y'} 
	& \stackrel{\textup{P}\ref{prop-lInter}}{\equiv} & \ubq{\pvec \tau^-_{\lTrans{A}}}{\pvec y'}{\eta(\pvec s_1 \pvec x \pvec y)} \exists z \ubq{\pvec \tau^-_{\lTrans{A}}}{\pvec y''}{\pvec y'}\lInter{A}{\pvec x}{\pvec y''}  \\[1mm]
	& \stackrel{\textup{\ConUnit}}{\Rightarrow} &  \exists z \ubq{\pvec \tau^-_{\lTrans{A}}}{\pvec y'}{\pvec s_1 \pvec x \pvec y}\lInter{A}{\pvec x}{\pvec y'} \\[1mm]
	& \stackrel{(\textup{IH}_{(i)})}{\Rightarrow} & \exists z \bInter{A}{\pvec t_1 \pvec x}{\pvec y}
	\stackrel{\textup{P}\ref{prop-bInter}}{\equiv} \bInter{\exists z A}{\pvec t_1 \pvec x}{\pvec y}
\end{array}
}
Existential quantifier $(ii)$. Let $\pvec a$ be as in Lemma \ref{comp-mon}, and $\pvec s_2, \pvec t_2$ be given by induction hypothesis and assume $\Wit_{\pvec \tau^+_{\bTrans{(\exists z A)}}, \pvec \tau^-_{\lTrans{(\exists z A)}}}(\pvec x, \pvec y)$, then:
\eqleft{
\begin{array}{lcl}
	\bInter{\exists z A}{\pvec x}{\pvec a \pvec x (\comp{(\lambda \pvec y . \singleton{(\pvec s_2 \pvec x \pvec y)})}{\pvec y})} 
		& \stackrel{\textup{L}~\ref{comp-mon}}{\Rightarrow} & \ubq{\pvec \tau^-_{\exists z A}}{\pvec y'}{\comp{(\lambda \pvec y . \singleton{(\pvec s_2 \pvec x \pvec y)})}{\pvec y}} \bInter{\exists z A}{\pvec x}{\pvec y'}  \\[1mm]
		& \stackrel{\textup{\ConApp}}{\Rightarrow} & \ubq{\pvec \tau^-_{\lTrans{ A}}}{\pvec y'}{\pvec y} \ubq{\pvec \tau^-_{\bTrans{(\exists z A)}}}{\pvec y''}{\singleton{(\pvec s_2 \pvec x \pvec y')}} \bInter{\exists z A}{\pvec x}{\pvec y''}  \\[1mm]
		& \stackrel{\textup{P}\ref{prop-bInter}}{\equiv} & \ubq{\pvec \tau^-_{\lTrans{ A}}}{\pvec y'}{\pvec y} \ubq{\pvec \tau^-_{\bTrans{ A}}}{\pvec y''}{\singleton{(\pvec s_2 \pvec x \pvec y')}} \exists z \bInter{A}{\pvec x}{\pvec y''}  \\[1mm]
		& \stackrel{\textup{\ConUnit}}{\Rightarrow} & \ubq{\pvec \tau^-_{\lTrans{ A}}}{\pvec y'}{\pvec y} \exists z \bInter{A}{\pvec x}{\pvec s_2 \pvec x \pvec y'}  \\[1mm]
		& \stackrel{(\textup{IH}_{(ii)})}{\Rightarrow} & \ubq{\pvec \tau^-_{\lTrans{ A}}}{\pvec y'}{\pvec y} \exists z \ubq{\pvec \tau^-_{\lTrans{ A}}}{\pvec y''}{\pvec y'} \lInter{A}{\pvec t_2 \pvec x}{\pvec y''} \\[1mm]
		& \stackrel{\textup{P}\ref{prop-lInter}}{\equiv} & \ubq{\pvec \tau^-_{\lTrans{(\exists z A)}}}{\pvec y'}{\pvec y} \lInter{\exists z A}{\pvec t_2 \pvec x}{\pvec y'} 
\end{array}
}
Implication $(i)$. Let $\pvec s_1^B, \pvec t_1^B, \pvec s_2^A, \pvec t_2^A$ be given by induction hypothesis. Assume that $\Wit_{\pvec \tau^+_{\lTrans{(A\to B)}}}(\pvec f, \pvec g)$ and $\Wit_{\pvec \tau^-_{\bTrans{(A\to B)}}}(\pvec x, \pvec w)$. Then
\[
\footnotesize{\begin{array}{lcl}
	& & 
	
	\ubq{\pvec \tau^+_{\lTrans{A}}, \pvec \tau^-_{\lTrans{B}}}{\pvec x', \pvec w'}
	    { \comp{(\lambda \pvec u, \pvec v . \singleton{(\pvec t_2^A \pvec u)}, \pvec s_1^B ( \pvec f (\pvec t_2 \pvec u) ) \pvec v)}{(\pvec x, \pvec w)}  } 
			\lInter{A \to B}{\pvec f, \pvec g}{\pvec x',\pvec w'} \\[1mm]

	& \stackrel{\mbox{\footnotesize{\ConApp}}}{\Rightarrow} & 

	\ubq{\pvec \tau^+_{\bTrans{A}}, \pvec \tau^-_{\bTrans{B}}}{\pvec x'', \pvec w''}{\pvec x, \pvec w} 
	\ubq{\pvec \tau^+_{\lTrans{A}}, \pvec \tau^-_{\lTrans{B}}}{\pvec x', \pvec w'}
	    { \singleton{(\pvec t_2^A \pvec x'')}, \pvec s_1^B ( \pvec f (\pvec t_2 \pvec x'') ) \pvec w''  } 
			\lInter{A \to B}{\pvec f, \pvec g}{\pvec x',\pvec w'} \\[1mm]
		
	& \stackrel{\mbox{\footnotesize{\ConUnit}}}{\Rightarrow} & 

	\ubq{\pvec \tau^+_{\bTrans{A}}, \pvec \tau^-_{\bTrans{B}}}{\pvec x'', \pvec w''}{\pvec x, \pvec w} 
	\ubq{\pvec \tau^-_{\lTrans{B}}}{\pvec w'}{ \pvec s_1^B ( \pvec f (\pvec t_2 \pvec x'') ) \pvec w''  } 
			\lInter{A \to B}{\pvec f, \pvec g}{\pvec t_2^A \pvec x'',\pvec w'} \\[1mm]

	& \stackrel{\textup{P}\ref{prop-lInter}}{\Leftrightarrow} & 
	
	\ubq{\pvec \tau^+_{\bTrans{A}}, \pvec \tau^-_{\bTrans{B}}}{\pvec x'', \pvec w''}{\pvec x, \pvec w} 
	\ubq{\pvec \tau^-_{\lTrans{B}}}{\pvec w'}{ \pvec s_1^B ( \pvec f (\pvec t_2 \pvec x'') ) \pvec w''  } \\[1mm]
	& & \hspace{10mm}
	(\ubq{\pvec \tau^-_{\lTrans{A}}}{\pvec y}{\pvec g (\pvec t_2^A \pvec x'') \pvec w'} \lInter{A}{\pvec t_2^A \pvec x''}{\pvec y} 
				\to \lInter{B}{\pvec f (\pvec t_2^A \pvec x'')}{\pvec w'}) \\[1mm]

	& \stackrel{\mbox{\footnotesize{\Quant{1}}}}{\Rightarrow} & 

	\ubq{\pvec \tau^+_{\bTrans{A}}, \pvec \tau^-_{\bTrans{B}}}{\pvec x'', \pvec w''}{\pvec x, \pvec w} \\[1mm]
	& & \hspace{10mm}
	(\ubq{\pvec \tau^-_{\lTrans{B}}}{\pvec w'}{ \pvec s_1^B ( \pvec f (\pvec t_2 \pvec x'') ) \pvec w''  } 
	 \ubq{\pvec \tau^-_{\lTrans{A}}}{\pvec y}{\pvec g (\pvec t_2^A \pvec x'') \pvec w'} \lInter{A}{\pvec t_2^A \pvec x''}{\pvec y} \\[1mm] 
	 & & \hspace{20mm}
				\to \ubq{\pvec \tau^-_{\lTrans{B}}}{\pvec w'}{ \pvec s_1^B ( \pvec f (\pvec t_2 \pvec x'') ) \pvec w''  }
				     \lInter{B}{\pvec f (\pvec t_2^A \pvec x'')}{\pvec w'}) \\[1mm]

	& \stackrel{\mbox{\footnotesize{\ConApp}}}{\Rightarrow} & 
	
	\ubq{\pvec \tau^+_{\bTrans{A}}, \pvec \tau^-_{\bTrans{B}}}{\pvec x'', \pvec w''}{\pvec x, \pvec w} \\[1mm]
	& & \hspace{10mm}
	(\ubq{\pvec \tau^-_{\lTrans{A}}}{\pvec y}{ \comp{ \pvec g (\pvec t_2^A \pvec x'') }{ (\pvec s_1^B ( \pvec f (\pvec t_2 \pvec x'') ) \pvec w'') } } 
	 	\lInter{A}{\pvec t_2^A \pvec x''}{\pvec y} \\[1mm] 
	 & & \hspace{20mm}
				\to \ubq{\pvec \tau^-_{\lTrans{B}}}{\pvec w'}{ \pvec s_1^B ( \pvec f (\pvec t_2 \pvec x'') ) \pvec w''  }
				     \lInter{B}{\pvec f (\pvec t_2^A \pvec x'')}{\pvec w'}) \\[1mm]

	& \stackrel{(\mathrm{IH}_{(i)}, \mathrm{IH}_{(ii)})}{\Rightarrow} & 
	
	\ubq{\pvec \tau^+_{\bTrans{A}}, \pvec \tau^-_{\bTrans{B}}}{\pvec x'', \pvec w''}{\pvec x, \pvec w} 
		(\bInter{A}{\pvec x''}{\pvec s^A_2 \pvec x'' (\comp{(\pvec g (\pvec t_2^A \pvec x''))}
		{(\pvec s_1^B (\pvec f (\pvec t_2^A \pvec x'')) \pvec w'')})} 
				\to \bInter{B }{\pvec t_1^B (\pvec f (\pvec t_2^A \pvec x''))}{\pvec w''}) \\[3mm]
				
	& \stackrel{\textup{P}\ref{prop-bInter}}{\Leftrightarrow} & 
	
	\bInter{A \to B}{\lambda \pvec x', \pvec w' . \pvec s^A_2 \pvec x' (\comp{(\pvec g (\pvec t_2^A \pvec x'))}{(\pvec s_1^B (\pvec f (\pvec t_2^A \pvec x')) \pvec w')}), \lambda \pvec x' . \pvec t_1^B (\pvec f (\pvec t_2^A \pvec x'))}{\pvec x,\pvec w} 
\end{array}}
\]
Implication $(ii)$. Let $\pvec s_1^A, \pvec t_1^A, \pvec s_2^B, \pvec t_2^B$ be given by induction hypothesis, and let $\pvec r[\pvec x, \pvec w] = \pvec s_2^B (\pvec f (\pvec t^A_1 \pvec x)) (\singleton{(\pvec w)})$. Assume $\Wit_{\pvec \tau^+_{\bTrans{(A\to B)}}, \pvec \tau^-_{\lTrans{(A\to B)}}}(\pvec x, \pvec y)$. Then
\[
\footnotesize{
\begin{array}{lcl}
	& & \bInter{A \to B}{\pvec f,\pvec g}{\lambda \pvec (x', w'). \comp{\singleton{(\pvec t^A_1 \pvec x')}, \singleton{(\pvec r[\pvec x', \pvec w'])}}{(\pvec x', \pvec w')}} \\[1mm]
	&\stackrel{\textup{P}\ref{prop-bInter}}{\equiv} & 
		\ubq{\pvec \tau^+_{\bTrans{A}}, \pvec \tau^-_{\bTrans{B}}}{\pvec x'', \pvec w''}{\lambda (\pvec x', \pvec w'). \comp{\singleton{(\pvec t^A_1 \pvec x')}, \singleton{(\pvec r[\pvec x', \pvec w'])}}{(\pvec x', \pvec w')}}  (\bInter{A}{\pvec x''}{\pvec g \pvec x'' \pvec w''} \to \bInter{B}{\pvec f \pvec x''}{\pvec w'} )\\[1mm]
	& \stackrel{\mbox{\footnotesize{\ConApp}}}{\Rightarrow} & 
		\ubq{\pvec \tau^+_{\lTrans{A}}, \pvec \tau^-_{\lTrans{B}}}{\pvec x'',\pvec w''}{\pvec x',\pvec w'} \ubq{\pvec \tau^+_{\bTrans{A}}, \pvec \tau^-_{\bTrans{B}}}{\pvec x, \pvec w}{\singleton{(\pvec t^A_1 \pvec x'')},\singleton{(\pvec r[\pvec x'', \pvec w''])}} (\bInter{A}{\pvec x}{\pvec g \pvec x \pvec w} \to \bInter{B}{\pvec f \pvec x}{\pvec w} )\\[1mm]
	& \stackrel{\mbox{\footnotesize{\ConUnit}}}{\Rightarrow} & 
		\ubq{\pvec \tau^+_{\lTrans{A}}, \pvec \tau^-_{\lTrans{B}}}{\pvec x'',\pvec w''}{\pvec x',\pvec w'} (\bInter{A}{\pvec t^A_1 \pvec x''}{\pvec g (\pvec t^A_1 \pvec x'') (\pvec r[\pvec x'', \pvec w''])} \to \bInter{B}{\pvec f (\pvec t^A_1 \pvec x'')}{\pvec r[\pvec x'', \pvec w'']} )\\[1mm]
		& \stackrel{(\mathrm{IH}_{(i)}, \mathrm{IH}_{(ii)})}{\Rightarrow} & 
		\ubq{\pvec \tau^+_{\lTrans{A}}, \pvec \tau^-_{\lTrans{B}}}{\pvec x'', \pvec w''}{\pvec x', \pvec w'} (\ubq{\pvec \tau^-_{\lTrans{A}}}{\pvec y'}{\pvec s^A_1 \pvec x'' (\pvec g (\pvec t^A_1 \pvec x'') (\pvec r[\pvec x'', \pvec w''])} \lInter{A}{\pvec x''}{\pvec y'} \\[1mm]
		&  & \hspace{30mm} \to \ubq{\pvec \tau^-_{\lTrans{B}}}{\pvec w}{\singleton{(\pvec w'')}} \lInter{B }{\pvec t_2^B (\pvec f (\pvec t_1^A \pvec x''))}{\pvec w}) \\[1mm]
	& \stackrel{\mbox{\footnotesize{\ConUnit}}}{\Rightarrow} & 
		\ubq{\pvec \tau^+_{\lTrans{A}}, \pvec \tau^-_{\lTrans{B}}}{\pvec x'', \pvec w''}{\pvec x', \pvec w'} (\ubq{\pvec \tau^-_{\lTrans{A}}}{\pvec y'}{\pvec s^A_1 \pvec x'' (\pvec g (\pvec t^A_1 \pvec x'') (\pvec r[\pvec x'', \pvec w'']))} \lInter{A}{\pvec x''}{\pvec y'} \\[1mm]
			& &	\hspace{30mm} \to \lInter{B }{\pvec t_2^B (\pvec f (\pvec t_1^A \pvec x''))}{\pvec w''}) \\[1mm]
	& \stackrel{\textup{P}\ref{prop-lInter}}{\equiv} & 
		\ubq{\pvec \tau^-_{\lTrans{(A\to B)}}, }{\pvec x'', \pvec w''}{\pvec x', \pvec w'} \lInter{A \to B}{\lambda \pvec x, \pvec w . \pvec s^A_1 \pvec x (\pvec g (\pvec t^A_1 \pvec x) (\pvec r[\pvec x, \pvec w])), \lambda \pvec x . \pvec t_2^B (\pvec f (\pvec t_1^A \pvec x))}{\pvec x'',\pvec w''} 
\end{array}}
\]

\end{document}